\documentclass[reqno,letterpaper,11pt]{article}

\usepackage{times}
\usepackage{hyperref}

\usepackage{amsmath}
\usepackage{amsfonts,amssymb,amsmath,latexsym,wasysym,mathrsfs,bbm,stmaryrd}
\usepackage[all]{xy}
\usepackage[usenames]{color}
\usepackage{epsfig}
\usepackage{amsthm}

\usepackage[mathcal]{euscript}

\usepackage{graphics,moreverb,hangcaption,gastex}

\def\R{\mathbb R}
\def\C{\mathbb C}

\def\N{\mathbb N}
\def\P{\mathbb P}
\def\E{\mathbb E}
\def\A{\mathscr{A}}

\def\H{\mathcal{H}}

\def\1{\mathbbm 1}

\def\e{\epsilon}
\def\d{\delta}

\def\tensor{\otimes}
\def\supp{\mathrm{supp}\,}
\def\del{\partial}

\def\a{\alpha}
\def\b{\beta}

\newcommand{\tr}{\mathrm{tr}\,}
\newcommand{\Tr}{\mathrm{Tr}\,}
\newcommand{\ff}{\varphi}

\newcommand{\interior}[1]{\raise0.2ex\hbox{$\displaystyle{\mathop{#1}^{\circ}}$}}
\newcommand{\mx}[1]{\mathbf{#1}}

\def\t{\tau}

\newtheorem{theorem}{Theorem}[section]

\newtheorem{proposition}[theorem]{Proposition}
\newtheorem{definition}[theorem]{Definition}
\newtheorem{corollary}[theorem]{Corollary}
\newtheorem{conjecture}[theorem]{Conjecture}

\newtheorem{lemma}[theorem]{Lemma}

\numberwithin{equation}{section} 

\theoremstyle{remark}
\newtheorem{example}[theorem]{Example}
\newtheorem{remark}[theorem]{Remark}

\long\def\symbolfootnote[#1]#2{\begingroup
\def\thefootnote{\fnsymbol{footnote}}\footnote[#1]{#2}\endgroup}

\begin{document}

\author{
  Beno\^\i{}t Collins\thanks{Supported in part by an NSERC discovery grant and Ontario's Early Research Award.} \\
  University of Ottawa \\
  \& CNRS Lyon I \& AIMR Tohoku \\
  585 King Edward, Ottawa, ON K1N 6N5 \\
  \texttt{bcollins@uottawa.ca}
  \and
  Todd Kemp\thanks{Supported in part by NSF grant DMS-1001894, NSF CAREER award DMS-1254807, and the Hellman Fellowship.} \\
  Department of Mathematics \\
  University of California, San Diego \\
  La Jolla, CA 92093-0112 \\
\texttt{tkemp@math.ucsd.edu}
}

\title{Liberation of Projections}

\date{\today}

\maketitle

\begin{abstract} We study the liberation process for projections: $(p,q)\mapsto (p_t,q)= (u_tpu_t^\ast,q)$ where $u_t$ is a free unitary Brownian motion freely independent from $\{p,q\}$.  Its action on the operator-valued angle $qp_tq$ between the projections induces a flow on the corresponding spectral measures $\mu_t$; we prove that the Cauchy transform of the measure satisfies a holomorphic PDE.  We develop a theory of subordination for the boundary values of this PDE, and use it to show that the spectral measure $\mu_t$ possesses a piecewise analytic density for any $t>0$ and any initial projections of trace $\frac12$.  We us this to prove the Unification Conjecture for free entropy and information in this trace $\frac12$ setting. \end{abstract}

\tableofcontents

\section{Introduction} \label{section introduction}

Let $V$ and $W$ be subspaces of a the finite-finite dimensional complex space $\C^d$.  From elementary linear algebra, it follows that   $\dim V\cap W \ge \max\{\dim V + \dim W-d,0\}$.  In fact, this inequality is {\em almost surely} equality:
\begin{equation} \label{eq gp 1} \dim(V\cap W) = \max\{\dim V + \dim W-d,0\} \;\; a.s. \end{equation}
The {\em almost surely} can be interpreted in a number of ways: for example, it can be taken with respect to any reasonable probability measure on the (product) Grassmannian manifold of subspaces.  We will discuss a more probabilistic interpretation shortly.

When two subspaces satisfy the equality of Equation \ref{eq gp 1}, they are said to be in general position.  For convenience, we may rewrite this relation as follows.  Let $P,Q$ be the orthogonal projections onto $V$ and $W$ respectively, and let $P\wedge Q$ denote the orthogonal projection onto $V\cap W$.  The dimension of a subspace is the trace of the projection onto it, so the subspaces are in general position if and only if $\Tr(P\wedge Q) = \max\{\Tr P + \Tr Q -d,0\}$.  Normalizing, letting $\tr = \frac1d\Tr$, two projections (i.e.\ subspaces) are in general position if and only if
\begin{equation} \label{eq gp 2} \tr(P\wedge Q) = \max\{\tr P + \tr Q -1,0\}. \end{equation}

In this language, a good way to express the genericity of general position is as follows: let $P,Q$ be as above, and let $U$ be a {\em random unitary matrix}.  If $\tilde{P} = UPU^\ast$ (i.e.\ the projection onto the rotation of the image of $P$ by $U$), then $\tilde{P}$ and $Q$ are in general position almost surely.  This statement is valid for all reasonable notions of {\em random unitary matrix}; for example, it holds true if $U$ is sampled from some measure that has a continuous, strictly positive  density with respect to the Haar measure.  (Indeed, the subset of the Grassmannian product where general position fails to hold is a subvariety of lower dimension, and so it is easy to see that any reasonable measure will assign it probability $0$.)  Since any neighborhood of the identity has positive measure, it follows that rotations arbitrarily close to the identity produce projections in general position, agreeing with our intuition.  Note: the same result applies as well even if $P,Q$ are {\em random} projections, provided that the random unitary $U$ is independent from $\{P,Q\}$.

An important flow of random unitaries is given by the unitary Brownian motion.  The group $\mathbb{U}_d$ of unitary $d\times d$ matrices is a compact Lie group, and so its left-invariant Riemannian metric (given by the Hilbert-Schmidt norm on the Lie algebra $\mathfrak{u}_d$) gives rise to a heat kernel, which generates a Markov process: Brownian motion $U_t$.  This stochastic process can be constructed from the standard Brownian motion $W_t$ on $\mathfrak{u}_d$: it satisfies the (Stratonovich) stochastic differential equation $dU_t = U_t\circ dW_t$.  The Lie algebra $\mathfrak{u}_d$ consists of skew-Hermitian matrices; in random matrix theory, it is more conventional to consider Brownian motion taking values in Hermitian matrices, so set $X_t = -iW_t$.  Making this substitution, and converting the SDE to It\^o form, gives
\begin{equation} \label{eq SDE 1} dU_t = iU_t\,dX_t - \frac12U_t\,dt. \end{equation}
(Note: this equation corresponds to the normalization $\E\,\tr(X_t^2) = t$ where $\tr$ is the normalized trace.) The distribution of $U_t$ is the heat kernel at time $t$; on any Lie group, this has a strictly positive smooth density for any $t>0$, cf.\ \cite{Varopoulos}.  Hence, the above discussion proves the following.

\begin{proposition} \label{prop gp 1} Let $P,Q$ be orthogonal projections on $\C^d$.  Let $U_t$ be the Brownian motion on the unitary group $\mathbb{U}_d$, and set $P_t = U_tPU_t^\ast$.  Then for each $t>0$, $P_t$ and $Q$ are almost surely in general position. \end{proposition}

\noindent Beyond the question of general position (relating only to dimension), it is interesting to consider the relative position of the two subspaces, and how it evolves under the Brownian conjugation.  The principle angles between $P_t$ and $Q$ are encoded in the {\em operator-valued angle} $QP_tQ$ (whose eigenvalues are trigonometric polynomials in said angles).

\bigskip

The purpose of the present paper is to address the nature of the flow of the operator-valued angle in an {\em infinite-dimensional} context.  Let $\A$ be a $\mathrm{II}_1$-factor -- a von Neumann subalgebra of the bounded operators on some Hilbert space $\H$, with a unique tracial state $\t$.  (For intuition, one may think of a limit as $d\to\infty$ of the algebra of $d\times d$ {\em random matrices}. We will make the notion of such a limit precise in Section \ref{sect free probability}; the limit can always be identified as living in a free group factor.)  Let $p,q\in\A$ be projections, and let $p\wedge q$ be the projection onto $p\H\cap q\H$.  Then $p\wedge q$ is in $\A$, since $\A$ is weakly closed.  As in the finite dimensional setting, say that $p,q$ are in general position if
\begin{equation} \label{eq gp 3} \t(p\wedge q) = \max\{\t(p)+\t(q)-1,0\}. \end{equation}
Instead of attempting to make sense of a random operator drawn from the unfathomably large group of unitaries on $\H$, we use the tools of {\em free probability} (Section \ref{sect free probability}) to proceed.  We may assume that $\A$ is rich enough to possess a {\em free unitary Brownian motion} $u_t$ (Section \ref{sect fubm}) free from $\{p,q\}$, by enlarging $\A$ if necessary.  The operator-valued process $u_t$ can be thought of as a limit of the Brownian motions $U_t$ in $\mathbb{U}_d$ as $d\to\infty$ (cf.\ \cite{ref Biane fubm}).  In fact, the analogue of Proposition \ref{prop gp 1} holds true in this context.

\begin{proposition} \label{prop gp 2} Let $p,q$ be projections in a $\mathrm{II}_1$-factor $(\mathscr{A},\tau)$, and let $u_t$ be a free unitary Brownian motion in $\mathscr{A}$, freely independent from $\{p,q\}$.  Set $p_t=u_tpu_t^\ast$.  Then for each $t>0$, $p_t$ and $q$ are in general position.
\end{proposition}

\begin{proof} In \cite[Lemma 12.5]{Voiculescu-free-analogs-VI}, it was proved that, for any two projections $\tilde{p},q$, if the algebras $A=\C\langle \tilde{p}\rangle$ and $B=\C\langle q\rangle$ possess an $L^1(\tau)$ liberation gradient $j(A:B)$, then $p,q$ are in general position.  Letting $\tilde{p}=upu^\ast$ for some unitary free from $\{p,q\}$, it was proved in \cite[Proposition 8.7]{Voiculescu-free-analogs-VI} that this liberation gradient exists (in fact in $L^2(\tau)\subset L^1(\tau)$) provided that the law of $u$ possesses an  $L^3$-density with respect to the Haar measure on the circle.  In fact, the density of $u_t$ is continuous and bounded, with a density that is real analytic on the set where it is positive, cf.\ \cite{ref Biane fubm} (refined in \cite[Corollary 1.7]{Voiculescu-free-analogs-VI}).  This proves the proposition.  \end{proof}

\begin{remark} This argument was not known initially to the authors of the present paper; it was also unknown to the authors of \cite{Bercovici-Collins-etc}.  Indeed, \cite[Thm. 8.2]{Bercovici-Collins-etc} gave an alternate proof of Proposition \ref{prop gp 2}; unfortunately, this proof was flawed (as will be discussed in Section \ref{section flaw}).  The present paper arose, in part, as an attempt to address this flaw.
\end{remark}

The process $t\mapsto (p_t,q)$ (as $t$ ranges through $[0,\infty)$) is known as the {\em free liberation} of the initial pair $(p,q)$.  It was introduced in \cite{Voiculescu-free-analogs-VI} as a technical tool for the analysis of free entropy and free Fisher information.  As $t\to\infty$, the free unitary Brownian motion $u_t$ tends (in the weak sense) to a Haar unitary operator.  (This is the infinite-dimensional version of the statement that the heat kernel measure on the unitary group flows towards the Haar measure.)  Thus, the pair $(p_t,q)$ tends towards $(upu^\ast,q)$ where $u$ is a Haar unitary free from $\{p,q\}$.  The two operators $upu^\ast$ and $q$ are therefore free (cf. \cite{Nica Speicher Book}), and so the spectral measure $\mu$ of the self-adjoint operator $q^{1/2}upu^\ast q^{1/2} = qp_tq$ is given by the free multiplicative convolution of the spectral measures of $upu^\ast$ and $q$ separately (cf.\ \cite{VDN}).  Let $\t(p)=\a$ and $\t(q)=\b$.  Then $\mu_{upu^\ast} = \mu_p = (1-\a)\delta_0 + \a\delta_1$ while $\mu_q = (1-\b)\delta_0 + \b\delta_1$.  This convolution was calculated explicitly in \cite[Ex.\ 3.6.7]{VDN}, via the Cauchy transform (cf. Section \ref{sect free probability}): setting $\mu = \mu_p\boxtimes\mu_q$, the authors show that
\begin{equation} \label{eq G infinity} G_\mu(z) = \int_0^1 \frac{\mu(dx)}{z-x} = \frac{z+\a+\b-2 - \sqrt{z^2-2(\a+\b-2\a\b)z+(\a-\b)^2}}{2z(z-1)}, \quad z\in\C_+. \end{equation}
The Stieltjes inversion formula (again see Section \ref{sect free probability}) then shows that
\begin{equation} \label{eq mu infinity} \mu = (1-\min\{\a,\b\})\delta_0 + \max\{\a+\b-1,0\}\delta_1 + \frac{\sqrt{(r_+-x)(x-r_-)}}{2\pi x(1-x)}\1_{[r_-,r_+]}\,dx,
\end{equation}
where $r_{\pm} = \a+\b-2\a\b\pm2\sqrt{\a\b(1-\a)(1-\b)}$ are the roots of the quadratic polynomial in the radical in Equation \ref{eq G infinity}.  It should be noted that this measure is precisely the limit of the empirical eigenvalue distribution of a Jacobi Ensemble in random matrix theory, which has recently been much studied in part due to its applications to MANOVA (multivariate analysis of variance) problems in statistics, cf.\ \cite{Collins Jacobi,Demni 1,Demni 2,Demni 3,Erdos Farrell,Farrell}.

The general position rank shows up as the mass of the spectral measure $\mu$ concentrated at $1$; this is no accident.  By a theorem of von Neumann (cf.\ \cite{von Neumann}), $p\wedge q$ is the weak $\lim_{n\to\infty} (pq)^n$, and so
\begin{equation} \label{eq gp mass at 1} \t(p\wedge q) = \lim_{n\to\infty} \t[(pq)^n] = \lim_{n\to\infty}\t[(qpq)^n] = \lim_{n\to\infty} \int_0^1 x^n\,\mu_{qpq}(dx) \end{equation}
which is precisely equal to $\mu_{qpq}\{1\}$.  Proposition \ref{prop gp 2} shows that this point mass at $1$ is structural: it is present in the law of $qp_tq$ for all $t>0$.

In the present paper, we study the law $\mu_t = \mu_{qp_tq}$ for all $t\ge 0$.  As the calculation above shows, in the limit as $t\to\infty$, the operator-valued angle $qp_tq$ flows towards a universal form determined only by $\t(p)$ and $\t(q)$.  For any $t>0$, the law $\mu_t$ completely determines the structure of the von Neumann algebra generated by $p_t$ and $q$, cf.\ \cite{Hiai Ueda 1,Hiai Ueda 2,Raeburn Sinclair}.  Properties of this measure $\mu_t$ are important for the analysis of free entropy in \cite{Hiai Ueda 2}.  The main theorems of this paper relate to the analysis of this flow.  Since $\mu$ (the weak limit of $\mu_t$ as $t\to\infty$) has a (structural) mass of $1-\min\{\a,\b\}$ at $0$, it is sensible to remove this static singularity from the dynamics.  The first major theorem of this paper shows that this produces a smooth flow equation for the Cauchy transform of $\mu_t$; and the Cauchy transform itself is analytic in both $z$ {\em and} $t$.

\begin{theorem} \label{thm PDE} Let $p,q$ be projections in a $\mathrm{II}_1$-factor, and let $u_t$ be a free unitary Brownian motion freely independent from $\{p,q\}$; set $p_t=u_tpu_t^\ast$.  Let $\t(p)=\alpha$ and $\t(q)=\beta$, and let $\mu_t$ denote the spectral measure of the operator-valued angle $qp_tq$.  For $t>0$ and $z$ in the upper half-plane $\C_+$, let
\begin{equation} \label{eq G} G(t,z) = \int_0^1 \frac{\mu_t(dx)}{z-x} - \frac{1-\min\{\alpha,\beta\}}{z}. \end{equation}
Then the function $G$ is analytic in {\em both $z\in\C_+$ and $t>0$}, and satisfies the complex PDE
\begin{equation} \label{eq PDE} \frac{\del}{\del t}G = \frac{\del}{\del z}\left[z(z-1)G^2-(az+b)G\right] \end{equation}
where $a=2\min\{\alpha,\beta\}-1$ and $b=|\alpha-\beta|$. \end{theorem}

\begin{remark} PDE \ref{eq PDE} is qualitatively similar to the complex (inviscid) Burger's equation, which is a well-known example exhibiting blow-up in finite time with bounded initial data, as well as shock behavior.  There is no reason to expect the stated analyticity result to follow from the PDE; in fact, our analysis will use tools from PDE, complex analysis, and from free probability to prove the a priori analyticity (in both variables) of the function $G$.
\end{remark}

A corollary to Theorem \ref{thm PDE} is the following discontinuity result for the flow of the joint law of $(p_t,q)$.

\begin{corollary} \label{cor vn(p_t,q) 1} Let $p,q$ be projections and let $u_t$ be a free unitary Brownian motion and $u_\infty$ a Haar unitary, both freely independent from $\{p,q\}$.  Set $p_t=u_tpu_t^\ast$ and $p_\infty=u_\infty p u_\infty^\ast$. Assume that the support of the spectral measure of $qpq$ is not equal to $[r_-,r_+]$ (cf.\ Equation \ref{eq mu infinity}). Then, for all sufficiently small $t>0$, the law of $(p_\infty,q)$ is  {\em not absolutely continuous} with respect to the law of $(p_t,q)$. \end{corollary}

\begin{proof} Since the flow $G(t,z)$ of the Cauchy transform of the spectral measure $\mu_t$ of $qp_tq$ is analytic in $t>0$ (cf.\ Theorem \ref{thm PDE}), the support of $\mu_t$ flows continuously, and hence cannot equal the support $[r_-,r_+]$ of the limit free product measure $\mu_p\boxtimes\mu_q$; this proves the result.
\end{proof}

\begin{remark} This actually shows a somewhat stronger claim: let $(E_{ij}^t)_{0\le i,j\le 1}$ be the four projections corresponding to $p_t,q$ as in Section \ref{section free entropy for projections} below ($E_{11} = p_t\wedge q$, $E_{10} = p_t\wedge q^\perp$, etc.).  Then, if the initial support of the spectral measure of $qpq$ is not equal to $[r_-,r_+]$, it follows that the {\em $C^\ast$-algebras}
\[ C^\ast(p_t,q,(E^t_{ij})_{0\le i,j\le 1}) \not\cong C^\ast(p_\infty,q,(E^\infty_{ij})_{0\le i,j\le 1})  \]
for all sufficiently small $t>0$.
\end{remark}

Section \ref{sect flow of mu t} is concerned with the proof of Theorem \ref{thm PDE}.  We use free stochastic calculus (Section \ref{sect fbm}) to find a system of ODEs satisfied by the time-dependent moments $\t((qp_tq)^n)$ for $n\in\N$.  These moments are the coefficients of the Laurent-series expansion of $G$ in a neighbourhood of $\infty$, and on this domain the ODEs combine to give PDE \ref{eq PDE}.  Analyticity, and continuation to all of $\C_+$, is then proved with careful estimates on the growth of derivatives given by the original ODEs and iteration of the PDE.

\begin{remark} \label{remark Benaych-Georges Levy} The papers \cite{Benaych-Georges Levy} and \cite{Demni 1} consider similar situations, developing PDEs governing the flow of analytic function transforms of spectral measures associated to the free liberation of two operators; in both cases, the PDEs are compatible with Theorem \ref{thm PDE} above.  In \cite{Benaych-Georges Levy}, the two initial operators are assumed to be classically independent, hence {\em commutative}; in this case,  the authors were able to solve the PDE explicitly: their solution is a dilation of the law of the free unitary Brownian motion $u_t$, studied in \cite{ref Biane fubm}.  The measure analogous to our $\mu_t$ is given meaning as a {\em $t$-free convolution}, yielding a continuous interpolation from classical independence to free independence.  The paper \cite{Demni 1} considers a situation similar to ours, with the assumption that one projection dominates the other; hence, the scope is similarly less extensive than the general situation we presently treat.  The benefit of these specializations is that they obtain explicit solutions in particular cases, as well as unexpected algebraic identities.
\end{remark}

In principle, the measure $\mu_t$ can be recovered from the function $G(t,z)$ via the Stieltjes inversion formula (cf. Section \ref{sect free probability}).  In practice, understanding the flow of the {\em boundary values} of the function from a PDE in the interior is a very difficult problem in partial differential equations.  We are, as yet, unable to complete the analysis of the measure $\mu_t$ in general; however, in the special case $\alpha=\beta=\frac12$ (corresponding to $a=b=0$ in PDE \ref{eq PDE}), we have the following complete analysis.

\begin{theorem} \label{thm smoothing} Let $\mu_t=\frac12\delta_0+\nu_t$ be the spectral measure in Theorem \ref{thm PDE} in the special case $\alpha=\beta=\frac12$.  Then for any $t>0$, the measure $\nu_t$ possesses a continuous density $\rho_t$ on $(0,1)$.  For any $t_0>0$, there is a constant $C(t_0)$ so that, for all $t\ge t_0$,
\begin{equation} \label{eq measure bound} \rho_t(x) \le \frac{C(t_0)}{\sqrt{x(1-x)}}. \end{equation}
Finally, the function $\rho_t$ is real analytic on the set $\{x\in(0,1)\colon \rho_t(x)>0\}$.
\end{theorem}
Note that the bound on $\rho_t$ precisely reflects the asymptotic form the measure takes as $t\to\infty$: in the case $\alpha=\beta=\frac12$, the Jacobi density of Equation \ref{eq mu infinity} reduces to the (shifted) arcsine law $(2\pi)^{-1}(x(1-x))^{-1/2}$ on $[0,1]$; thus, we cannot expect any better behavior at the boundary of the interval.  Theorem \ref{thm smoothing} shows that the measure $\mu_t$ does not possess any mass at the endpoints: although it may blow up at $x=1$, the singularity is milder than would reflect the Cauchy transform of a point mass.

The smoothing results of Theorem \ref{thm smoothing} mirror similar properties of the so-called {\em free heat flow}: if $\mu$ is a compactly-supported measure on $\R$ and $\sigma_t(dx) = \frac{1}{2\pi t}\sqrt{(4t-x^2)_+}\,dx$ is the semicircular law of variance $t$, then the free convolution $\mu\boxplus\sigma_t$ possesses a continuous density which is real analytic on the set where it is $>0$, cf.\ \cite{Biane free heat}.  The techniques used to prove this theorem do not involve any PDEs, but are based on Biane's theory of {\em subordination} for the Cauchy transform.  Motivated by those ideas, our approach to Theorem \ref{thm smoothing} is to develop an analogous theory of subordination for the liberation process: we show that, changing variables $H(t,z) = \sqrt{z}\sqrt{z-1}G(t,z)$, the flow of $H$ may be encoded by a deformation of the identity in the initial condition: $H(t,z) = H(0,f_t(z))$ for a subordinate function $f_t$ which extends to a homeomorphism from the closed upper half-plane $\overline{\C_+}$ onto a domain $\overline{\Omega_t}\subseteq \overline{\C_+}$.

\begin{remark} In the recent preprint \cite{Zhong}, Zhong has studied the multiplicative version of Biane's free heat flow, using related subordination technology to prove similar smoothness properties of the law of the free unitary Brownian motion multiplied by a free unitary.  These results bear on the very recent related work of Izumi and Ueda \cite{Izumi-Ueda}, as described below in Remark \ref{r.Izumi-Ueda} below. \end{remark}

Aside from the motivating question of general position for projections, the main application of the present results is to Voiculescu's theory of free entropy.  In a (still ongoing) attempt to resolve the free group factor isomorphism problem, Voiculescu invented free probability as a way to import tools from classical probability, notably entropy and information theory, to produce invariants to distinguish von Neumann algebras.  Motivated both by Shannon's original entropy constructions involving spacial microstates, and more sophisticated constructions in information theory (using conjugate variables, for example), he introduced free analogues of entropy, Fisher's information, and mutual information of collections of non-commuting random variables, in a series of six papers from 1993--1996; the most relevant for our purposes is \cite{Voiculescu-free-analogs-VI}.  Classically, there are many relationships that hold between these information measures; the general question of proving the analogous relationships for their free counterparts is known as the {\em Unification Conjecture} in free probability.  In Section \ref{section unification conjecture}, we briefly describe the precise form of the unification conjecture, and prove it in the special case of von Neumann algebras generated by two projections (of trace $\frac12$).  Our main result in this direction is as follows.

\begin{theorem} \label{thm unification conjecture} Let $p,q\in(\mathscr{A},\t)$ be two projections in a $\mathrm{II}_1$-factor, with trace $\t(p)=\t(q)=\frac12$.  Then the free mutual information of $p,q$ and the free relative entropy of $p,q$ are equal:
\begin{equation} \label{eq i=i*} i^\ast(p,q) = -\chi_{\mathrm{proj}}(p,q) +\chi_{\mathrm{proj}}(p) + \chi_{\mathrm{proj}}(q) = -\chi_{\mathrm{proj}}(p,q). \end{equation}
\end{theorem}
\noindent The assumption on the traces is in place since it is required for Theorem \ref{thm smoothing}; we fully expect techniques similar to those presently developed will solve the general problem for all traces.

\begin{remark} \label{r.Izumi-Ueda} Theorem \ref{thm unification conjecture} was also recently proved in the preprint \cite{Izumi-Ueda}, which was posted to the arXiv seven months after the first public version of the present manuscript.  Indeed, using the Jakowski transform as in \cite{Demni 1}, Izumi and Ueda show that $\mu_t$ can be identified with the spectral measure of the product of a free unitary Brownian motion with a free unitary operator whose distribution is determined by the initial projections $(p,q)$.  Thus, the appropriate smoothing analogues of our Theorem \ref{thm smoothing} follow from \cite{Zhong}.  They use this to prove Theorem \ref{thm unification conjecture} much the same way we do below.  They go further, and prove subordination results akin to our Section \ref{section subordination} below, without the restriction that $\tau(p)=\tau(q)=\frac12$, and they use this to give some partial results generalizing Theorem \ref{thm unification conjecture} beyond the trace $\frac12$ regime.  The present authors find this to be a promising avenue of research.  \end{remark}

This paper is organized as follows.  The remainder of Section \ref{section introduction} is devoted to the relevant background for the rest of the work presented: Section \ref{sect free probability} fixes the basic ideas and notation of free probability; Section \ref{sect fbm} discusses the free It\^o calculus; Section \ref{sect fubm} is devoted to the free unitary Brownian motion that is central to this paper; and Section \ref{section flaw} describes the flaw in the proof of the general position result \cite[Thm. 8.2]{Bercovici-Collins-etc} that partly motivated the present work.  The main results of this paper are in Sections \ref{sect flow of mu t}--\ref{section unification conjecture}: Section \ref{sect flow of mu t} is the proof of Theorem \ref{thm PDE}; Section \ref{section local properties} uses this result to develop local properties of the measure $\mu_t$, including the proof of Theorem \ref{thm smoothing}; and Section \ref{section unification conjecture} presents our main application, proving Theorem \ref{thm unification conjecture}.

\subsection{Free Probability} \label{sect free probability} Here we briefly record the basic ideas and notation used in the sequel; the reader is directed to the books  \cite{Nica Speicher Book} and \cite{VDN} for a full treatment.  The setting of free probability is a non-commutative probability space; we will work in the richer framework of a {\bf $W^\ast$-probability space} $(\mathscr{A},\tau)$ where $\mathscr{A}$ is a von Neumann algebra, and $\tau$ is a normal, faithful tracial state on $\mathscr{A}$.  The elements in $\mathscr{A}$ are called (non-commutative) random variables.  The motivating example is the space $\mathscr{A}=\mathrm{End}_\C(\C^d)\tensor L^\infty(\Omega,\mathscr{F},\P)$ of all random matrices with bounded entries (over a probability space $(\Omega,\mathscr{F},\P)$); here the tracial state is $\tau = \frac{1}{d}\mathrm{Tr}_d\tensor \E$, the expected normalized trace.  Voiculescu's key observation was that this $W^\ast$-probability space can be used to approximate (in the sense of moments) non-commutative random variables in infinite-dimensional von Neumann algebras, and classical independence combined with random rotation of eigenvectors for random matrices gives rise to a more general independence notion modeled on free groups.

\begin{definition} \label{def free} Let $(\mathscr{A},\tau)$ be a $W^\ast$-probability space.  The $\ast$-subalgebras $A_1,\ldots,A_n\subseteq\mathscr{A}$ are called {\bf free} or {\bf freely independent} if, given any centered elements $a_i\in A_i$, $\tau(a_i)=0$, and any sequence $i_1,i_2,\ldots,i_m\in\{1,\ldots,n\}$ with $i_k\ne i_{k+1}$ for $1\le j<m$, $\tau(a_{i_1}a_{i_2}\cdots a_{i_m}) = 0$.  Two random variables $a,b\in\mathscr{A}$ are freely independent if the $\ast$-subalgebras they generate are freely independent.
\end{definition}

\noindent Free independence is a moment-factorization condition.  For example, if $a$ and $b$ are freely independent, then $\tau(a^nb^m) = \tau(a^n)\tau(b^m)$ for any $n,m\in\N$, coinciding the classical independence of bounded random variables; for non-commutating variables, freeness includes more complicated factorizations such as $\tau(abab) = \tau(a^2)\tau(b)^2 + \tau(a)^2\tau(b^2) - \tau(a)^2\tau(b)^2$.  Freeness is modeled on freeness in group theory.  Let $\mathscr{A} = L(\mathbb{F}_k)$ denote the {\bf free group factor} with $k$ generators $u_1,\ldots, u_k$.  ($L(\mathbb{F}_k)$ is the von Neumann algebra generated by the left-regular representation of the free group $\mathbb{F}_k$ on $\ell^2(\mathbb{F}_k)$; i.e.\ the von Neumann algebra generated by convolution on the free group.)  Then subalgebras generated by disjoint subsets of the generators $\{u_1,\ldots,u_k\}$ are freely independent.

If $\{(\mathscr{A}_n,\tau_n)\colon 1\le n\le \infty\}$ are $W^\ast$-probability spaces, and $k\in\N$, a sequence $\mx{a}_n=(a_n^1,\ldots,a_n^k)\in (\mathscr{A}_n)^k$ is said to {\bf converge in distribution} to $\mx{a} = (a^1,\ldots,a^k)\in\mathscr{A}_\infty^k$ if, for any polynomial $P$ in $k$ non-commuting indeterminates, $\tau_n[P(\mx{a}_n)] \to \tau_\infty[P(\mx{a})]$ as $n\to\infty$.  That is: each mixed moment in $\mx{a}_n$ converges to that moment in $\mx{a}$.  Such a sequence is said to be {\bf asymptotically free} if its limit consists of random variables $a^1,\ldots,a^k$ that are freely independent.

\begin{proposition}[\cite{Nica Speicher Book,VDN}] \label{prop asymptotic freeness} Let $X_n$ and $Y_n$ be $n\times n$ random matrices with all moments finite, and suppose that each has a limit in distribution separately.  Let $U_n$ be a unitary matrix sampled from the Haar measure on $U(n)$, independent from $X_n$ and $Y_n$.  Then $(X_n, U_nY_nU_n^\ast)$ are asymptotically free. \end{proposition}

\noindent Proposition \ref{prop asymptotic freeness} asserts that freeness is an asymptotic statement about the {\em eigenvectors} of random matrices: freeness means that their eigenspaces are independently, uniformly randomly rotated against each other.  This result allows the realization of free random variables and stochastic processes as limits of random matrix ensembles; more on that in Section \ref{sect fbm}.

If $a\in(\mathscr{A},\tau)$ is a single self-adjoint random variable, it possesses a spectral resolution $E^a$ taking values in the projections of $\mathscr{A}$.  The composition $\tau\circ E^a$ produces a probability measure $\mu_a$ on the spectrum of $a$, known as the {\bf spectral measure} of $a$; it is determined by the moments of $a$:
\[ \tau(a^n) = \int_\R x^n\,\mu_a(dx). \]
In general the non-commutative distribution of a random vector $\mx{a}\in\mathscr{A}^k$ is the set of all traces of non-commutative polynomials in $\mx{a}$; in the case of a single self-adjoint element ($k=1$), these moments are encoded by the single measure $\mu_a$, coinciding with the law of a classical random variable.  In the case of a self-adjoint random matrix $a$, $\mu_a$ is the average empirical eigenvalue distribution (the average of the random probability measure placing a point-mass at each eigenvalue of the matrix).

\begin{definition} \label{def Cauchy transform} Let $\mu$ be a compactly-supported finite positive measure on $\R$.  The {\bf Cauchy transform} $G_\mu$ is the analytic function on the upper half-plane $\C_+$ defined by
\begin{equation} \label{eq Cauchy transform} G_\mu(z) = \int_\R \frac{1}{z-x}\,\mu(dx). \end{equation}
\end{definition}
\noindent The function $G_\mu$ is analytic on $\C-\supp\mu$, but does not have a continuous extension across $\supp\mu$.  A generally useful uniform estimate for the Cauchy transform in the upper half-plane is
\begin{equation} \label{eq G uniform bound} |G_\mu(z)| \le \frac{\mu(\R)}{|\Im z|}. \end{equation}
Another reason it is customary to restrict it to the upper half-plane is that the measure can be recovered from its action there, via the {\bf Stieltjes inversion formula}:
\begin{equation} \label{eq Stieltjes} \mu(dx) = -\frac1\pi\lim_{\e\downarrow 0} \Im G_\mu(x+i\e). \end{equation}
The limit in Equation \ref{eq Stieltjes} is a weak limit: if $f$ is a continuous test function, the integral of $f$ against the $\e$-dependent measure on the right converges to $\int f\,d\mu$; if $\mu$ possesses a sufficiently integrable continuous density, the limit is also true pointwise for the density.  If $\mu(dx) = \rho(x)\,dx$ where $\rho\in L^p(\R)$ for some $p\in(0,\infty)$, then the {\bf Hilbert transform}
\begin{equation} \label{eq Hilbert transform} H\rho(x) \equiv \frac{1}{\pi}\;p.v.\!\!\int\frac{\rho(y)}{x-y}\,dy \end{equation}
is also in $L^p$, and gives the boundary values of the real part of $G_\mu$; that is
\begin{equation} \label{eq Stieltjes Hilbert} -\frac{1}{\pi}\lim_{\e\downarrow 0}\Im G_\mu(x+i\e) = \rho(x), \quad \frac{1}{\pi}\lim_{\e\downarrow 0}\Re G_\mu(x+i\e) = H\rho(x). \end{equation}

The Stieltjes inversion formula is also robust under vague limits of measures.
\begin{theorem}[Stieltjes continuity theorem] \label{thm Stieltjes continuity} Let $\mu_n$ be a sequence of probability measures on $\R$.
\begin{itemize}
\item[(1)] If $\mu_n\to \mu$ weakly, then $G_{\mu_n}\to G_{\mu}$ uniformly on compact subsets of $\C_+$.
\item[(2)] Conversely, if $G_{\mu_n}\to G$ pointwise on $\C_+$, then $G$ is the Cauchy transform of a finite positive measure $\mu$, and $\mu_n\to\mu$ vaguely.  If it is known a priori that $G=G_\mu$ for a probability measure $\mu$, then $\mu_n\to\mu$ weakly.
\end{itemize}
 \end{theorem}
\noindent It is possible for mass to escape at $\infty$ in a vague limit; for example, it is possible for $G_{\mu_n}\to 0$ pointwise.  In our applications, the limit will be directly identified as the Cauchy transform of a probability measure.

Freeness, which is a property of moments (hence of distributions), can be encoded in terms of the Cauchy transform.  Given a compactly-supported probability measure $\mu$, the {\em $\mathscr{R}$-transform $\mathscr{R}_\mu$} is the analytic function on a neighborhood of the identity in the upper half-plane determined by the functional equation
\begin{equation} \label{eq R-transform} G_\mu(\mathscr{R}_\mu(z)+1/z)=z, \quad z\in\C_+,\quad |z|\text{ small}. \end{equation}
The $\mathscr{R}$-transform in fact determines the Cauchy transform, by analytic continuation, modulo the constraint $\lim_{|z|\to\infty}zG_\mu(z)=\mu(\R)=1$.  It is the free analogue of the (log-)Fourier transform: it linearizes freeness.

\begin{proposition}[\cite{VDN}] \label{prop free sum} Let $a,b$ be self-adjoint random variables.  Then $a,b$ are freely independent if and only if
\[ \mathscr{R}_{\mu_{a+b}}(z) = \mathscr{R}_{\mu_a}(z) + \mathscr{R}_{\mu_b}(z) \]
for all sufficiently small $z\in\C_+$. \end{proposition}
\noindent This highlights the fact that free independence yields a free convolution operation on measures.  Given two (compactly-supported) probability measures $\mu,\nu$, their {\bf (additive) free convolution} $\mu\boxplus\nu$ is the measure determined by $\mathscr{R}_{\mu\boxplus\nu} = \mathscr{R}_\mu+\mathscr{R}_\nu$.  That is: realize $\mu$ and $\nu$ as the laws of two free random variables $a,b$; then $\mu\boxplus\nu$ is the law of $a+b$.

There is a similar notion of {\bf (multiplicative) free convolution} for positive operators.  If $a$ and $b$ are free, then the law of $a^{1/2}ba^{1/2}$ is denoted $\mu_a\boxtimes \mu_b$.  It is determined through another analytic function transform known as the {\em $\mathscr{S}$-transform}.  For a compactly-supported probability measure $\mu$, let $\chi_\mu(z) = \frac{1}{z}G_\mu(\frac1z)-1$ be its shifted moment-generating function, defined in a neighborhood of $0$.  The $\mathscr{S}$-transform $\mathscr{S}_\mu$ is the analytic function defined in a neighborhood of $0$ by the functional equation
\begin{equation} \label{eq S-transform} \chi_\mu\left(\textstyle{\frac{z}{z+1}}\mathscr{S}_\mu(z)\right) = z, \quad |z|\text{ small}. \end{equation}
Then $\mathscr{S}_{\mu\boxtimes\nu}(z) = \mathscr{S}_\mu(z)\mathscr{S}_\nu(z)$, as shown in \cite{VDN}.  It is through this operation that the free Jacobi law of Equation \ref{eq mu infinity} was determined.

\subsection{Free Brownian Motion and Free Stochastic Calculus} \label{sect fbm} Let $X_n(t)$ denote an $n\times n$ Hermitian matrix-valued Brownian motion: the upper-triangular entries $[X_n(t)]_{ij}, i<j$ are independent complex Brownian motions of total variance $t/n$; the diagonal does not matter to the limit. Then there is a $W^\ast$-probability space $(\mathscr{A},\tau)$ in which the limit in distribution of any collection of instances of $X_n$ at different times $(X_n(t_1),X_n(t_2),\ldots,X_n(t_k))$ for $0\le t_1<t_2<\cdots<t_k$ may be realized.  For fixed $t$, $X_n(t)$ is a Gaussian unitary ensemble, whose limit empirical eigenvalue distribution is Wigner's semicircle law $\sigma_t(dx)=\frac{1}{2\pi t}\sqrt{(4t-x^2)_+}\,dx$ \cite{Wigner 1,Wigner 2}; hence the limits $x_{t_1},\ldots,x_{t_k}$ are semicircular random variables.   Mirroring the isonormal Gaussian process construction of Brownian motion, the limit may be realized in any $W^\ast$-probability space rich enough to contain an infinite sequence of freely independent identically distributed semicircular random variables (e.g.\ any free group factor).  The limit $x_t$ is a non-commutative stochastic process with properties analogous to Brownian motion.

\begin{definition} \label{def free Brownian motion} A(n additive) {\bf free Brownian motion} in a $W^\ast$-probability space $(\mathscr{A},\tau)$ is a non-commutative stochastic process $(x_t)_{t\ge 0}$ with the following properties:
\begin{itemize}
\item[(1)] The increments of $x_t$ are freely independent: for $0\le t_1<t_2<\cdots<t_k$,
\[ x_{t_2}-x_{t_1},x_{t_3}-x_{t_2},\ldots,x_{t_k}-x_{t_{k-1}} \]
are freely independent.
\item[(2)] The process is stationary, with semicircular increments: for $0\le s<t$, the law of $x_t-x_s$ is $\sigma_{t-s}$.
\item[(3)] The $\mathscr{A}$-valued function $t\mapsto x_t$ is weakly continuous.
\end{itemize}
\end{definition}
\noindent Free Brownian motion is the limit of matrix-valued Brownian motion; it can also be constructed as an isonormal process, or through a Fock space construction, cf.\ \cite{Biane Speicher 1}.  With the properties of Definition \ref{def free Brownian motion} in hand, the standard construction of the It\^o integral may be mirrored.  If $\theta_t$ is a process adapted to $x_t$ (meaning that $\theta_t$ is in the von Neumann subalgebra generated by $\{x_s\}_{s\le t}$ for each $t\ge 0$), then one can define the stochastic integral
\[ \int \theta_t\,dx_t \]
as an $L^2(\mathscr{A},\tau)$-limit of step functions of the form $\sum_k \theta_{t_k}(x_{t_k}-x_{t_{k-1}})$.  The relationship $\phi_t = \int \theta_t\,dx_t$ is abbreviated as $d\phi_t = \theta_t\,dx_t$.  A (left) {\bf free It\^o process} $y_t$ is a stochastic process of the form
\begin{equation} \label{eq free Ito process 1} y_t = \int_0^t \theta_s\,dx_s + \int_0^t \phi_s\,ds \end{equation}
where $\theta_t$ and $\phi_t$ are adapted processes; here $\int_0^t \theta_s\,dx_s$ is short-hand for $\int \theta_s \1_{[0,t]}(s)\,dx_s$ and so forth.  Evidently, a free It\^o process is adapted.  Equation \ref{eq free Ito process 1} is usually written in the form
\begin{equation} \label{eq free Ito process 2} dy_t = \theta_t\,dx_t + \phi_t\,dt. \end{equation}
This stochastic differential notation is useful, and allows for succinct description of the rules of free It\^o calculus.  The most important is the {\bf free It\^o formula} which we will use in product form:
\begin{equation} \label{eq Ito product rule} d(y_t z_t) = (dy_t)z_t + y_t(dz_t)+(dy_t)(dz_t). \end{equation}
Here $y_t$ and $z_t$ are free It\^o processes.  For example, if $dy_t = \theta_t\,dx_t + \phi_t\,dt$ and $dz_t = \theta'_t\,dx_t + \phi'_t\,dt$, then
\[ y_t(dz_t) = y_t(\theta'_t\,dx_t + \phi'_t\,dt) = (y_t\theta'_t)\,dx_t + (y_t\phi'_t)\,dt \]
while
\[ (dy_t)z_t = (\theta_t\,dx_t + \phi_t\,dt)z_t = \theta_t\,dx_t\,z_t + (\theta_t z_t)\,dt. \]
The processes are non-commutative, so we must be able to make sense of such terms; moreover, the product $(dy_t)(dz_t)$ will contain mixed terms $dx_t\,dt$ and so forth.  The rules, akin to standard It\^o calculus, for these terms are as follows:

\begin{align} \label{eq Ito calculus} dx_t\,\theta_t\,dx_t = \t(\theta_t)\,dt \quad\;\; \\
 \label{eq Ito calculus 2} dx_t\, dt = dt\, dx_t = (dt)^2 = 0 \end{align}
Equation \ref{eq Ito calculus} may seem counter-intuitive from classical It\^o calculus: for a standard $1$-dimensional Brownian motion $B_t$, $(dB_t)^2 = dt$.  It is easy to calculate, however, that for $n\times n$ matrix-valued Brownian motion $X_n(t)$ and an adapted matrix-valued process $Y_n(t)$, the It\^o product-rule takes the form $dX_n(t)\,Y(t) dX_n(t) = \frac{1}{n}\Tr(Y_n(t))\,dt$.  Therefore, Equation \ref{eq Ito calculus} follows from standard trace-concentration results, cf.\ \cite{Biane Speicher 1}.  One final useful result is that free It\^o integrals of adapted processes, like their classical cousins, are centered; i.e.\
\begin{equation} \label{eq st int tr 0} \t(\theta_t\,dx_t) = 0. \end{equation}
The standard approach from classical stochastic calculus (using the Picard iteration technique, for example) can then be used to solve {\bf free stochastic differential equations} of the form
\begin{equation} \label{eq fSDE} dy_t = a(t,y_t)\,dx_t + b(t,y_t)\,dt \end{equation}
for sufficiently smooth and slowly-growing functions $a,b\colon\R_+\times\mathscr{A}\to\mathscr{A}$.  The reader is referred to \cite{Biane Speicher 1,Biane Speicher 2} for details.

\begin{remark} As the calculations following Equation \ref{eq Ito product rule} demonstrate, the product of two {\em left} free It\^o processes is not, in general, a {\em left} free It\^o process: it may contain terms like $dx_t\,\theta_t$ in differential form, which is not equal to $\theta_t\,dx_t$.  A proper treatment of free It\^o calculus should be formulated in terms of biprocesses $t\mapsto \omega_t\in\mathscr{A}\tensor\mathscr{A}$ that can act on the left and the right simultaneously.  Biane and Speicher develop this theory in \cite{Biane Speicher 1}; the corresponding stochastic integral is denoted
\[ \int \omega_t\sharp dx_t, \]
defined as an $L^2$-limit of sums of the form $\sum_k \theta_{t_k}(x_{t_k}-x_{t_{k-1}})\phi_{t_k}$ where $\omega_t$ is approximated by $\sum_k \theta_{t_k}\!\tensor \phi_{t_k}$.  We have described free It\^o calculus in slightly imprecise terms to avoid the new notational complexity; for our purposes, the more general theory is essentially unnecessary.  The interested reader can find a careful overview in \cite{KNPS}.
\end{remark}

\subsection{Free Unitary Brownian Motion} \label{sect fubm} Introduced in \cite{ref Biane fubm}, the {\bf free unitary Brownian motion} is the solution to the free It\^o stochastic differential equation
\begin{equation} \label{eq fubm sde} du_t = iu_t\,dx_t - \frac12 u_t\,dt \end{equation}
with initial condition $u_0=1$; here, as usual, $x_t$ is a(n additive) free Brownian motion.  It would, perhaps, be most accurate to call the solution of this free SDE a {\em left} free unitary Brownian motion.  To be sure, note that the adjoint $u_t^\ast$ satisfies
\begin{equation} \label{eq fubm sde*} du_t^\ast = (du_t)^\ast = -idx_t\,u_t^\ast - \frac12 u_t^\ast\,dt \end{equation}
since $x_t$ is self-adjoint.  Equation \ref{eq fubm sde} is the exact free analogue of Equation \ref{eq SDE 1} for the finite-dimensional unitary Brownian motion.  Indeed, this comes from the fact that $u_t$ is the limit in distribution of the Brownian motion on the unitary group, in the same sense that $x_t$ is the limit in distribution of the Hermitian-matrix valued Brownian motion.  It is a unitary-valued stochastic process, whose distribution is a continuous function on the unit circle for each $t>0$.  Biane calculated the moments of this measure:
\begin{equation} \label{fubm moments} \tau((u_t)^k) = e^{-kt/2}\sum_{j=0}^{k-1}\frac{(-t)^j}{j!}\binom{k}{j+1}k^{j-1}, \quad k\ge0. \end{equation}
Using these moments, it is possible to compute an implicit description of the density of $\mu_{u_t}$ (with respect to the uniform probability measure on the unit circle); for $t<2$, the measure is supported in a strict symmetric subset, and then achieves full support for $t\ge 2$.

Analogous to Definition \ref{def free Brownian motion}, and in line with the properties of the Brownian motion on the unitary groups, the free unitary Brownian motion has the following properties, which can be derived directly from the stochastic differential equation \ref{eq fubm sde}.

\begin{proposition}[\cite{ref Biane fubm}] Let $\nu_t$ be the measure on the unit circle possessing the moments on the right-hand-side of Equation \ref{fubm moments}.  Then the free unitary Brownian motion satisfies the following properties.
\begin{itemize}
\item[(1)] The multiplicative increments of $u_t$ are freely independent: for $0\le t_1<t_2<\cdots<t_k$,
\[ u_{t_1}^\ast u_{t_2},\; u_{t_2}^\ast u_{t_3},\;\ldots,\;u_{t_{k-1}}^\ast u_{t_k}\]
are freely independent.
\item[(2)] The process is stationary: for $0\le s< t$, the law of the unitary random variable $u_tu_s^\ast$ is $\mu_{u_{t-s}}$.
\item[(3)] The $\mathscr{A}$-valued function $t\mapsto u_t$ is weakly continuous.
\end{itemize}
\end{proposition}
\noindent Comparable statements can be made about the right free unitary Brownian motion $u_t^\ast$.

\subsection{The Flaw in \cite[Thm. 8.2]{Bercovici-Collins-etc}} \label{section flaw} The final theorem in the first author's paper \cite{Bercovici-Collins-etc} claimed to prove the motivating theorem of the present paper: if $p,q$ are projections in a $W^\ast$-probability space $(\mathscr{A},\tau)$, and if and $u_t$ is a free unitary Brownian motion, freely independent from $\{p,q\}$, then $p_t=u_tpu_t^\ast$ and $q$ are in general position for all $t>0$.  The idea of the given proof was as follows.  First, by making replacements $p\leftrightarrow 1-p$ and $q\leftrightarrow 1-q$ if necessary, it suffices to prove the theorem in the case $\tau(p),\tau(q)\le\frac12$, in which case the general position statement is that $\tau(p_t\wedge q) = 0$.  Let $r_t=p_t\wedge q$, and define a function $F_t\colon\R_+\to\R$ by
\[ F_t(s) = \tau[(r_tp_sr_t-r_t)^2], \; s\ge 0. \]
The function $F_t$ is non-negative. Since $r_t$ is a projection onto a subspace of the image of $p_t$, $r_tp_tr_t = p_t$ and so $F_t(t)=0$; hence $F_t(s)$ has a minimum at $s=t$.  The claim is then made that the function $F_t$ is differentiable, and so $F_t'(t)= 0$.  Free It\^o calculus is then applied to calculate the derivative; it is then claimed that $F_t'(t) \le -\tau(r_t)^2\tau(p)$.  Thus, $F_t'(t)<0$ unless $\tau(r_t)=0$ as required.

There are two flaws with this argument.  The first, fairly subtle, is the assumption of differentiability.  The stochastic process $s\mapsto (r_tp_sr_t-r_t)^2$ is manifestly {\em not adapted} to the filtration generated by $(u_s)_{s\ge 0}$: for $s<t$, it depends on the value of $u_t$.  It is only adapted for $s\ge t$, in which case the tools of It\^o calculus indeed apply as stated.  Thus, it was correctly demonstrated that the function $F_t(s)$ has a minimum at $s=t$, and is differentiable for $s\in[t,\infty)$ (meaning {\em right-differentiable} at $s=t$).  If it were indeed true that $\alpha'_t(t) \le -\t(r_t)^2\t(p)$, the argument would remain valid: a right-differentiable function cannot have a minimum at a point where the right-derivative is strictly negative.  Unfortunately, there was also a calculation error in the determination of this derivative.  In fact:

\begin{proposition} \label{prop flawed derivative} The right-derivative of $F_t$ satisfies $F_t'(t) = 2\t(r_t)(1-\t(p))\ge 0$.
\end{proposition}
\noindent This produces no contradiction to the possibility that $\t(r_t)>0$, however: since the function $F_t(s)$ is not known to be differentiable in a neighborhood of $s=t$, its right-derivative may well be strictly positive at the minimum (with behavior akin to the function $s\mapsto |s-t|$, for example).  The proof of Proposition \ref{prop flawed derivative} is delayed until Section \ref{section flawed derivative}.

It is possible this proof could be mended using a free version of stochastic calculus for non-adapted processes, which has yet to be developed.  In the classical case, such techniques are based on the Malliavin calculus \cite{Nualart}, which does have an analogue in the world of free probability, developed in \cite{Biane Speicher 1,Biane Speicher 2} and further developed in the second author's recent paper \cite{KNPS}.  Nevertheless, even if such a non-adapted calculus were applicable and mirrored the classical behavior, it is likely the function $F_t(s)$ could still not be proved differentiable at the point $s=t$, but only on the complement of this point in $\R_+$.  Thus, to prove the general position claim, fundamentally different techniques are required; this is part of the impetus for the present paper.

\section{The Flow of the Spectral Measure $\mu_t$} \label{sect flow of mu t}

\subsection{The Flow of Moments} \label{section 2.1} Let $p,q\in\A$ be projections with $\t(p)=\a$ and $\t(q)=\b$.  Let $u_t$ be a free unitary Brownian motion, free from $p,q$, and (as usual) define $p_t = u_tpu_t^\ast$ for $t\ge 0$.  Our present goal is to understand the moments
\begin{equation} \label{eq fn} g_n(t) = \t[(qp_tq)^n] \quad n\ge 1, t\ge 0. \end{equation}
We will use stochastic calculus to derive a system of ODEs satisfied by the functions $g_n$ on $(0,\infty)$.  To that end, we need a generalization of the It\^o product rule of Equation \ref{eq Ito product rule} to products of many adapted processes.  An easy induction argument shows that if $a_1(t),\ldots,a_n(t)$ are adapted processes then
\[ d(a_1\cdots a_n) = \sum_{j=1}^n a_1\cdots a_{j-1}\,da_j\, a_{j+1}\cdots a_n + \sum_{1\le i<j\le n} a_1\cdots a_{i-1}\,da_i\,a_{i+1}\cdots a_{j-1}\,da_j\,a_{j+1}\cdots a_n. \]
(The induction follows from the fact that $da_i\,da_j\,da_k = 0$ for any $i,j,k$; this follows from repeated applications of the rules of Equation \ref{eq Ito calculus 2}.)  Specializing to the case $a_1=\cdots=a_n=a$, we have
\begin{equation} \label{eq Ito product 2} d(a^n) = \sum_{j=1}^n a^{j-1}\,da\,a^{n-j} + \sum_{1\le i<j\le n} a^{i-1}\,da\,a^{j-i-1}\,da\,a^{n-j} \end{equation}
(Note: Equation \ref{eq Ito product 2} only makes sense for $n\ge 2$.)  We will apply this to the adapted process $a_t = qp_tq$, and then take the trace to find $dg_n(t)$.  This affords immediate simplifications:
\begin{align*} \t[d(a^n)] &= \sum_{j=1}^n \t[a^{j-1}\,da\,a^{n-j}] + \sum_{1\le i<j\le n} \t[a^{i-1}\,da\,a^{j-i-1}\,da\,a^{n-j}] \\
&= \sum_{j=1}^n \t[a^{n-1}\,da] + \sum_{1\le i<j\le n} \t[a^{n-(j-i)-1}\,da\, a^{(j-i)-1}\,da]
\end{align*}
where we have used the trace property and combined terms on the left.  The terms in the first sum do not depend on the summation variable $j$, and so this simply becomes $n\t[a^{n-1}\,da]$.  In the second summation, the summands depend on the summation variables $i,j$ only through their difference $k=j-i$ which ranges from $1$ up to $n-1$.  We therefore reindex by this variable, and that sum becomes
\[ \sum_{1\le i<j\le n} \t[a^{n-(j-i)-1}\,da\, a^{(j-i)-1}\,da] = \sum_{k=1}^{n-1} \sum_{1\le i<j\le n \atop j-i=k} \t[a^{n-k-1}\,da\, a^{k-1}\,da]. \]
For fixed $k\in\{1,\ldots,n-1\}$, the number of pairs $(i,j)$ with $j-i=k$ is equal to $n-k$, and so this summation $S$ becomes
\begin{equation} \label{eq S 1} S= \sum_{k=1}^{n-1} (n-k) \t[a^{n-k-1}\,da\,a^{k-1}\,da]. \end{equation}
Reindexing $j=n-k$ shows that this sum is also given by
\begin{equation} \label{eq S 2} S= \sum_{j=1}^{n-1} j\,\t[a^{j-1}\,da\,a^{n-j-1}\,da]. \end{equation}
Using the trace property, adding Equations \ref{eq S 1} and \ref{eq S 2} gives the simplification
\[ 2S = \sum_{j=1}^{n-1} (n-j+j)\,\t[ a^{n-j-1}\,da\,a^{j-1}\,da] = n\sum_{j=1}^{n-1} \t[a^{n-j-1}\,da\,a^{j-1}\,da]. \]
Thus, we have
\begin{equation} \label{eq trace dan 1}
\t[d(a^n)] = n\,\t[a^{n-1}\,da] + \frac12 n\,\sum_{j=1}^{n-1} \t[a^{n-j-1}\,da\,a^{j-1}\,da].
\end{equation}
Now, $a_t = qp_t q$ and so (applying the It\^o product rule \ref{eq Ito product rule} twice) $da_t = q dp_t q$.  To evaluate the differential $dp_t$, the product rule again gives
\[ dp_t = d(u_tpu_t^\ast) = (du_t)pu_t^\ast + u_t[d(pu_t^\ast)] + (du_t)[d(pu_t^\ast)], \]
and since $p$ is constant with respect to time,
\[ dp_t = (du_t)pu_t^\ast + u_tpdu_t^\ast + (du_t)p(du_t^\ast). \]
We now substitute the stochastic differential equations \ref{eq fubm sde} and \ref{eq fubm sde*} to express $dp_t$ as
\[ dp_t = (iu_t\,dx_t - \frac12u_t\,dt)pu_t^\ast + u_tp(-idx_t\,u_t^\ast-\frac12u_t^\ast\,dt) + (iu_t\,dx_t - \frac12u_t\,dt)p(-idx_t\,u_t^\ast -\frac12u_t^\ast dt). \]
The first two terms simplify to give
\[ iu_t\,dx_t\,pu_t^\ast -i u_tp\, dx_t\,u_t^\ast - u_tpu_t^\ast\,dt.  \]
The final term has only one surviving factor: by Equation \ref{eq Ito calculus}
\[ (iu_t\,dx_t)p(-idx_t\,u_t^\ast) = u_t dx_t\,p\,dx_t u_t^\ast = u_t\t(p)u_t^\ast\,dt = \t(p)\,dt. \]
Altogether, then, we have
\begin{equation} \label{eq dpt 1} dp_t = -u_tpu_t^\ast\,dt +i u_t\,dx_t\,pu_t^\ast - iu_tp\,dx_t\,u_t^\ast + \t(p)\,dt. \end{equation}
Recalling that $\t(p)=\alpha$ and that $u_tpu_t^\ast = p_t$, the first and last term combine to $(\alpha-p_t)\,dt$.  For the middle terms, it is useful to introduce a new process.  Let
\begin{equation} \label{eq dyt} dy_t = iu_t\,dx_t\,u_t^\ast. \end{equation}
While it is easy to see that this SDE has a unique solution $y_t$ (satisfying $y_0=0$), we need not concern ourselves with this fact.  Indeed, all that is important is that the rules (Equations \ref{eq Ito calculus} and \ref{eq Ito calculus 2}) of It\^o calculus applied with $y_t$ instead of $x_t$ take the following form.  If $\theta_t$ is an adapted process, then
\begin{equation} \label{eq Ito calc y 1} dy_t\,\theta_t\,dy_t = (iu_t\,dx_t\,u_t^\ast)\theta_t(iu_t\,dx_t\,u_t^\ast) = -u_t\t(u_t^\ast\theta_t u_t)u_t^\ast\,dt = -\t(\theta_t)\,dt \end{equation}
and
\begin{equation} \label{eq Ito calc y 2} dy_t\,dt = dt\,dy_t = (dt)^2 = 0. \end{equation}
Now, Equation \ref{eq dpt 1} can be rewritten as
\begin{equation} \label{eq dpt 2} dp_t = (\alpha-p_t)\,dt + dy_t\,p_t - p_t\,dy_t. \end{equation}
Thus, with $a_t = qp_tq$, it follows that
\begin{equation} \label{eq da 1} da_t = q\,dp_t\,q = (\alpha q - a_t)\,dt + q\,dy_t\,p_tq - qp_t\,dy_t\,q. \end{equation}
We now simplify the first term in Equation \ref{eq trace dan 1}.  
\[ a_t^{n-1}\,da_t = a_t^{n-1}(\alpha q - a_t)\,dt + a_t^{n-1}\,q\,dy_t\,p_tq - a_t^{n-1}qp_t\,dy_t\,q. \]
By Equation \ref{eq st int tr 0}, the last two terms have trace $0$, and so we simply have
\begin{equation} \label{eq trace dan 1.1} \t(a_t^{n-1}\,da_t) = \t[a_t^{n-1}(\alpha q-a_t)]\,dt = [\alpha\,\t(a_t^{n-1})-\t(a_t^n)]\,dt \end{equation}
where we have simplified $a_t^{n-1}q = (qp_tq)^{n-1}q = (qp_tq)^{n-1} = a_t^{n-1}$ since $q^2=q$.

For the second term (the summation) in Equation \ref{eq trace dan 1}, it is convenient to make yet another transformation.  Define $z_t$ by
\begin{equation} \label{eq dzt} dz_t = q\,dy_t\,p_tq - qp_t\,dy_t\,q. \end{equation}
Again, one can use standard theory to show that there is a unique adapted process with $z_0=0$ satisfying this SDE, but this is not important for present considerations.  The following lemma expresses the form of the It\^o calculus in terms of the process $z_t$.

\begin{lemma} \label{lemma zt} Let $z_t$ be defined by Equation \ref{eq dzt}.  Then $dz_t\,dt = dt\,dz_t = (dt)^2 = 0$, and if $\theta_t$ is an adapted process, then
\begin{equation} \label{eq dzt 1} dz_t\,\theta_t\,dz_t = \left[-2\t(a_t\theta_t)a_t + \t(a_t\theta_t)q + \t(q\theta_t)a_t\right]\,dt. \end{equation}
\end{lemma}

\begin{proof}  Since $z_t$ is a stochastic integral, the It\^o rules regarding product with $dt$ apply as usual.  For Equation \ref{eq dzt 1}, we simply expand
\begin{align*} dz_t\,\theta_t\,dz_t &= (q\,dy_t\,p_tq - qp_t\,dy_t\,q)\theta_t(q\,dy_t\,p_tq - qp_t\,dy_t\,q) \\
&= q\,dy_t\,(p_tq\theta_tq)\,dy_t\,p_tq - q\,dy_t\,(p_tq\theta_tqp_t)\,dy_t\,q - qp_t\,dy_t\,(q\theta_tq)\,dy_t\,p_tq + qp_t\,dy_t\,(q\theta_tqp_t)\,dy_t\,q.
\end{align*}
Applying Equation \ref{eq Ito calc y 1} to each of these four terms yields
\[ -q\cdot\t(p_tq\theta_tq)\cdot p_tq\,dt + q\cdot\t(p_tq\theta_tqp_t)\cdot q\,dt + qp_t\cdot\t(q\theta_tq)\cdot p_tq\,dt - qp_t\cdot\t(q\theta_tqp_t)\cdot q\,dt. \]
Using the trace property, and simplifying with the relations $q^2=q$, $p_t^2=p_t$, and $qp_tq = a_t$, yields
\[ dz_t\,\theta_t\,dz_t = -\t(a_t\theta_t)a_t\,dt + \t(a_t\theta_t)q\,dt + \t(q\theta_t)a_t\,dt - \t(a_t\theta_t)a_t\,dt, \]
which simplifies to give Equation \ref{eq dzt 1}.
\end{proof}

Refer now to the summands in the second term in Equation \ref{eq trace dan 1}.  From Equations \ref{eq da 1} and \ref{eq dzt}, we have
\[ da_t = (\alpha q-a_t)\,dt + dz_t. \]
Hence, for $j\in\{1,\ldots,n-1\}$,
\[ a_t^{n-j-1}\,da_t\,a_t^{j-1}\,da_t = a_t^{n-j-1}\,[(\alpha q-a_t)\,dt + dz_t]\, a_t^{j-1}\, [(\alpha q-a_t)\,dt + dz_t]. \]
We may expand this into four terms.  However, since $dt$ commutes with everything and, by Lemma \ref{lemma zt} $dt\,dz_t = 0$, and as always $(dt)^2 = 0$, the only surviving term is
\[ a_t^{n-j-1}\,da_t\,a_t^{j-1}\,da_t = a_t^{n-j-1}\,dz_t\,a_t^{j-1}\,dz_t. \]
Employing Equation \ref{eq dzt 1} of Lemma \ref{lemma zt}, we therefore have
\[ a_t^{n-j-1}\,da_t\,a_t^{j-1}\,da_t = a_t^{n-j-1}[-2\t(a_t^j)a_t + \t(a_t^j)q + \t(qa_t^{j-1})a_t]\,dt. \]
Taking the trace, we have
\begin{equation} \label{eq trace dan intermediate}
\t(a_t^{n-j-1}\,da_t\,a_t^{j-1}\,da_t) = \left[-2\t(a_t^j)\t(a_t^{n-j}) + \t(a_t^j)\t(a_t^{n-j-1}q) + \t(qa_t^{j-1})\t(a_t^{n-j})\right]\,dt.
\end{equation}
Provided $j-1\ne 0$ and $n-j-1\ne 0$ (i.e. $n\ge 4$ and $j\in\{2,\ldots,n-2\}$), $qa_t^{j-1} = a_t^{j-1}$ and $qa_t^{n-j-1} = a_t^{n-j-1}$, as per the discussion following Equation \ref{eq trace dan 1}.  So, in this regime, we have
\begin{equation} \label{eq trace dan intermediate 1}
\t(a_t^{n-j-1}\,da_t\,a_t^{j-1}\,da_t) = \left[-2\t(a_t^j)\t(a_t^{n-j}) + \t(a_t^j)\t(a_t^{n-j-1}) + \t(a_t^{j-1})\t(a_t^{n-j})\right]\,dt, \quad 2\le j\le n-2.
\end{equation}
The case $j=1$ corresponds to $\t(a_t^{n-2}\,da_t\,da_t)$, while $j=n-1$ corresponds to $\t(da_t\,a^{n-2}\,da_t$).  By the trace property, these are equal.  In each case, one of the $q$-terms is $\t(a_t^{n-2}q) = \t(a_t^{n-2})$, while the other is $\t(a_t^0q) = \t(q) = \b$.  So, we have
\begin{equation} \label{eq trace dan intermediate 2}
\t(a_t^{n-j-1}\,da_t\,a_t^{j-1}\,da_t) = \left[-2\t(a_t)\t(a_t^{n-1}) + \t(a_t)\t(a_t^{n-2}) + \beta\t(a_t^{n-1})\right]\,dt, \quad j\in\{1,n-1\}.
\end{equation}
Hence, the second term in Equation \ref{eq trace dan 1} is equal to
\begin{equation} \label{eq trace dan 1.2}
\begin{aligned}  &n[-2\t(a_t)\t(a_t^{n-1}) + \t(a_t)\t(a_t^{n-2}) + \beta\t(a_t^{n-1})]\,dt \\
+& \frac12 n \sum_{j=2}^{n-2}\, [-2\t(a_t^j)\t(a_t^{n-j}) + \t(a_t^j)\t(a_t^{n-j-1}) + \t(a_t^{j-1})\t(a_t^{n-j})]\,dt
\end{aligned}
\end{equation}
where the summation is taken to be $0$ in the case $n\le 3$.  Combining Equations \ref{eq trace dan 1.1} and \ref{eq trace dan 1.2} with Equation \ref{eq trace dan 1}, and using the notation $g_n(t) = \t(a_t^n)$, we therefore have
\begin{equation} \label{eq almost there 1}
\begin{aligned} dg_n = \t[d(a^n)] &= [n\a\,g_{n-1} - n\,g_n - 2n\, g_1g_{n-1} + n\,g_1g_{n-2} + n\beta\,g_{n-1}]\,dt \\
& \quad + \frac12 n\sum_{j=2}^{n-2} \,[-2g_jg_{n-j} + g_jg_{n-j-1} + g_{j-1}g_{n-j}]\,dt.
\end{aligned}
\end{equation}
From here we see that $g_n$ is actually differentiable, and that Equation \ref{eq almost there 1} is a differential equation.  One more simplification is in order.  By making the change of index $k = n-j$, note that
\[ \sum_{j=2}^{n-2} g_j g_{n-j-1} = \sum_{k=2}^{n-2} g_{n-k} g_{k-1}. \]
Hence, the second and third summands in the second line of Equation \ref{eq almost there 1} have the same sum, and we have
\begin{equation} \label{eq almost there 2} g_n' = -n\,g_n + n(\alpha+\beta)g_{n-1} -2n\, g_1g_{n-1} + n\,g_1g_{n-2} + \frac12n\sum_{j=2}^{n-2}[-2g_jg_{n-j} + 2g_{j-1}g_{n-j}]. \end{equation}
We can then recombine the third and fourth terms (just before the summation) as follows:
\begin{align*} &-2n\, g_1g_{n-1} + n\,g_1g_{n-2} + \frac12n\sum_{j=2}^{n-2}[-2g_jg_{n-j} + 2g_{j-1}g_{n-j}] \\
=& \left(-ng_1g_{n-1} - n\sum_{j=2}^{n-2} g_jg_{n-j} - ng_{n-1}g_1\right) + \left(ng_1g_{n-2} + n\sum_{j=2}^{n-2} g_{j-1}g_{n-j}\right) \\
=& -n\sum_{j=1}^{n-1} g_jg_{n-j} + n\sum_{j=2}^{n-1} g_{j-1}g_{n-j}.
\end{align*}
Combining this with Equation \ref{eq almost there 2}, we are led to the following result.

\begin{proposition} \label{prop ODEs} Let $\{g_n\,:\,n\ge 1\}$ be defined as in Equation \ref{eq fn}.  Then $g_n$ is continuous on $[0,\infty)$ and differentiable on $(0,\infty)$ for each $n$.  Furthermore, the functions $g_n$ satisfy the following infinite system of ordinary differential equations.
\begin{align}
\label{eq system 1} g_1' &= -g_1 + \alpha\beta \\
\label{eq system 2} g_2' &= -2g_2 + 2(\alpha+\beta)g_1 - 2g_1^2 \\
\label{eq system 3} g_n' &= -ng_n + n(\alpha+\beta)g_{n-1} -n\sum_{j=1}^{n-1} g_jg_{n-j} + n\sum_{j=2}^{n-1} g_{j-1}g_{n-j}, \quad n\ge 3.
\end{align}
\end{proposition}

\noindent Note that similar equations appear in \cite{Demni 1} (see also \cite{Hiai Petz, Voiculescu-free-analogs-VI}).  It will be convenient to define $g_0\equiv \alpha+\beta$.  With this convention, Equations \ref{eq system 2} and \ref{eq system 3} can be written in the more compact form
\begin{equation} \label{eq system 3'} g_n' = -n\left[g_n-\sum_{j=1}^{n-1}(g_j-g_{j-1})g_{n-j}\right], \quad n\ge 2. \end{equation}

\subsection{The Proof of Proposition \ref{prop flawed derivative}} \label{section flawed derivative}

Our goal here is to calculate the right-derivative of $F_t(s) = \t[(r_t-r_tp_sr_t)^2]$ at $s=t$.  The process $s\mapsto (r_t-r_tp_sr_t)^2$ is adapted for $s\in [t,\infty)$, and so we may use the tools of It\^o calculus to compute this derivative.  To begin, we expand
\[ (r_t-r_tp_sr_t)^2 = r_t^2 - r_t^2p_sr_t - r_tp_sr_t^2 + (r_tp_sr_t)^2 = r_t - 2r_tp_sr_t + r_tp_sr_tp_sr_t. \]
Recalling that $t$ is constant, the stochastic differential is
\[ -2r_t dp_s r_t + d(r_tp_sr_tp_sr_t). \]
Using the It\^o product rule \ref{eq Ito product rule}, the second term expands to
\begin{align*} d[(r_tp_s r_t)(p_sr_t)] &= (r_tdp_sr_t)p_sr_t + r_tp_sr_t(dp_s r_t) + d(r_tp_sr_t)\cdot d(p_sr_t) \\
&= r_tdp_sr_tp_sr_t + r_tp_sr_tdp_sr_t + r_tdp_sr_tdp_sr_t.
\end{align*}
Combining and taking the trace, this yields
\[ dF_t(s) = \t[d(r_t-r_tp_sr_t)^2] = \t[-2r_tdp_sr_t + r_tdp_sr_tp_sr_t + r_tp_sr_tdp_sr_t + r_tdp_sr_tdp_sr_t]. \]
Using the trace property (and the fact that $r_t=r_t^2$), this simplifies to three terms:
\begin{equation} \label{eq df_t 1} dF_t(s) = -2\t(r_tdp_s)+2\t(r_tp_sr_tdp_s) + \t(r_tdp_sr_tdp_s). \end{equation}
From Equation \ref{eq dpt 2} above, we have $dp_s = (\t(p)-p_s)\,dt + dy_s\,p_s - p_s\,dy_s$ where the process $y_s$ is determined by fSDE \ref{eq dyt}, and obeys the It\^o calculus of Equations \ref{eq Ito calc y 1} and \ref{eq Ito calc y 2}.

We consider now the three terms in Equation \ref{eq df_t 1} separately.  First:
\[ \t(r_tdp_s) = [-\t(r_tp_s)+\t(p)r_t]\,ds +i \t(r_t\, dy_s\, p_s) - i\t(r_t\, p_s\, dy_s). \]
The last two terms are $0$: each can be expressed in the form $\t(\theta_s\, dx_s)$ where $\theta_s$ is adapted (since $s\ge t$), and It\^o integrals are centered.  Thus
\begin{equation} \label{eq df_t 2.1} \t(r_tdp_s) = [-\t(r_tp_s) + \t(p)\t(r_t)]\,ds. \end{equation}

\noindent Now for the second term in Equation \ref{eq df_t 1}.  Combining with Equation \ref{eq dpt 2}, we have
\[ \t(r_tp_sr_tdp_s) = [-\t(r_tp_sr_tp_s)+\t(r_tp_sr_t)\t(p)]\,ds + i\tau(r_tp_sr_t\,dy_s\,p_s) -i\t(r_tp_sr_t p_s\,dy_s). \]
Again, the last two terms are of the form $\t(\theta_s\,dx_s)$ for adapted $\theta_s$, and so we have
\begin{equation} \label{eq df_t 2.2} \t(r_tp_sr_tdp_s) = [-\t((r_tp_s)^2) + \t(r_tp_s)]\,ds. \end{equation}

\noindent Finally, we come to the third term in Equation \ref{eq df_t 1}.  We must calculate the trace of
\[ (r_tdp_s)^2 = ([-r_tp_s+\t(p_s)r_t]\,ds + r_t\,dy_s\,p_s - r_tp_s\,dy_s)^2. \]
All products involving the $ds$ term vanish, since $ds^2 = ds dy_s = 0$ (cf.\ Equation \ref{eq Ito calc y 2}).  Thus, we simply have
\begin{align*} (r_tdp_s)^2 &= ( r_t\,dy_s\,p_s - r_tp_s\,dy_s)^2 \\
&= (r_t\,dy_s\,p_s - r_tp_s\,dy_s)^2 = r_t dy_s p_sr_t dy_s p_s - r_t dy_s p_s r_t p_s dy_s - r_tp_s dy_s r_t dy_s p_s + r_t p_s dy_s r_t p_s dy_s . \end{align*}
Now using Equation \ref{eq Ito calc y 1}, these terms simplify as
\begin{itemize}
\item $r_t dy_s p_sr_t dy_s p_s = -\t(p_sr_t)r_tp_s\,ds$
\item $r_tdy_s p_sr_tp_sdy_s = -\t(p_sr_tp_s)r_t\,ds = -\t(p_sr_t)r_t\,ds$
\item $r_tp_sdy_sr_tdy_sp_s = -r_tp_s\t(r_t)p_s\,ds = -\t(r_t)r_tp_s\,ds$
\item $r_tp_sdy_s r_tp_sy_s = -\t(r_tp_s)r_tp_s\,ds$
\end{itemize}
Summing and taking the trace we have
\begin{equation} \label{eq df_t 2.3} \t[(r_tdp_s)^2] = [-2(\t(p_sr_t))^2+2\t(r_t)\t(p_sr_t)]\,ds. \end{equation}
Combining Equations \ref{eq df_t 2.1}, \ref{eq df_t 2.2}, and \ref{eq df_t 2.3} with Equation \ref{eq df_t 1}, we have
\begin{align*} \frac{dF_t}{ds}(s) &= -2[-\t(r_tp_s)+\t(p)\t(r_t)]+2[-\t((r_tp_s)^2)+\t(r_tp_s)]+[-2(\t(r_tp_s))^2+2\t(r_t)\t(r_tp_s)] \\
&= 2\left(2\t(r_tp_s) + \t(r_t)[\t(r_tp_s)-\t(p)] - [\t((rp_s)^2)+ (\t(rp_s))^2]\right).
\end{align*}
Evaluating at $s=t$, we get (using $r_t = r_t^2 = r_tp_t$)
\begin{align*} \frac12\frac{dF_t}{ds}(t) &= 2\t(r_tp_t) + \t(r_t)[\t(r_tp_t)-\t(p)]-[\t((rp_t)^2)+ (\t(rp_t))^2] \\
&= 2\t(r_t) + \t(r_t)[\t(r_t)-\t(p)] - [\t(r_t)+\t(r_t)^2] \\
&= \t(r_t) - \t(r_t)\t(p) = \t(r_t)(1-\t(p))
\end{align*}
as claimed in the corollary. \hfill $\square$

\subsection{Smoothness of the Moment Generating Function in $t$} Our next goal is to prove that the (centered) moment-generating function of $\mu_t$
\begin{equation} \label{eq moment gen fn} \psi(t,w) = \sum_{n\ge 1} g_n(t)w^n \end{equation}
is $C^\infty$ jointly in $(t,w)$ for $t>0$ and $|w|<1$.  The moments $g_n(t)$ are solutions to Equations \ref{eq system 1}--\ref{eq system 3}, which can, in principle, be solved explicitly.  Presently, we only use the fact that the solution has a simple analytic form.

\begin{lemma} \label{lemma poly exp} For $n\ge 1$, the function $g_n(t)$ of Proposition \ref{prop ODEs} is a polynomial in $t$ and $e^{-t}$. \end{lemma}

\begin{proof} To begin, we may solve Equation \ref{eq system 1} explicitly: $g_1(t) = g_1(0)e^{-t} + \alpha\beta(1-e^{-t})$, having the desired form.  We proceed by induction on $n$. Equations \ref{eq system 2} and \ref{eq system 3} give, for $n\ge 2$, $g_n' + ng_n = h_n$ where $h_n$ is a polynomial in $g_1,\ldots,g_{n-1}$, and is therefore a polynomial in $t$ and $e^{-t}$ by the induction hypothesis.  The ODE can then be written in the form $\frac{d}{dt}[e^{nt}g_n(t)] = e^{nt}h_n(t)$, with solution
\[ g_n(t) = e^{-nt}\left[\int_0^t e^{ns}h_n(s)\,ds+g_n(0)\right]. \]
Since $h_n(s)$ is a polynomial in $s$ and $e^{-s}$, $e^{ns}h_n(s)$ is a polynomial in $s$ and $e^{\pm s}$ whose positive degree in $e^{s}$ is $\le n$.  A separate induction argument and elementary calculus show that the antiderivative of $e^{ns}h_n(s)$ is therefore also a polynomial in $s$ and $e^{\pm s}$ whose positive degree in $e^s$ is $\le n$.  Thus $g_n(t)$ has the desired form, proving the corollary.
\end{proof}

We can also iterate the ODEs \ref{eq system 1}--\ref{eq system 3} to find a general recurrence form for the $k$th derivatives.

\begin{lemma} \label{lemma kth derivative recurrence} Let $\{g_n\}_{n\ge 1}$ be defined as in Equation \ref{eq fn}, with $g_0\equiv\alpha+\beta$ as in Equation \ref{eq system 3'}.  Then for all $t\ge 0$ and $n,k\ge 1$, the $k$th derivative $g_n^{(k)}$ satisfies
\begin{equation} \label{eq kth derivative recurrence} g_n^{(k)}= \sum_{s=0}^{k+1}\sum_{1\le j_1,\ldots,j_s\le n} c_s^{n,k}(j_1,\ldots,j_s) g_{j_1}\cdots g_{j_s} \end{equation}
for some constants $c_s^{n,k}(j_1,\ldots,j_s)$ satisfying $|c_s^{n,k}(j_1,\ldots,j_s)| \le (4sn^2)^k$.
\end{lemma}

\begin{proof} When $n=0$, $g_0^{(k)}=0$ for $k\ge 1$.  When $n=1$, iterating Equation \ref{eq system 1} shows that $g_1^{(k)} = (-1)^k[g_1-\alpha\beta]$, which has the desired form of Equation \ref{eq kth derivative recurrence} with $c^{1,k}_0=(-1)^{k+1}\alpha\beta$, $c^{1,k}_1(1) = (-1)^k$, and $c^{1,k}_s(j_1,\ldots,j_s)=0$ for $s\ge 2$.  For $n\ge 2$, we proceed by induction on $k$.  For the base case $k=1$, we note that, by Equation \ref{eq system 3'},
\begin{equation} \label{e.ind1} g_n' = -n\left[g_n-\sum_{j=1}^{n-1}(g_j-g_{j-1})g_{n-j}\right] = \sum_{s=0}^2 \sum_{1\le j_1,\ldots,j_s\le n} c_s^1(j_1,\ldots,j_s)g_{j_1}\cdots g_{j_s} \end{equation}
with $c^{n,1}_0=0$, $c^{n,1}_1(j) = -n\delta_{j,n}$, and
\[ c^{n,1}_2(j_1,j_2) = \begin{cases} -n, & j_1+j_2=n \\ n & j_1+j_2=n-1 \\ 0 & \text{otherwise}. \end{cases} \]
Now for the inductive step.  Assume that Equation \ref{eq kth derivative recurrence} holds up to level $k-1$.  Then
\begin{align} \nonumber g_n^{(k)} = \frac{d}{dt}g_n^{(k-1)} &= \sum_{s=1}^k \sum_{1\le j_1,\ldots,j_s\le n} c_s^{n,k-1}(j_1,\ldots,j_s) \frac{d}{dt}\left(g_{j_1}\cdots g_{j_s}\right) \\ \label{e.ind2}
&= \sum_{s=1}^k \sum_{\ell=1}^s \sum_{1\le j_1',\ldots,j_s'\le n} c_s^{n,k-1}(\ell;j_1',\ldots,j_s')g_{j_1'}'\cdot g_{j_2'}\cdots g_{j_s'},
\end{align}
where we have reindexed $(j_1',\ldots,j_s') = (j_\ell,\ldots,j_s,j_1,\ldots,j_{\ell-1})$, and the new constants $c_s^{n,k-1}(\ell;j_1',\ldots,j_s')$ are reordered accordingly: $c_s^{n,k-1}(\ell;j_1',\ldots,j_s') = c_s^{n,k-1}(j_{s-\ell+2}',\ldots,j_s',j_1',\ldots,j_{s-\ell+1}')$ (with $j_1'$ in the $\ell$th slot).  We now relabel $j_r'\mapsto j_r$, and do the internal sum over $j_1$ first:
\[ g_{j_1}' = c^{j_1,1}_0 + \sum_{i=1}^{j_1} c^{j_1,1}_1(i)g_i + \sum_{1\le i_1,i_2\le j_1} c^{j_1,1}_2(i_1,i_2)g_{i_1}g_{i_2} \]
which yields terms of orders $s-1$, $s$, and $s+1$ in the internal sum in Equation \ref{e.ind2}.
\begin{align} \label{e.s-1} \sum_{1\le j_1,\ldots,j_s\le n} c_s^{n,k-1}(\ell;j_1',\ldots,j_s') &= \sum_{1\le j_1,\ldots,j_s\le n}   c_s^{n,k-1}(\ell;j_1,\ldots,j_s)c_0^{j_1,1}\cdot g_{j_2}\cdots g_{j_s} \\
\label{e.s}&+ \sum_{1\le j_1,\ldots,j_s\le n} c_s^{k-1}(\ell;j_1,\ldots,j_s) \sum_{i=1}^{j_1} c_1^{j_1,1}(i) g_ig_{j_2}\cdots g_{j_s} \\
\label{e.s+1} &+ \sum_{1\le j_1,\ldots,j_s\le n}  c_s^{k-1}(\ell;j_1,\ldots,j_s) \sum_{1\le i_1,i_2\le j_1} c_2^{j_1,1}(i_1,i_2)g_{i_1}g_{i_2}g_{j_2}\cdots g_{j_s}.
\end{align}
Reindexing (\ref{e.s-1}) and summing over $\ell$ gives
\begin{equation} \label{e.s-1'} \sum_{1\le j_1,\ldots,j_{s-1}\le n} \left(\sum_{\ell=1}^s\sum_{i=1}^n c_s^{n,k-1}(\ell;i,j_1,\ldots,j_{s-1})c_0^{i,1}\right) g_{j_1}\cdots g_{j_{s-1}}  \equiv \sum_{1\le j_1,\ldots,j_{s-1}\le n} d^{n,k}_{s,-}(j_1,\ldots,j_{s-1})g_{j_1}\cdots g_{j_{s-1}}, \end{equation}
where in the case $s=1$ this is just a constant.  Changing the order of summation in (\ref{e.s}) and summing over $\ell$ yields
\begin{equation*}  \sum_{1\le j_2,\ldots,j_s,i\le n} \sum_{\ell=1}^s \sum_{j_1=i}^n c_s^{k-1}(\ell;j_1,\ldots,j_s)c_1^{j_1,1}(i) g_ig_{j_2}\cdots g_{j_s},\end{equation*}
or, for convenience exchanging $i\leftrightarrow j_1$,
\begin{equation} \label{e.s'}  \sum_{1\le j_1,\ldots,j_s\le n} \left(\sum_{\ell=1}^s \sum_{i=j_1}^n c_s^{k-1}(\ell;i,j_2,\ldots,j_s)c_1^{i,1}(j_1)\right) g_{j_1}\cdots g_{j_s} \equiv \sum_{1\le j_1,\ldots,j_s\le n} d^{n,k}_{s,0}(j_1,\ldots,j_s)g_{j_1}\cdots g_{j_s}. \end{equation}
Finally, in (\ref{e.s+1}), we change the order of summation between $j_1$ and $\{i_1,i_2\}$,
\[ \sum_{j_1=1}^n \sum_{1\le i_1,i_2\le n} = \sum_{1\le i_1,i_2\le n} \left( \1_{i_2\le i_1}\sum_{j_1=i_1}^n + \1_{i_2>i_1} \sum_{j_1=i_2}^n\right) \]
which yields (\ref{e.s+1}) in the form
\[ \sum_{1\le j_2,\ldots,j_s,i_1,i_2\le n} \sum_{\ell=1}^s \left(\1_{i_2\le i_1} \sum_{j_1=i_1}^n + \1_{i_2>i_1}\sum_{j_1=i_2}^n\right)c_s^{k-1}(\ell;j_1,\ldots,j_s)c_2^{j_1,1}(i_1,i_2) g_{i_1}g_{i_2}g_{j_2}\cdots g_{j_s}, \]
which, after reindexing, gives
\begin{equation} \label{e.s+1'} \sum_{1\le j_1,\ldots,j_{s+1}\le n} d^{n,k}_{s,+}(j_1,\ldots,j_{s+1}) g_{j_1}\cdots g_{j_{s+1}} \end{equation}
where
\begin{equation} \label{e.s+1''} d^{n,k}_{s,+}(j_1,\ldots,j_{s+1}) = \sum_{\ell=1}^s  \left(\1_{j_2\le j_1} \sum_{i=j_1}^n + \1_{j_2>j_1}\sum_{i=j_2}^n\right)c_s^{k-1}(\ell;i,j_3,\ldots,j_s)c_2^{i,1}(j_1,j_2).  \end{equation}

Now, for any $s$-tuple $\mx{j}_s = (j_1,\ldots,j_s)$ in $[n]^s$ where $[n]=\{1,\ldots,n\}$, let $g_{\mx{j}_s}=g_{j_1}\cdots g_{j_s}$.  With this notation, combining (\ref{e.ind2}) with (\ref{e.s-1}--\ref{e.s+1''}), we have
\begin{align*} g^{(k)}_n &= \sum_{s=1}^k\left( \sum_{\mx{j}_{s-1}\in[n]^{s-1}} d^{n,k}_{s,-}(\mx{j}_{s-1}) g_{\mx{j}_{s-1}} +  \sum_{\mx{j}_{s}\in[n]^{s}} d^{n,k}_{s,0}(\mx{j}_{s}) g_{\mx{j}_{s}} + \sum_{\mx{j}_{s+1}\in[n]^{s+1}} d^{n,k}_{s,+}(\mx{j}_{s+1}) g_{\mx{j}_{s+1}} \right) \\
&= \sum_{s=1}^{k+1} \sum_{\mx{j}_s\in[n]^s} \Big(d^{n,k}_{s-1,+}(\mx{j}_s)+d^{n,k}_{s,0}(\mx{j}_s)+d^{n,k}_{s+1,-}(\mx{j}_s)\Big)g_{\mx{j}_s},
\end{align*}
with the convention that $d^{n,k}_{s,\varepsilon}=0$ for $s\notin[k]$ and $\varepsilon\in\{-,0,+\}$.  Hence, if we define
\[ c^{n,k}_s(\mx{j}_s) \equiv d^{n,k}_{s-1,+}(\mx{j}_s) + d^{n,k}_{s,0}(\mx{j}_s) + d^{n,k}_{s+1,-}(\mx{j}_s), \qquad 1\le s\le k+1, \]
we see that $g_n^{(k)}$ has the desired form (\ref{eq kth derivative recurrence}), and it remains only to show that the bound $|c^{n,k}_s(\mx{j}_s)|\le (4sn^2)^k$ is satisfied.

Let $C_s^{n,k-1} = \sup_{\mx{j}_s\in[n]^s} |c_s^{n,k-1}(\mx{j}_s)|$.  From (\ref{e.s-1}) we have
\[ \left|d^{n,k}_{s,-}(\mx{j}_{s-1})\right| = \left|\sum_{\ell=1}^s \sum_{i=1}^n c_s^{n,k-1}(\ell;i,\mx{j}_{s-1})c^{i,1}_0\right| \le sC^{n,k-1}_s\sum_{i=1}^n C_0^{i,1} \le sC_s^{n,k-1}, \]
since $C_0^{i,1} = \alpha\beta\le 1$ if $i=1$ and is $0$ if $i>1$.  From (\ref{e.s}) we have
\[ |d^{n,k}_{s,0}(\mx{j}_s)| =  \left|\sum_{\ell=1}^s \sum_{i=j_1}^n c_s^{k-1}(\ell;i,j_2,\ldots,j_s)c_1^{i,1}(j_1)\right| \le sC_s^{n,k-1}\sum_{i=1}^n C^{i,1}_1 \le sn^2 C_s^{n,k-1},  \]
since $C^{i,1}_1 \le i \le n$ for all $i$.  Finally, from (\ref{e.s+1''}), we have
\[ |d^{n,k}_{s,+}(\mx{j}_{s+1})| = \left| \sum_{\ell=1}^s  \left(\1_{j_2\le j_1} \sum_{i=j_1}^n + \1_{j_2>j_1}\sum_{i=j_2}^n\right)c_s^{k-1}(\ell;i,j_3,\ldots,j_s)c_2^{i,1}(j_1,j_2)\right| \le 2sC^{n,k-1}_s\sum_{i=1}^n C^{i,1}_2 \le 2sn^2 C_s^{n,k-1}. \]
Thus,
\[ C^{n,k}_s = \sup_{\mx{j}_s\in[n]^s }|c_s^{n,k}(\mx{j}_s)|\le s(3n^2+1)C_s^{n,k-1} \le 4sn^2 C_s^{n,k-1}. \]
The result now follows from the inductive hypothesis that $C_s^{n,k-1} \le (4sn^2)^{k-1}$.
\end{proof}

Next, we use the equations to prove the following family of (blunt) growth estimates for the derivatives of the moments $g_n(t)$.

\begin{lemma} \label{lemma growth estimates for gn} Let $\{g_n\}_{n\ge 1}$ be defined as in Equation \ref{eq fn}.  Then for each $k\in\N$, the $k$th derivative $g_n^{(k)}$ is uniformly bounded by
\[ |g_n^{(k)}| \le (k+1)^{k}(2n)^{3k+1}. \]
\end{lemma}

\begin{proof} For all $n\ge 1$, $g_n(t)$ is the $n$th moment of a probability measure $\mu_t$ supported in $[0,1]$; hence $|g_n(t)|\le 1$.  By the convention used in Lemma, \ref{lemma kth derivative recurrence} $g_0\equiv\alpha+\beta$, which is in $[0,2]$.  Hence, in general $|g_n(t)|\le 2$ for all $n\ge 0$.  Thus, referring to the recursive form of the derivative given in Equation \ref{eq kth derivative recurrence} in the previous lemma, we have
\[ |g_n^{(k)}| \le \sum_{s=0}^{k+1}\sum_{1\le j_1,\ldots,j_s\le n} |c_s^{n,k}(j_1,\ldots,j_s)| |g_{j_1}|\cdots |g_{j_s}|  \le \sum_{s=0}^{k+1} \sum_{1\le j_1,\ldots,j_s\le n} (4sn^2)^k\cdot 2^s = \sum_{s=0}^{k+1}(2n)^s(4sn^2)^k, \]
and therefore
\[ |g_n^{(k)}| \le \sum_{s=0}^{k+1}(2n)^{k+1}(4(k+1)n^2)^k = (k+1)^{k}(2n)^{3k+1}, \]
concluding the proof.
\end{proof}

\begin{remark} A more careful estimate is possible, showing that $|g_n^{(k)}| = O(n^{2k})$ for each $k$; for example, direct estimation of Equation \ref{eq system 3} gives
\begin{align*} | g_n'| &\le n\left(|g_n| + (\a+\b)|g_{n-1}| + \sum_{j=1}^{n-1}|g_j||g_{n-j}| + \sum_{j=2}^{n-1}|g_{j-1}||g_{n-j}|\right) \\
&= n\left(1+\a+\b + (n-1)+(n-2)\right) \le 2n^2,
\end{align*}
much smaller than the $64n^4$ bound proved in above.  The improvement this bound represents over the bound proven in Lemma \ref{lemma growth estimates for gn} is of no consequence to our application, however. \end{remark}

We now prove the main theorem of this section: that the (centered) moment-generating function $\psi(t,w) = \sum_{n\ge 1} g_n(t)w^n$ is $C^\infty$ in both variables.

\begin{proposition} \label{prop psi smooth} The function $\psi(t,w)$ of Equation \ref{eq moment gen fn} is $C^\infty$ jointly in both variables, for $t>0$ and $|w|<1$.
\end{proposition}

\begin{proof} Since the coefficients $g_n(t)$ are uniformly bounded in modulus by $1$, the power series $\psi(t,\cdot)$ converges uniformly on compact subsets of $\mathbb{D}$, and analyticity is an elementary result from complex variables.  Now, for $k\ge 1$ and $m\ge 0$ define
\[ \varphi_{k,m}(t,w) = \sum_{n=m}^\infty g_n^{(k)}(t)\cdot  n(n-1)\cdots(n-m+1) w^{n-m}. \]
Throughout this proof (only), we use the non-standard convention that $0^0=0$; thus $\varphi_{0,0} = \psi$.  For fixed $t\ge 0$, by Lemma \ref{lemma growth estimates for gn}, the coefficients of this power series in $w$ are bounded by
\[ |g_n^{(k)}(t)\cdot n(n-1)\cdots (n-m+1)| \le (k+1)^{k}(2n)^{3k+1}\cdot n^m \]
and so by the root test, the power-series converges uniformly on compact subset of the unit disk $\mathbb{D}$, defining a function analytic in $w$.  Similarly, for fixed $|w|<1$, $\varphi_{k,m}(t,w)$ is absolutely summable; it is a series of functions that are continuous (by Lemma \ref{lemma poly exp}) and hence, by the Weierstra\ss\ M-test, $\varphi_{k,m}(t,w)$ is continuous in $t$.  This shows that $\varphi_{k,m}\in C^0(\R_+\times\mathbb{D})$.

The uniform convergence of $\varphi_{k,m}$ and pointwise convergence of $\varphi_{k-1,m}$ (at any single point in $\mathbb{D}$) imply (for example by \cite[Thm.\ 7.17]{Rudin}) that $\varphi_{k,m}(\cdot,w)$ is differentiable for each $w$ and that
\[ \frac{\del}{\del t}\varphi_{k-1,m}(t,w) = \varphi_{k,m}(t,w). \]
By induction on $k$, this shows that $\varphi_{0,m}(t,w)$ is at least $C^k$ in $t$ for each $k$, and has time derivatives $\frac{\del^k}{\del t^k}\varphi_{0,m}(t,w) = \varphi_{k,m}(t,w)$.  Of course, by elementary complex variables,
\[ \varphi_{k,m}(t,w) = \frac{\del^m}{\del w^m}\sum_{n=1}^\infty g_n^{(k)}(t)w^n =\frac{\del^m}{\del w^m}\varphi_{k,0}(t,w); \]
thus, we have shown that
\[ \varphi_{k,m}(t,w) = \frac{\del^m}{\del w^m}\frac{\del^k}{\del t^k}\varphi_{0,0}(t,w) = \frac{\del^m}{\del w^m}\frac{\del^k}{\del t^k}\psi(t,w) \]
for $m,k\ge 0$.  As we have shown that $\varphi_{k,m}\in C^0(\R_+\times\mathbb{D})$, this proves that $\psi$ is $C^\infty(\R_+\times\mathbb{D})$ as required.  \end{proof}

Our next goal is to extend this result to analyticity of $\psi$.  Analyticity in the spacial variable $w\in\mathbb{D}$ follows immediately from the proof of Proposition \ref{prop psi smooth}. Analyticity in $t$ is much more involved.  It will pay to first translate the ODEs of Equations \ref{eq system 1}--\ref{eq system 3} into a PDE for the Cauchy transform of $\mu_t$ (which is simply related to $\psi(t,\cdot)$), and then use analytic function techniques in the spacial variable; this is the purpose of the next two sections.

\begin{remark} \label{remark bad bounds} The bound of Lemma \ref{lemma growth estimates for gn} are woefully inadequate to prove  analyticity in $t$ using Taylor's theorem, since the coefficients grow superfactorially in $k$.  In Section \ref{section G analytic t}, we will use more sophisticated techniques to prove that these Taylor series coefficients are, in fact, exponentially bounded.
\end{remark}

\subsection{The Flow of the Cauchy Transform}
 
The moment function $\psi(t,w)$ of Equation \ref{eq moment gen fn} is closely related to the Cauchy transform of the measure $\mu_t$, which concerns us in this paper. Indeed, the coefficients $g_n(t)$ are the moments of $\mu_t$.     Since $\psi(t,w)$ converges uniformly for $w\in\mathbb{D}$, taking $z=1/w$ for $w\ne 0$, we see that $\psi(t,1/z)$ converges uniformly for $|z|>1$, and hence by the Fubini-Tonelli theorem,
\[ \int_0^1 \sum_{n\ge 1} \left(\frac{x}{z}\right)^n\,\mu_t(dx) = \sum_{n\ge 1} \frac{1}{z^n} \int_0^1 x^n\,\mu_t(dx) = \sum_{n\ge 1} g_n(t)\frac{1}{z^n} = \psi(t,1/z), \quad |z|>1. \]
On the other hand, since the support of $\mu_t$ is contained in $[0,1]$, for $|z|>1$ the series $\sum_{n\ge 1} (x/z)^n$ converges, and we have
\[ \frac{1}{z}\psi(t,1/z) = \frac{1}{z}\int_0^1 \left(\frac{1}{1-x/z}-1\right)\,\mu_t(dx) = \int_0^1 \frac{\mu_t(dx)}{z-x} - \frac1z = G_{\mu_t}(z)-\frac1z, \quad |z|>1. \]
Finally, this means that the function $G$ of Equation \ref{eq G} is related to $\psi$ by
\begin{equation} \label{eq G phi} G(t,z) = G_{\mu_t}(z) -\frac1z + \frac{\min\{\a,\b\}}{z} = \frac{1}{z}\left(\psi(t,1/z)+\min\{\a,\b\}\right).  \end{equation}
From Equation \ref{eq G phi} and Proposition \ref{prop psi smooth} it follows immediately that

\begin{corollary} \label{corollary G smooth} The function $G(t,z)$ of Equation \ref{eq G} is $C^\infty(\R_+\times(\C-\overline{\mathbb{D}}))$, and for each $t\ge 0$ $G(t,z)$ is analytic in $|z|>1$.  Moreover, on this domain,
\begin{equation} \label{eq diff G t} \frac{\del}{\del t}G(t,z) = \frac{\del}{\del t}\frac{1}{z}\psi(t,1/z) = \sum_{n\ge 1} \frac{g_n'(t)}{z^{n+1}}. \end{equation}
\end{corollary}
We now combine Equation \ref{eq diff G t} with Equations \ref{eq system 1}--\ref{eq system 3} to deduce a partial differential equation satisfied by $G$.  To begin,
\begin{align*} \frac{\del}{\del t}G = \sum_{n\ge 1}\frac{1}{z^{n+1}}g_n'
&= \frac{1}{z^2}(-g_1 + \alpha\beta) + \frac{1}{z^3}(-2g_2 + 2(\alpha+\beta)g_1 - 2g_1^2) \\
 &+ \sum_{n\ge 3} \frac{1}{z^{n+1}}\left(-ng_n + n(\alpha+\beta)g_{n-1} -n\sum_{j=1}^{n-1} g_jg_{n-j} + n\sum_{j=2}^{n-1} g_{j-1}g_{n-j}\right).
\end{align*}
By the uniform convergence on the domain $|z|>1$, the order of all summations may be interchanged.  It is convenient to recombine the expression into two parts: $\frac{\del}{\del t}G = S_1 + S_2$ where
\[ S_1 = \frac{\a\b}{z^2} - \sum_{n\ge 1} \frac{ng_n}{z^{n+1}} + (\a+\b)\sum_{n\ge 2}\frac{n g_{n-1}}{z^{n+1}}  \]
and
\[ S_2 = -\frac{2g_1^2}{z^3} -\sum_{n\ge 3}\frac{n}{z^{n+1}}\sum_{j=1}^{n-1}g_jg_{n-j} + \sum_{n\ge 3}\frac{n}{z^{n+1}}\sum_{j=2}^{n-1}g_{j-1}g_{n-j}. \]
To deal with the first sum $S_1$, note (from Equation \ref{eq G phi}) that
\[ G(t,z) = \frac{1}{z}(\ff(t,1/z)+\min\{\a,\b\}) = \frac{\min\{\a,\b\}}{z}+\sum_{n\ge 1} \frac{g_n(t)}{z^{n+1}}. \]
For convenience, let us define
\begin{equation} \label{eq G1} G_1(t,z) = \sum_{n\ge 1} \frac{g_n(t)}{z^{n+1}} = G(t,z) - \frac{\min\{\a,\b\}}{z} = G_{\mu_t}(z)-\frac1z. \end{equation}
Then $\frac{\del}{\del t}G_1 = \frac{\del}{\del t}G = S_1+S_2$.  Now,
\begin{equation} \label{eq S1 1} -\sum_{n\ge 1} \frac{n}{z^{n+1}}g_n = \frac{\del}{\del z}\sum_{n\ge 1} \frac{g_n}{z^n} = \frac{\del}{\del z}(zG_1) \end{equation}
which holds true for $|z|>1$.  Similarly,
\[ \sum_{n\ge 2} \frac{ng_{n-1}}{z^{n+1}} = \sum_{n\ge 1} \frac{(n+1)g_n}{z^{n+2}} = \frac1z\left(\sum_{n\ge 1}\frac{g_n}{z^{n+1}} + \sum_{n\ge 1} \frac{n}{z^{n+1}}g_n\right), \]
and employing Equations \ref{eq G1} and \ref{eq S1 1} this becomes
\begin{equation} \label{eq S1 2} \sum_{n\ge 2} \frac{ng_{n-1}}{z^{n+1}} = \frac1z\left(G_1 - \frac{\del}{\del z}(zG_1)\right) = -\frac{\del}{\del z}G_1. \end{equation}
Combining Equations \ref{eq S1 1} and \ref{eq S1 2} and simplifying, this shows that
\begin{equation} \label{eq S1 3} \begin{aligned} S_1 &= \frac{\a\b}{z^2} +\frac{\del}{\del z}(zG_1) - (\a+\b)\frac{\del}{\del z}G_1 \\
&= \frac{\a\b}{z^2} + G_1 + (z-\a-\b)\frac{\del}{\del z}G_1.
\end{aligned} \end{equation}
The terms in $S_2$ can similarly be expressed in terms of $G_1$: in fact, in terms of $G_1^2$.  For $|z|>1$,
\[ (G_1)^2 = \left(\sum_{n\ge 1} \frac{g_n}{z^{n+1}}\right)^2 = \frac{1}{z^2} \sum_{n\ge 2} \frac{1}{z^n} \sum_{j=1}^{n-1} g_jg_{n-j}, \quad |z|>1. \]
Whence, for $|z|>1$,
\begin{equation} \label{eq S2 1} \begin{aligned} \frac{\del}{\del z} (zG_1)^2 = \frac{\del}{\del z}\sum_{n\ge 2} \frac{1}{z^n}\sum_{j=1}^{n-1} g_jg_{n-j}
&= -\sum_{n\ge 2}\frac{n}{z^{n+1}}\sum_{j=1}^{n-1} g_jg_{n-j} \\
&= -\frac{2}{z^3}g_1^2 - \sum_{n\ge 3}\frac{n}{z^{n+1}}\sum_{j=1}^{n-1} g_jg_{n-j}.
\end{aligned}  \end{equation}
Regarding the second term in $S_2$, note that for each $n\ge 3$
\[ \sum_{j=2}^{n-1}g_{j-1}g_{n-j} = \sum_{j+k=n-1\atop j,k\ge 1} g_jg_k. \]
Whence, we can manipulate the sum as
\[ (zG_1)^2 = \sum_{n\ge 2} \frac{1}{z^n} \sum_{j,k\ge 1\atop j+k=n} g_jg_k = \sum_{n\ge 3} \frac{1}{z^{n-1}} \sum_{j+k=n-1\atop j,k\ge 1} g_jg_k, \]
and so
\begin{equation} \label{eq S2 2} \frac{\del}{\del z}\left(z(G_1)^2\right) = \frac{\del}{\del z}\sum_{n\ge 3} \frac{1}{z^n}\sum_{j+k=n-2\atop j,k\ge 1} g_jg_k = -\sum_{n\ge 3} \frac{1}{z^{n+1}}\sum_{j=1}^{n-2}g_{j-1}g_{n-j} \end{equation}
holds true for $|z|>1$.  Combining Equations \ref{eq S2 1} and \ref{eq S2 2}, we have
\begin{equation} \label{eq S2 3} S_2 = \frac{\del}{\del z}(zG_1)^2 - \frac{\del}{\del z}\left(z(G_1)^2\right). \end{equation}
We have thus proved the following result.

\begin{proposition} \label{thm PDE G1} Let $G_1(t,z)$ be defined (for $t>0$ and $|z|>1$) as in Equation \ref{eq G1}.  Then $G_1$ satisfies the partial differential equation
\begin{equation} \label{eq PDE G1} \frac{\del}{\del t} G_1 = \frac{\del}{\del z} \left( z(z-1)(G_1)^2 + (z-\a-\b)G_1 - \frac{\a\b}{z}\right). \end{equation}
\end{proposition}

\begin{proof} As noted following Equation \ref{eq G1}, $\frac{\del}{\del t}G_1 = S_1+S_2$. Simplifying Equation \ref{eq S1 3} reversing the product rule,
\[ S_1 = \frac{\a\b}{z^2} + G_1 + (z-\a-\b)\frac{\del}{\del z}G_1 = \frac{\del}{\del z}\left(-\frac{\a\b}{z} + (z-\a-\b)G_1\right). \]
On the other hand, Equation \ref{eq S2 3} immediately yields
\[ S_2 = \frac{\del}{\del z}(zG_1)^2 - \frac{\del}{\del z}\left(z(G_1)^2\right) = \frac{\del}{\del z} \left( (z^2-z)G_1^2 \right). \]
Combining these equations, which are valid for $|z|>1$, yields Equation \ref{eq PDE G1}.
\end{proof}

\begin{remark} Since $\frac{\del}{\del t}G_1 = \frac{\del}{\del t}G_{\mu_t}$, we can rewrite Equation \ref{eq PDE G1} in terms of the Cauchy transform directly; it is easy to check that the result is
\begin{equation} \label{eq PDE Stieltjes} \frac{\del}{\del t} G_{\mu_t} = \frac{\del}{\del z}\left( z(z-1)(G_{\mu_t})^2 + \left((1-\a)+(1-\b)-z\right)G_{\mu_t} + \frac{(1-\a)(1-\b)}{z}\right). \end{equation}
While we have shown the equation holds only in the regime $|z|>1$, we will show below that it actually holds true on all of $\C_+$.  Hence, the poor behaviour near $z=0$ becomes a technical issue.  It is partly for this reason that the shifted transform $G(t,z) = G_{\mu_t}(z) - \frac{1-\min\{\a,\b\}}{z}$ is useful: it encapsulates the singularity in a static form, leaving a smooth flow in the vicinity of $0$, as demonstrated by the form of Equation \ref{eq PDE G} below.
\end{remark}

\begin{corollary} \label{cor PDE G} Let $G$ be the shifted Cauchy transform of Equation \ref{eq G}.  Then for $t>0$ and $|z|>1$,
\begin{equation} \label{eq PDE G} \frac{\del}{\del t}G = \frac{\del}{\del z}\left[z(z-1)G^2-(az+b)G\right] \end{equation}
where $a=2\min\{\a,\b\}-1$ and $b=|\alpha-\beta|$.
\end{corollary}

\begin{proof} From Equation \ref{eq G1}, $G_1 = G - \frac{\min\{\a,\b\}}{z}$.  Consider the case $\a\le\b$, so $G_1 = G-\frac{\a}{z}$.  We simply change variables in Equation \ref{eq PDE G1}, which holds for $|z|>1$.  As $\frac{\del}{\del t}G_1 = \frac{\del}{\del t}G$, we need only transform the quantity differentiated with respect to $z$ in Equation \ref{eq PDE G1}.
\begin{align*} & z(z-1)(G_1)^2 + (z-\a-\b)G_1-\frac{\a\b}{z} \\
=& z(z-1)\left(G-\frac{\a}{z}\right)^2 + (z-\a-\b)\left(G-\frac{\a}{z}\right) - \frac{\a\b}{z} \\
=& z(z-1)\left(G^2-\frac{2\a}{z}G + \frac{\a^2}{z^2}\right) + (z-\a-\b)G - \a + \frac{\a(\a+\b)}{z}-\frac{\a\b}{z} \\
=& z(z-1)G^2 +\left(z-\a-\b - 2\a(z-1)\right)G + \frac{\a^2}{z}(z-1)+\frac{\a(\a+\b)}{z}-\frac{\a\b}{z}-\a \\
=& z(z-1)G^2 + \left((1-2\a)z+\a-\b\right)G + \a^2-\a.
\end{align*}
Differentiating yields Equation \ref{eq PDE G} in the case $\a\le\b$; the reader may readily verify the formula also holds true in the case $\a\ge\b$.
\end{proof}

\begin{remark} Without prior knowledge of the structural singularity in the measure $\mu_t$ as $t\to\infty$ (cf.\ Equation \ref{eq mu infinity}), one might simply try to change variables by removing a pole of unknown mass at $0$: $G = G_1 + \frac{m}{z}$.  Easy calculations show that the only masses $m$ that transform Equation \ref{eq PDE G1} to a form without an explicit singularity at $0$ are $m=\a$ and $m=\b$; hence, the choice here is natural. \end{remark}

\subsection{Analyticity of the Cauchy Transform in $t$} \label{section G analytic t}

Our goal in this section is to extend Corollary \ref{corollary G smooth} to show that $G(t,z)$ is not only $C^\infty$ in $t$ but, in fact, analytic in $t$ for $|z|>1$.  To do so, we will actually use the PDE \ref{eq PDE G} proved to hold in Corollary \ref{cor PDE G}, together with our a priori knowledge of analyticity in $z$.

\begin{remark} \label{rk bad PDE} PDE \ref{eq PDE G} is a semilinear complex PDE similar in form to the complex inviscid Burger's equation.  In general, it exhibits all the hallmark pathological behaviour of nonlinear PDEs: blow-up in finite time, even with uniformly bounded initial data, and shock formation.  Since the equation is non-linear, Holmgren's uniqueness theorem does not apply.  Thus, although the Cauchy-Kowalewski theorem proves the existence of an analytic solution with given analytic initial condition, it is only unique amongst potential {\em analytic} solutions, and there may well be non-analytic solutions with the same initial data $G(0,z)$, so the $C^\infty$ result of Corollary \ref{corollary G smooth} does not immediately prove analyticity.  In this section, we use the form of the PDE, together with the a priori knowledge of analyticity in $z$, to show that our particular solution $G(t,z)$ is, indeed, analytic in $t$ as well.
\end{remark}

We begin with the following recursion for the $t$-derivatives of $G$.

\begin{lemma} \label{lemma G dot recursion} For $t>0$ and $|z|>1$, and for $k\ge 0$, define
\[ G_k = \frac{1}{k!}\frac{\del^k}{\del t^k}G(t,z). \]
Let $p=p(z) = z(z-1)$ and $q=q(z)=-(az+b)$, cf.\ Equation \ref{eq PDE G}, so that the $C^\infty(\R_+\times(\C-\overline{\mathbb{D}}))$ function $G$ satisfies the PDE $\del_t G = \del_z[pG^2+qG]$ on its domain.  Then
\begin{equation} \label{eq G dot recursion} (k+1)G_{k+1} = \frac{\del}{\del z}\left[p\sum_{j=0}^k G_jG_{k-j} + q G_k\right], \quad k\ge 0. \end{equation}
\end{lemma}

\begin{proof} The case $k=0$ is the statement of PDE \ref{eq PDE G}.  Proceeding by induction, since $G$ is $C^\infty$ we may commute $t$ and $z$ derivatives.  We have
\begin{equation} \label{eq G dot recursion 2} (k+2)G_{k+2} = \frac{k+2}{(k+2)!}\frac{\del^{k+2}}{\del t^{k+2}}G = \frac{\del}{\del t}\frac{1}{(k+1)!}\frac{\del^{k+1}}{\del t^{k+1}}G = \frac{\del}{\del t}G_{k+1}, \end{equation}
and so by the inductive hypothesis
\[ (k+2)G_{k+2} = \frac{1}{k+1}\frac{\del}{\del t}\frac{\del}{\del z}\left[p\sum_{j=0}^k G_jG_{k-j} + q G_k\right] = \frac{\del}{\del z}\left[\frac{p}{k+1}\sum_{j=0}^k \frac{\del}{\del t}(G_jG_{k-j}) + \frac{q}{k+1}\frac{\del}{\del t}G_k\right]. \]
Shifting the index down one in Equation \ref{eq G dot recursion 2} shows that $\frac{1}{k+1}\frac{\del}{\del t}G_k = G_{k+1}$, as desired.  For the quadratic term, we use the product rule.  Again utilizing Equation \ref{eq G dot recursion 2},
\[ \frac{\del}{\del t}(G_jG_{k-j}) = \frac{\del G_j}{\del t}G_{k-j} + G_j\frac{\del G_{k-j}}{\del t} = (j+1)G_{j+1}G_{k-j} + (k-j+1)G_jG_{k-j+1}. \]
Thus
\begin{align*} \sum_{j=0}^k \frac{\del}{\del t}(G_jG_{k-j}) &= \sum_{j=0}^k (j+1)G_{j+1}G_{k-j} + \sum_{j=0}^k (k-j+1)G_jG_{k-j+1} \\
&= \sum_{i=1}^{k+1} i G_iG_{k-i+1} + \sum_{j=0}^k (k-j+1)G_jG_{k-j+1}
\end{align*}
where we have made the substitution $i=j+1$ in the first sum.  Separating off the $i=k+1$ term in the first sum and the $j=0$ term in the second sum, and relabeling $j\to i$ in the second sum, this yields
\[ \sum_{j=0}^k \frac{\del}{\del t}(G_jG_{k-j}) = (k+1)G_{k+1}G_0 + \sum_{i=1}^k iG_iG_{k-i+1} + \sum_{i=1}^k (k-i+1)G_iG_{k-i+1} + (k+1)G_0G_{k+1} \]
which simplifies to
\[ \sum_{j=0}^k \frac{\del}{\del t}(G_jG_{k-j}) = (k+1)\left[G_{k+1}G_0 + \sum_{i=1}^k G_i G_{k+1-i} + G_0G_{k+1}\right] = (k+1)\sum_{i=0}^{k+1}G_iG_{k+1-i}. \]
Combining with the equation following \ref{eq G dot recursion 2} and the following comment yields the result.
\end{proof}
We will use the recursion Equation \ref{eq G dot recursion}, in conjunction with analyticity of $G$, to prove much tighter bounds on the derivatives $G_k(t,z)$ than those discussed in Remark \ref{remark bad bounds}, allowing us to prove convergence of the Taylor series of $G(t,z)$ centered at any $t>0$.  First we show that the $t$-derivatives $G_k$ are also analytic.

\begin{lemma} \label{lemma Gk analytic} For each $k\ge 0$, and each $t\ge 0$, the function $G_k(t,\cdot)$ from Lemma \ref{lemma G dot recursion} is analytic on $\C-\overline{\mathbb{D}}$. \end{lemma}

\begin{proof} First, $G_0(t,\cdot)=G(t,\cdot)$ is analytic on $\C-[0,1]$ (cf.\ Equation \ref{eq G} and Definition \ref{def Cauchy transform}), verifying the claim in this base case.  We proceed by induction. Let $K$ be any compact subset of $\C-\overline{\mathbb{D}}$, let $t_0>0$ and let $0<\e<t_0$.  By Corollary \ref{corollary G smooth}, $G_k(t,z)$ is $C^\infty$ for $t\ge 0$ and $|z|>1$, and thus $|G_{k+1}(t,z)| = \frac{1}{(k+1)!}|\frac{\del^{k+1}}{\del t^{k+1}}G(t,z)|$ is uniformly bounded on $[t_0-\e,t_0+\e]\times K$.  Now, for $|h|<\e$, the fundamental theorem of calculus asserts that
\[ \frac{G_k(t_0+h,z)-G_k(t_0,z)}{h} = \int_0^1  \frac{\del}{\del t}G_k(t_0+sh,z)\,ds = (k+1)\int_0^1 G_{k+1}(t_0+sh,z)\,ds \]
with the second equality following from Equation \ref{eq G dot recursion 2}.  Thus we have
\begin{equation} \label{eq Montel bound 1}  \left|\frac{G_k(t_0+h,z)-G_k(t_0,z)}{h}\right| \le (k+1)\sup_{z\in K\atop |t-t_0|<\e}|G_{k+1}(t,z)|. \end{equation}
Let $|h_n|<\e$ be any sequence tending to $0$, and let $\upsilon_n(z) = \frac{1}{h_n}[G_k(t_0+h_n,z)-G_k(t_0,z)]$.  By the inductive hypothesis, $\upsilon_n$ is analytic on a neighborhood of $K$, and Inequality \ref{eq Montel bound 1} shows that the family $\{\upsilon_n\}$ is uniformly bounded on $K$.  By Montel's theorem, there is a subsequence that converges normally to an analytic function on $K$.  But the sequence $\upsilon_n$ converges to $\del_t G_k(t_0,\cdot) = (k+1)G_{k+1}(t_0,\cdot)$.  This proves that $G_{k+1}$ is analytic on $K$, and hence on the domain $\C-\overline{\mathbb{D}}$ as claimed.
\end{proof}

Before proceeding to the main estimates, we state and prove a lemma which is a version of the Cauchy estimates from complex analysis.

\begin{lemma} \label{lemma Cauchy estimate} Let $z_0\in\C$ and $\eta>0$. If $h$ is analytic on a neighborhood of the closed disk $\overline{\mathbb{D}(z_0,\eta)}$, and $0<\eta'<\eta$, then
\begin{equation} \label{Cauchy estimate} |h'(z)| \le \frac{1}{\eta-\eta'}\max_{\zeta\in\overline{\mathbb{D}(z_0,\eta)}}|h(\zeta)|, \quad \text{ for all }\;\; z\in \mathbb{D}(z_0,\eta'). \end{equation}
\end{lemma}

\begin{proof} From the Cauchy integral formula, if $r>0$ is such that $h$ is analytic on a neighborhood of $\overline{\mathbb{D}(z,r)}$, then
\[ h'(z) = \frac{1}{2\pi i}\oint_{\del\mathbb{D}(z,r)}\frac{h(\zeta)}{(\zeta-z)^2}\,d\zeta. \]
Thus
\begin{equation} \label{eq Cauchy estimate 1} |h'(z)| \le \frac{1}{2\pi} \oint_{\del\mathbb{D}(z,r)}\frac{|h(\zeta)|}{|\zeta-z|^2}d\zeta = \frac{1}{2\pi r^2} \oint_{\del\mathbb{D}(z,r)}|h(\zeta)|d\zeta \le \frac{1}{2\pi r^2}\cdot 2\pi r \cdot \max_{\zeta\in\del\mathbb{D}(z,r)}|h(\zeta)|. \end{equation}
If $z\in\mathbb{D}(z_0,\eta')$, then the closed disk $\overline{\mathbb{D}(z,\eta-\eta')}$ is contained in $\mathbb{D}(z_0,\eta)$ where $h$ is holomorphic; so, taking $r=\eta-\eta'$ in Inequality \ref{eq Cauchy estimate 1} yields
\[ |h'(z)| \le \frac{1}{\eta-\eta'}\max_{\zeta\in\del\mathbb{D}(z,\eta-\eta')}|h(\zeta)| \le \frac{1}{\eta-\eta'}\max_{\zeta\in\overline{\mathbb{D}(z_0,\eta)}}|h(\zeta)| \]
as desired.
\end{proof}

Let us introduce the following maximal functions.  For $k\ge 0$, $t_0>0$, $|z_0|>1$, and $0<\eta<|z_0|-1$,
\begin{equation} \label{eq max fns} M_k(t_0,z_0;\eta) = \max_{\zeta\in\overline{\mathbb{D}(z_0,\eta)}}|G_k(t_0,\zeta)|, \qquad M_k'(t_0,z_0;\eta) = \max_{\zeta\in\overline{\mathbb{D}(z_0,\eta)}}|G_k'(t_0,\zeta)|. \end{equation}
When the point $(t_0,z_0)$ is understood from context, we will shorten the notation to $M_k(\eta) \equiv M_k(t_0,z_0;\eta)$ and $M_k'(\eta)\equiv M_k'(t_0,z_0;\eta)$.  By the analyticity result of Lemma \ref{lemma Gk analytic}, the Cauchy estimates of Lemma \ref{lemma Cauchy estimate} yield the following {\em maximal Cauchy estimate}:
\begin{equation} \label{eq max Cauchy estimate} M_k'(\eta') \le \frac{M_k(\eta)}{\eta-\eta'}. \end{equation}
We now use Inequality \ref{eq max Cauchy estimate}, together with the recursion of Equation \ref{eq G dot recursion}, to prove exponential-growth bounds for $G_k$ inductively.  The key idea is to apply the maximal Cauchy estimate repeatedly with an appropriately chosen $\eta''\in(\eta',\eta)$.  It is important that $\eta''$ be chosen to minimize the bound for each individual term, or else the resulting estimates blow up super-exponentially.

\begin{proposition} Let $t_0>0$ and $|z_0|>1$, and let $0<\eta<\min\{|z_0|-1,1\}$, so that the disc $\mathbb{D}(z_0,\eta)$ of radius $\eta$ centered at $z_0$ is contained in $\C-\overline{\mathbb{D}}$.  There is a constant $c = c(z_0,\eta)$ so that, for all $0<\eta'<\eta$ and all $k\ge 0$,
\begin{equation} \label{eq Oct13 estimate} M_k(\eta') \le \frac{c^k(\max\{M_0(\eta),\textstyle{\frac12}\})^{k+1}}{(\eta-\eta')^k}. \end{equation}
\end{proposition}

\begin{remark} It will be important that the constant $c$ does not depend on $t_0$.  Indeed, we will see that $c$ does not depend on the value of $G_0$ at all; it can be taken to equal $52$ times the maximum modulus of the polynomials $p,q,p',q'$ on $\mathbb{D}(z_0,\eta)$, with $p(z) = z(z-1)$ and $q(z)=-(az+b)$. \end{remark}

\begin{proof} The function $G_0(t_0,\cdot)$ is analytic on a neighborhood of $\overline{\mathbb{D}(z_0,\eta)}$, so it is uniformly bounded by $M_0(\eta)$ on the disk, which proves the inequality in the base case $k=0$.  We proceed by induction.  Fix $k\ge 0$ and suppose that we have shown that, for $0\le \ell\le k$, there are constants $c_\ell= c_\ell(\eta,z_0)$ so that\begin{equation} \label{eq maximal induction} M_\ell(\eta') \le \frac{c_\ell}{(\eta-\eta')^\ell}, \quad 0\le \ell\le k. \end{equation}
Proceeding to $k+1$, we use the recursion of Equation \ref{eq G dot recursion}, which we now expand:
\begin{equation} \label{eq expand recursion} \begin{aligned} (k+1)G_{k+1} &= \frac{\del}{\del z}\left[p\sum_{j=0}^k G_jG_{k-j} + q G_k\right] \\
&= p'\sum_{j=0}^k G_jG_{k-j} + 2p\sum_{j=0}^k G_j'G_{k-j} + q' G_k + qG_k'. \end{aligned} \end{equation}
We will bound each term in this recursion separately.  Recall that $p(z)=z(z-1)$ and $q(z)=-(az+b)$ are polynomials.  Hence there is a constant $\lambda \ge 1$ so that $\max\{|p(z)|,|p'(z)|,|q(z)|,|q'(z)|\} \le \lambda$ for all $z\in\overline{\mathbb{D}(z_0,\eta)}$.  Thus, on $\mathbb{D}(z_0,\eta')$,
\[ \left| p'\sum_{j=0}^k G_jG_{k-j}\right| \le \lambda \sum_{j=0}^k |G_j||G_{k-j}| \le \lambda \sum_{j=0}^k M_j(\eta') M_{k-j}(\eta') \le \lambda\sum_{j=0}^k \frac{c_j}{(\eta-\eta')^j}\frac{c_{k-j}}{(\eta-\eta')^{k-j}}. \]
Hence, the first term is bounded by
\begin{equation} \label{eq maximal first term} 
\left| p'\sum_{j=0}^k G_jG_{k-j}\right| \le \frac{1}{(\lambda-\lambda')^k}\cdot\lambda\sum_{j=0}^k c_jc_{k-j} \le \frac{1}{(\eta-\eta')^{k+1}}\cdot\lambda\sum_{j=0}^k c_jc_{k-j},
\end{equation}
where the final inequality is justified by the assumption that $\eta\le 1$ so that $\eta-\eta'<1$.

For the second term in \ref{eq expand recursion}, we can make the initial estimate
\[ \left| 2p\sum_{j=0}^k G_j'G_{k-j} \right| \le 2\lambda\sum_{j=0}^k M_j'(\eta')M_{k-j}(\eta') \le 2\lambda\sum_{j=0}^k \frac{c_{k-j}}{(\eta-\eta')^{k-j}}M_j'(\eta'). \]
Now, for any $\eta''\in(\eta',\eta)$, Inequality \ref{eq max Cauchy estimate} with $\eta''$ in the role of $\eta$ yields
\begin{equation} \label{eq 2nd and 4th term 1} M_j'(\eta')\le \frac{M_j(\eta'')}{\eta''-\eta'}. \end{equation}
Now applying the inductive hypothesis Inequality \ref{eq maximal induction}, this time with $\eta''$ in the role of $\eta'$, yields
\begin{equation} \label{eq 2nd and 4th term 2} M_j(\eta'') \le \frac{c_j}{(\eta-\eta'')^j}. \end{equation}
Thus, we have the estimate
\begin{equation} \label{eq maximal 2nd term 1} \left| 2p\sum_{j=0}^k G_j'G_{k-j} \right| \le 2\lambda\sum_{j=0}^k \frac{c_jc_{k-j}}{(\eta''-\eta')(\eta-\eta'')^j(\eta-\eta')^{k-j}} \end{equation}
which holds for any $\eta''\in(\eta',\eta)$.  We now optimize the inequality over $\eta''$ separately in each term in the sum.  By elementary calculus, we find that the minimum occurs at $\eta'' = \frac{j}{j+1}\eta' + \frac{1}{j+1}\eta$.  At this point,
 \[ \eta-\eta'' = \textstyle{\frac{j}{j+1}}(\eta-\eta'), \qquad \eta''-\eta' = \textstyle{\frac{1}{j+1}}(\eta-\eta'). \]
 Thus,
 \begin{equation} \label{eq eta'' maximizer} \inf_{\eta'<\eta''<\eta}\frac{1}{(\eta''-\eta')(\eta-\eta'')^j} = \left(\textstyle{\frac{j+1}{j}}\right)^j (j+1)\frac{1}{(\eta-\eta')^{j+1}} \le \frac{e\cdot (j+1)}{(\eta-\eta')^{j+1}}. \end{equation}
Inserting these estimates into the terms in Inequality \ref{eq maximal 2nd term 1}, we have
\begin{equation} \label{eq maximal 2nd term} \left| 2p\sum_{j=0}^k G_j'G_{k-j} \right| \le 2\lambda\sum_{j=0}^k \frac{e\cdot(j+1)\cdot c_jc_{k-j}}{(\eta-\eta')^{j+1}(\eta-\eta')^{k-j}} 
= \frac{1}{(\eta-\eta')^{k+1}}\cdot 2e\lambda(k+1) \sum_{j=0}^k c_jc_{k-j}.
\end{equation}

The third term in \ref{eq expand recursion} is straightforward to estimate:
\begin{equation} \label{eq maximal 3rd term} |q'G_k| \le \lambda M_k(\eta') \le \lambda \frac{c_k}{(\eta-\eta')^k} \le \frac{1}{(\eta-\eta')^{k+1}} \cdot\lambda c_k \end{equation}
again using the assumption $\eta-\eta'<1$.

For the fourth term, we use the same approach as the second term.  First we have $|qG_k'| \le \lambda M_k'(\lambda')$.  Let $\lambda''=\frac{k}{k+1}\eta'+\frac{1}{k+1}\eta$.  Then Inequalities \ref{eq 2nd and 4th term 1}, \ref{eq 2nd and 4th term 2}, and \ref{eq eta'' maximizer} give
\begin{equation} \label{eq maximal 4th term} |qG_k'| \le \frac{c_k}{(\eta''-\eta')(\eta-\eta'')^k} \le \frac{1}{(\eta-\eta')^{k+1}}\cdot e\lambda(k+1)c_k. \end{equation}

Finally, combining Inequalities \ref{eq maximal first term}, \ref{eq maximal 2nd term}, \ref{eq maximal 3rd term}, and \ref{eq maximal 4th term} with Inequality \ref{eq expand recursion} shows that
\[ (k+1)|G_{k+1}| \le \frac{1}{(\eta-\eta')^{k+1}}\cdot\left[\lambda\sum_{j=0}^k c_jc_{k-j} + 2e\lambda(k+1)\sum_{j=0}^kc_jc_{k-j} + \lambda c_k + e\lambda(k+1) c_k\right] \]
and thus
\[ M_{k+1}(\eta') \le \frac{1}{(\eta-\eta')^{k+1}}\cdot (1+2e)\lambda\left[\sum_{j=0}^k c_jc_{k-j} + c_k\right]. \]
This completes the induction to show that Inequality \ref{eq maximal induction} holds additionally for $\ell=k+1$, provided that
\begin{equation} \label{eq verify inducion} c_{k+1} \ge  (1+2e)\lambda\left[\sum_{j=0}^k c_jc_{k-j} + c_k\right]. \end{equation}
To this end, we now recursively define
\begin{equation} \label{eq Catalan} c_{k+1} = 2(1+2e)\lambda\sum_{j=0}^k c_jc_{k-j} \;\text{ for }k\ge 0, \quad c_0 =\max\{M_0(\eta),\textstyle{\frac12}\}. \end{equation}
Note that $2c_0\ge 1$ and (by induction) $c_j\ge 0$ for all $j$, and so
\begin{align*} c_{k+1} = 2(1+2e)\lambda\left[\sum_{j=0}^k c_jc_{k-j} + \sum_{j=0}^k c_jc_{k-j}\right] &\ge (1+2e)\lambda\left[\sum_{j=0}^k c_jc_{k-j}+2c_0c_k\right] \\
&\ge (1+2e)\lambda\left[\sum_{j=0}^k c_jc_{k-j} + c_k\right] \end{align*}
as desired.  Thus, to conclude the proof, it only remains to show that the constants defined by Equation \ref{eq Catalan} are bounded by the given exponential form.  In fact, it is straightforward to verify that the solution to the recursion Equation \ref{eq Catalan} is given by the scaled Catalan numbers
\[ c_k = (2(1+2e)\lambda)^k (\max\{M_0(\eta),\textstyle{\frac12}\})^{k+1} C_k, \]
where $C_k = \frac{1}{k+1}\binom{2k}{k} \le 4^k$.  Therefore, taking $c = 8(1+2e)\lambda$ proves the result.

\end{proof}

This brings us to the main result of this section: that $G(t,z)$ is analytic in $t$ for $|z|>1$.

\begin{proposition} \label{prop analytic z>1} The function $G(t,z)$ of Equation \ref{eq G} is analytic in both variables $(t,z)$ for $t>0$ and $|z|>1$.
\end{proposition}

\begin{proof} Since $G(t,\cdot)$ is analytic on $\C-[0,1]$ (cf.\ Equation \ref{eq G} and Definition \ref{def Cauchy transform}), we need only concern ourselves with (real) analyticity in $t$.  Let $|z_0|>1$ and $t_0>0$. Corollary \ref{corollary G smooth} shows that $t\mapsto G(t,z_0)$ is $C^\infty$, and so it suffices to show that the Taylor series $T_{t_0}G(\cdot,z_0)$ of this function is convergent in a neighborhood of $t=t_0$.  By definition (cf.\ Lemma \ref{lemma G dot recursion}),
\begin{equation} \label{eq Taylor series G(t)} T_{t_0}G(t,z_0) = \sum_{k=0}^\infty \frac{1}{k!}\frac{\del^k}{\del t^k}G(t_0,z_0)(t-t_0)^k = \sum_{k=0}^\infty G_k(t_0,z_0)(t-t_0)^k. \end{equation}
Let $0<\eta<\min\{|z_0|-1,1\}$, and $0<\eta'<\eta$; then Inequality \ref{eq Oct13 estimate} gives exponential bounds
\[ |G_k(t_0,z_0)| \le M_k(\eta') \le  \frac{c^k(\max\{M_0(\eta),\textstyle{\frac12}\})^{k+1}}{(\eta-\eta')^k}. \]
The root test then implies that the Taylor series $T_{t_0}G(t,z_0)$ converges for $|t-t_0| < (\eta-\eta')/c\max\{M_0(\eta),\frac12\}$.  This proves the proposition.
\end{proof}

\subsection{Analytic Continuation to $\C_+$ and the Proof of Theorem \ref{thm PDE}}

We have now proven the statement of Theorem \ref{thm PDE} {\em restricted to $\C_+-\overline{\mathbb{D}}$}: $G(t,z)$ is analytic in both $t>0$ and $z\in\C_+$ subject to the constraint $|z|>1$, and it satisfies the PDE \ref{eq PDE} on this domain (cf. Corollary \ref{cor PDE G}).  We knew a priori that $G(t,z)$ is analytic in $z$ for all $z\in\C_+$ (cf.\ Equation \ref{eq G} and Definition \ref{def Cauchy transform}); it thus remains to extend the analyticity in $t$ into the larger $z$-domain.  This is actually quite simple, once we reinterpret $t$ as a complex variable.

\begin{lemma} \label{lemma analytic continuation} For each $z_0\in\C-\overline{\mathbb{D}}$, there is $\e=\e(z_0)>0$ so that $t\mapsto G(t,z_0)$ has an analytic continuation to the strip $\{t\in\C\colon \Re t>0, |\Im t|<\e(z_0)\}$.  The tolerance $\e(z_0)$ is independent of $t$.
\end{lemma}

\begin{proof} This follows immediately from the proof of Proposition \ref{prop analytic z>1}: for $t>0$, $G(t,z_0)$ is given by the power series $T_{t_0}G(t,z_0)$ of Equation \ref{eq Taylor series G(t)} centered at $t_0$, with radius of convergence at least $(\eta-\eta')/c(\eta)\max\left\{M_0(\eta),\frac12\right\}$ for any $0<\eta'<\eta<\min\{|z_0|-1,1\}$.  We can therefore take $\eta'\downarrow 0$ and $\eta\uparrow \eta_0\equiv \min\{|z_0|-1,1\}$, and therefore define $\e(z_0) = \eta_0/c(\eta_0)\max\{M_0(\eta_0),\frac12\}>0$, independent of $t_0$; these power series, for all base points $t_0>0$, define the analytic continuation.
\end{proof}

Hence $G(t,z)$ can be viewed as a complex analytic function of two variables.  This brings to bear all the tools of many complex variables.  We can now complete the proof of Theorem \ref{thm PDE}, with the help of a lemma of Hartog.

\begin{proof}[Proof of Theorem \ref{thm PDE}]  \label{proof of thm PDE} Fix a point $(t_0,z_0)\in\R_+\times\C-\overline{\mathbb{D}}$.  Let $D$ be the largest disk centered at $z_0$ that does not intersect $[0,1]$, and let $D'$ be the largest disk centered at $z_0$ that does not intersect $\mathbb{D}$ (so $D'\subseteq D$).  Then $G(t,z)$ is complex analytic in $t>0$ for $z\in D'$ by Lemma \ref{lemma analytic continuation}.  Since $G(t,z)$ is the Cauchy transform of a positive measure $\nu_t = \mu_t-(1-\min\{\alpha,\beta\})\delta_0$ of total mass $\le 1$, supported in $[0,1]$ (cf.\ Equation \ref{eq G} and Definition \ref{def Cauchy transform}), it is complex analytic in $z$ for all $z\in D$, and is also uniformly bounded on compact subsets of $\C_+$ (cf.\ Inequality \ref{eq G uniform bound}). It therefore follows from a lemma of Hartog \cite[Lemma 2.2.11]{Hormander} that $G(t,z)$ is jointly analytic in $(t,z)$ for $t>0$ and $z$ in the larger disk $D$.  Applying this at each point of $z_0\in\C_+-\overline{\mathbb{D}}$ shows that $G(t,z)$ is analytic on $\R_+\times \C_+$, as desired.  Thus, the functions on both sides of PDE \ref{eq PDE} are analytic on this domain, and by Corollary \ref{cor PDE G} the are equal on the open set $\R_+\times (\C_+-\overline{\mathbb{D}})$, it follows that they are equal on their larger analytic domain $\R_+\times \C_+$, concluding the proof.
\end{proof}

\section{Local Properties of the Flow $\mu_t$} \label{section local properties}

In this section, we develop properties of the measure $\mu_t$ directly from the PDE of Theorem \ref{thm PDE} that determines its Cauchy transform.  Let us define a new positive finite measure $\nu_t$, supported in $[0,1]$, by
\begin{equation} \label{eq nu} \mu_t = \nu_t + (1-\min\{\alpha,\beta\})\delta_0 \end{equation}
so that $G(t,z) = G_{\nu_t}(z)$ is the Cauchy transform of $\nu_t$, cf.\ Equation \ref{eq G}.  Since $\mu_t$ is a probability measure, the total mass of $\nu_t$ is
\begin{equation} \label{eq mass nu} \nu([0,1]) = \min\{\alpha,\beta\}\ge 0. \end{equation}

\subsection{Steady-State Solution} \label{section steady state} To begin, as a sanity check, note that the steady-state equation (determined by $\frac{\del}{\del t} G(t,z) = 0$) takes the form
\[ \del_z[z(z-1)G^2 - (az+b)G] = 0. \]
Due to the analyticity of $G$ on the connected domain $\C_+$, this forces $z(z-1)G^2-(az+b)G$ to be constant.  The constant can be determined from the known limit behaviour of the Cauchy transform:
\[ \lim_{|z|\to\infty}zG(z) = \nu_t([0,1]) = \min\{\alpha,\beta\}. \]
Thus
\[ \lim_{|z|\to\infty}[z(z-1)G(z)^2-(az+b)G(z)] = \min\{\alpha,\beta\}^2 - a\min\{\alpha,\beta\}. \]
Using $a=2\min\{\alpha,\beta\}-1$, it follows that the steady state solution $G_\infty$ satisfies
\[ z(z-1)G_\infty(z)^2-(az+b)G_\infty(z) = \min\{\alpha,\beta\}(1-\min\{\alpha,\beta\}) \]
This quadratic equation has (simplified) solutions
\[ G_\infty(z) = \frac{(az+b)\pm\sqrt{z^2-2(\alpha+\beta-2\alpha\beta)z+(\alpha-\beta)^2}}{2z(z-1)}. \]
The discriminant matches that in Equation \ref{eq G infinity}, but the terms outside the radical do not match; of course, that equation describes the Cauchy transform of the full limit measure $\mu$, while our measures $\nu_t$ have the point mass at $z=0$ removed.  Adding it back in,
\[ \frac{1-\min\{\alpha,\beta\}}{z} + G_\infty(z) = \frac{z+\alpha+\beta-2\pm\sqrt{z^2-2(\alpha+\beta-2\alpha\beta)z+(\alpha-\beta)^2}}{2z(z-1)} \]
which precisely matches the Jacobi measure's Cauchy transform from Equation \ref{eq G infinity}, as expected.  In other words, using the Stieltjes continuity Theorem \ref{thm Stieltjes continuity}, we have:
\begin{proposition} The spectral measure $\mu_t$ of $qp_tq$ converges weakly to the free Jacobi measure $\mu$ of Equation \ref{eq mu infinity} as $t\to\infty$. \end{proposition}
\noindent This was already known, as a consequence of the fact that $p_t,q$ are asymptotically free as $t\to\infty$; it is interesting that this can be seen directly from PDE \ref{eq PDE}.

\subsection{Conservation of Mass and Propagation of Singularities} \label{section conservation} The support of the measure $\mu_t$ in the limit as $t\to\infty$ is the full interval $[0,1]$; however, the initial condition $\mu_0$ is only constrained to have point masses of appropriate magnitudes at the endpoints $\{0,1\}$, cf.\ Equation \ref{eq mu infinity}, and thus $\supp\mu_0$ may be any closed subset of $[0,1]$.  It is therefore possible that, for some $t>0$, $\supp\mu_t$ is disconnected.  The following results deal with the flow of such support ``bumps'' under PDE \ref{eq PDE}.

\begin{lemma} \label{lemma finite speed 1} Let $t_0>0$.  Let $U_1,U_2$ be two disjoint open subintervals of $[0,1]$ (with the relative topology), and let $K_1\subset U_1$ and $K_2\subset U_2$ be closed subsets.  Suppose that $\supp\mu_{t_0}\subseteq K_1\sqcup K_2$.  Then, for some $\e>0$, $\supp\mu_t\subset U_1\sqcup U_2$ for all $t\in(t_0-\e,t_0+\e)$.
\end{lemma}

\begin{proof} Applying the analytic continuation argument in the proof of Theorem \ref{thm PDE} (on page \pageref{proof of thm PDE}) to the larger domain $\C-\overline{U_1\sqcup U_2}$ shows that $G(t,z)$ is analytic in both variables for $z$ in this domain, and PDE \ref{eq PDE} also holds there in a neighborhood of time $t_0$.  The result follows immediately. \end{proof}
In particular, this shows that, if for some $t_0>0$ the support of $\mu_{t_0}$ consists of a finite (or countable) collection of disjoint closed intervals in $[0,1]$, the same holds true for all $t\ge t_0$.  Moreover, we can quantify the motion of the endpoints of these intervals.  For the following, we assume that $G_{\nu_t}$ has a continuous extension to $(0,1)$ (i.e.\ $\mu_t$ has a continuous density); we will prove this assumption holds true for all $t>0$ in Section \ref{section subordination}, at least in the special case $\alpha=\beta=\frac12$.

\begin{lemma} \label{lemma finite speed 2} Suppose $G(t,z)$ has a continuous extension to $z\in(0,1)$ for all $t>0$.  Let $x_t$ be a point in the boundary of the support of $\mu_t$.  Then $t\mapsto x_t$ satisfies the ODE
\begin{equation} \label{eq ODE xt} \dot{x}_t = 2G(t,x_t)x_t(1-x_t)+ax_t+b, \quad t>0. \end{equation}
\end{lemma}

\begin{proof} The Cauchy transform of a compactly-supported measure is one-to-one on a neighborhood of $\infty$ in $\C_+$, so at least for small $w$, there is an analytic function $K(t,w)$ so that $w = G(t,K(t,w))$.  To simplify notation, denote $G(t,z) = G_t(z)$ and $K(t,w)=K_t(w)$; let $G_t'(z) = \frac{\del}{\del z}G(t,z)$ and $K_t'(w) = \frac{\del}{\del w}K(t,w)$.  Noting that $w = G_t(K_t(w))$ for small $w$ and differentiating with respect to $t$, we have
\begin{equation} \label{PDE K 0} 0 = \frac{\del}{\del t}G_t(K_t(w)) = \frac{\del G_t}{\del t}(K_t(w)) + G_t'(K_t(w))\frac{\del K_t}{\del t}(w). \end{equation}
Now, PDE \ref{eq PDE} can be written in the form
\[ \frac{\del G_t}{\del t}(z) = (2z-1)G_t(z)^2 + 2z(z-1)G_t(z)G_t'(z) - aG_t(z) -(az+b)G_t'(z), \]
and so, using $G_t(K_t(w))=w$, Equation \ref{PDE K 0} yields
\begin{equation} \begin{aligned} &-G_t'(K_t(w))\frac{\del K_t}{\del t}(w) \\ & \qquad = (2K_t(w)-1)w^2 + 2K_t(w)(K_t(w)-1)wG_t'(K_t(w))-aw-(aK_t(w)+b)G_k'(K_t(w)). \end{aligned} \label{eq PDE K 1}  \end{equation}
The inverse function theorem shows that $G_t'(K_t(w)) = 1/K_t'(w)$.  Combining this with Equation \ref{eq PDE K 1} and simplifying, we have
\begin{equation} \label{eq PDE K 2} -\frac{\del K_t}{\del t}(w) = (2K_t(w)-1)K_t'(w)w^2+2K_t(w)(K_t(w)-1)w-aK_t'(w)w -(aK_t(w)+b) \end{equation}
which can be written in the form
\begin{equation} \label{eq PDE K }
\frac{\del K_t}{\del t}(w) = \frac{\del}{\del w}\left[w^2K_t(w)(1-K_t(w)) + awK_t(w)\right]+b.
\end{equation}
The existence of a continuous extension of $G_t$ to a neighborhood of $(0,1)$, coupled with standard analytic continuation arguments, then implies that the PDE \ref{eq PDE K 2} holds up to the boundary $G_t((0,1))$.

Now, let $x_t$ be a boundary point of a support interval of $\nu_t$; then $G_t$ is singular at $x_t$.  Let $y_t = G_t(x_t)$ be the corresponding singular value; then $K_t'(y_t)=0$.  Hence, with $w=y_t$, equation \ref{eq PDE K 2} asserts that
\[ -\frac{\del K_t}{\del t}(y_t) = 2K_t(y_t)(K_t(y_t)-1)y_t - (aK_t(y_t)+b) = 2x_t(x_t-1)y_t-(ax_t+b). \]
On the other hand,
\[ \frac{dx_t}{dt}= \frac{\del}{\del t}K_t(y_t)= \frac{\del K_t}{\del t}(y_t) + K_t'(y_t)\frac{dy_t}{dt} = \frac{\del K_t}{\del t}(y_t) \]
since $y_t$ is (by assumption) a critical point for $K_t$.  Hence, we have the ODE
\[ \dot{x}_t = 2y_tx_t(1-x_t)+(ax_t+b) \]
which yields the result, since $y_t = G_t(x_t)$.
\end{proof}
Equation \ref{eq ODE xt} gives a precise (implicit) formula for the speed of propagation of singularities on the boundary: the edges of the support move with finite speed $\dot{x}_t$ so long as they stay away from the endpoints $\{0,1\}$, since $G(t,x)x(1-x)$ is continuous on $(0,1)$.  In the sequel, we will see that, although $G(t,x)$ may blow up at $x=0,1$, the function $G(t,x)\sqrt{x(1-x)}$ remains bounded; thus, Equation \ref{eq ODE xt} yields further information: as any support ``bump'' approaches the endpoints, its speed decreases to $0$.

\begin{remark} Equation \ref{eq ODE xt} at first seems to suggest the boundary point $x_t$, which is of course in $[0,1]$, evolves into the lower half-plane: indeed, the function $G(t,z)$ takes values in the closed lower half-plane even for $z\in[0,1]$.  However, $x_t$ is assumed to be at the boundary of the support of $\mu_t$ which, under the assumption of the lemma, possesses a continuous density.  The Stieltjes inversion formula of Equation \ref{eq Stieltjes} then shows that $\Im G(t,x_t)=0$, since the density of $\mu_t$ is $0$ at a boundary point of the support; hence, $G(t,x_t)$ is, in fact, real.
\end{remark}

\begin{remark} If the support of $\mu_t$ is not the full interval, or more generally if $G_t$ possesses a singular point in $[0,1]$, this singular point cannot dissipate in finite time; this follows from Hartog's second theorem and the continuity principle, cf.\ \cite{Hormander}.  Lemmas \ref{lemma finite speed 1} and \ref{lemma finite speed 2} are in line with this observation: singular points propagate with finite speed and slow down as they approach the endpoints, therefore never dissipating.
\end{remark}

One final result that follows from this framework is conservation of mass of support ``bumps''.

\begin{lemma} \label{lemma conservation of mass} Let $t_0,U_1,U_2$ be as in Lemma \ref{lemma finite speed 1}. The total mass of $\left.\mu_t\right|_{U_1}$ is preserved for $t\in(t_0-\e,t_0+\e)$.
\end{lemma}

\begin{proof} Since, by Lemma \ref{lemma finite speed 1}, $\supp\nu_t\subseteq \supp\mu_t$ is contained in $U_1\sqcup U_2$ for $t\in(t_0-\e,t_0+\e)$, the Cauchy transform $G(t,z)=G_{\nu_t}(z)$ is analytic for $z\in\C-\overline{U_1\sqcup U_2}$.  Let $\alpha_1$ be a closed $C^1$ curve in $\C-\overline{U_1\sqcup U_2}$ which encloses $\overline{U_1}$ and has winding number $0$ around each point in $\overline{U_2}$.  Then
\begin{equation} \label{eq residue} \mu_t(\overline{U_1}) = \frac{1}{2\pi i}\oint_{\alpha_1} G(t,z)\,dz. \end{equation}
Equation \ref{eq residue} holds by the standard Residue Theorem if $\mu_t$ is a discrete measure -- a convex combination of point-masses -- supported in $\overline{U_1}$; any measure may be weakly approximated by discrete measures, and by the Stieltjes continuity Theorem \ref{thm Stieltjes continuity} weak convergence of measures implies uniform convergence of the Cauchy transforms on compact subsets of $\C-\overline{U_1\sqcup U_2}$, justifying Equation \ref{eq residue}.

Lemma \ref{lemma finite speed 1} justifies that PDE \ref{eq PDE} holds in $\C-\overline{U_1\sqcup U_2}$, and the solution is analytic in $z$ and $t$.  In particular, it follows that the integral on the right-hand-side of Equation \ref{eq residue} can be differentiated with respect to $t$ under the integral, so $\mu_t(\overline{U_1})$ is differentiable, and PDE \ref{eq PDE} then gives
\begin{equation} \label{eq residue 2} \frac{\del}{\del t}\mu_t(\overline{U_1}) = \frac{1}{2\pi i}\oint_{\alpha_1} \frac{\del}{\del t}G(t,z)\,dz = \frac{1}{2\pi i}\oint_{\alpha_1}\frac{\del}{\del z}\left[z(z-1)G(t,z)^2 + (az+b)G(t,z)\right]\,dz \end{equation}
for $t\in(t_0-\e,t_0+\e)$.  The analyticity of $G(t,z)$ in a neighborhood of the image of $\alpha_1$ now guarantees the integral on the right-hand-side of Equation \ref{eq residue 2} is $0$, by the fundamental theorem of calculus.
\end{proof}

\begin{remark} Of course, if support ``bumps'' merge in finite time, the total mass combines additively. \end{remark}

\subsection{Subordination for the Liberation Process} \label{section subordination} This section is devoted to the proof of Theorem \ref{thm smoothing}.  We assume, from this point forward, that $\alpha=\beta=\frac12$.  It is plausible that similar techniques may apply more generally (indeed \cite{Izumi-Ueda} gives promising progress in this direction), but that discussion is left to a future publication. The proof is fairly involved: we essentially develop an analogue in the present context of Biane's theory of subordination for the additive free convolution.  To clarify the proof, we begin with an outline.  The case $\alpha=\beta=\frac12$ corresponds to $a=b=0$ in Equation \ref{eq PDE}.  So $\nu_t$ is a positive measure of mass $\frac12$ supported in $[0,1]$, and its Cauchy transform $G(t,z) = G_{\nu_t}(z)$ satisfies the PDE
\begin{equation} \label{eq PDE 0} \frac{\del}{\del t}G(t,z) = \frac{\del}{\del z}\left[z(z-1)G(t,z)^2\right], \quad t>0,\; z\in\C_+. \end{equation}
We begin by making a change of variables.  The function $z\mapsto\sqrt{z}\sqrt{z-1}$ (where we use the standard branch of the square root function) is analytic on $\C-[0,1]$.  Let
\begin{equation} \label{eq def H} H(t,z) = H_t(z) = \sqrt{z}\sqrt{z-1}G(t,z), \quad t>0,\; z\in\C_+. \end{equation}
Then $H(t,z)$ is analytic in both variables for $z\in\C_+$ and $t>0$, and satisfies the PDE
\begin{equation} \label{eq PDE H} \frac{\del}{\del t}H(t,z) = \sqrt{z}\sqrt{z-1}\frac{\del}{\del z}\left[H(t,z)^2\right] \end{equation}
on its analytic domain.  We introduce the auxiliary function $M$, and the domain $S$:
\begin{equation} \label{def M S} M(w) = \textstyle{\frac12e^{-2w}(e^{2w}+\frac12)^2, \qquad S = \{w=u+iv\in\C\colon u> \frac12\ln\frac12, 0<v<\frac{\pi}{2}\}}. \end{equation}
The function $M$ is entire; when restricted to the strip $S$, $\left.M\right|_S$ is injective, and its image $M(S) = \C_+$.  Let $L = (\left.M\right|_S)^{-1}$; explicitly
\begin{equation} \label{def L} L(z) = \textstyle{\frac12\ln M(z) = \frac12\ln\left(z-\frac12+\sqrt{z}\sqrt{z-1}\right)}, \quad z\in\C. \end{equation}
In fact $L$ has an analytic continuation to $\C-(-\infty,1]$ (given by Equation \ref{def L}), and this extension is injective: it is the inverse of $M$ on the larger domain $\overline{S}\cup \bar{S}$ (closure union conjugate).  On this domain, $L'(z) = \frac{1}{2\sqrt{z}\sqrt{z-1}}$.  More generally, $L$ extends to a multivalued holomorphic function (on the infinite-sheet Riemann surface covering the slit plane $\C-[0,1]$).

Given the solution $H_t$ to PDE \ref{eq PDE H}, define the {\bf subordinator} $f_t$ as
\begin{equation} \label{def ft}  f_t(z) = M[L(z)+tH_t(z)], \qquad z\in\C_+. \end{equation}
This subordinator is a deformation of the identity: $f_0(z) = M(L(z)) = z$.  (One should be wary, however: when $w\in\C_+$, $L(M(w))$ is only equal to $w$ for $w\in S$.)  Using the method of characteristics, we will show that $H(t,z)$ satisfies the fixed-point equation
\begin{equation} \label{eq Ht ft} H_t(z) = H_0(f_t(z)). \end{equation}
Equation \ref{eq Ht ft} transfers the dynamics of PDE \ref{eq PDE H} to a deformation $f_t$ of the identity, changing the role of $H_t$ from active to passive.  Immediately from this equation, we see that smoothness is propagated: if $\left.H_0\right|_{\C_+}$ happens to have an analytic continuation to a complex neighborhood of $(0,1)$, then the same must be true for $H_t$ for all $t>0$.

In fact, we will prove (Theorem \ref{thm homeomorphism}) that, for any initial measure $\nu_0$ and any $t>0$, $f_t$ extends to a homeomorphism from the closed upper half-plane $\overline{\C_+}$ to a region in $\overline{\C_+}$ bounded by a Jordan curve; it then follows from Carath\'eodory's Theorem (cf.\ \cite[Thm 2.6]{Pommerenke}) that $H_t$ has a continuous extension to $\overline{\C_+}$.  Since the measure $\nu_t$ is given by the boundary values of $G_t$, it follows that $\nu_t$ has a density $\rho_t$ which is continuous (with potential blow up at $\{0,1\}$ due to the factor $\sqrt{x(1-x)}$ relating $H_t(x)$ and $G_t(x)$ for $x\in[0,1]$).  Finally, we show that the extension of $H_t$ actually has an analytic continuation to a complex neighborhood of any point in the interior of $\supp\nu_t$, concluding the proof of Theorem \ref{thm smoothing}.

\bigskip

To begin, we record the relevant properties of the auxiliary function $L$; the following lemma is a straightforward exercise.

\begin{lemma} \label{lemma L properties} Let $L\colon \C_+\to \C$ be defined as in Equation \ref{def L}.  Then $L$ is holomorphic and injective, its range is $L(\C_+)=S$, and $L$ satisfies
\begin{equation} \label{eq L'} \frac{\del L}{\del z} = \frac{1}{2\sqrt{z}\sqrt{z-1}}, \qquad z\in\C_+. \end{equation}
The extension of $L$ to a multivalued function satisfies
\begin{equation} \label{eq L(R)} \textstyle{L[1,\infty) = [\frac12\ln\frac12,\infty), \quad L(-\infty,0] =[\frac12\ln\frac12,\infty)+i\frac{\pi}{2}, \quad L(0,1) = \frac12\ln\frac12+i(0,\frac{\pi}{2})}. \end{equation}
In particular, $L$ has an extension to a holomorphic map on a complex neighborhood of $(0,1)$. \end{lemma}

The first task is to prove the fixed-point Equation \ref{eq Ht ft}.

\begin{lemma} \label{lemma fixed point} Let $H_t$ be the solution in the upper half-plane to PDE \ref{eq PDE H} with initial condition $H_0$ analytic on $\C_+$.  Define $f_t$ in terms of $H_t$ by Equation \ref{def ft}.  Then $H_t(z) = H_0(f_t(z))$ for $z\in\C_+$.
\end{lemma}

\begin{proof} PDE \ref{eq PDE H} is well set-up for applying the method of characteristics, cf.\ \cite{Evans}.  Writing it in the form
\[ \frac{\del H}{\del t}(t,z) = 2\sqrt{z}\sqrt{z-1}H(t,z)\frac{\del H}{\del z}(t,z) \]
we see that it is a homogeneous semilinear equation of the form
\begin{equation} \label{eq vector field} \mx{b}(t,z,H(t,z))\cdot \nabla H(t,z) = 0 \end{equation}
where $\mx{b}(t,z,w) = \left[ 1,  -2\sqrt{z}\sqrt{z-1}w\right]$.  Fix $(t_0,z_0)\in\R_+\times\C_+$, and define $\mx{z}(t) = (t,z(t))$ by the ODE
\begin{equation} \label{eq characteristic 0} \frac{d\mx{z}}{dt} = \mx{b}\left[\mx{z}(t),H(t,z(t))\right] \end{equation}
passing through the point $\mx{z}(t_0) = (t_0,z_0)$.  Then taking $w(t) = H(\mx{z}(t)) = H(t,z(t))$ and applying the chain rule, Equation \ref{eq vector field} shows that the {\em characteristic curve} $(\mx{z}(t),w(t))$ is contained in the graph $\{(t,z,w)\in\R_+\times\C_+\times\C\colon w=H(t,z)\}$; moreover, the set of all such curves (for all choices $(t_0,z_0)$) traces out the entire graph.  (This follows from the a priori knowledge that $H(t,z)$ is analytic in both variables.)  It is customary to write Equation \ref{eq characteristic 0} as a system for the separate variables $z,w$.  Because of the homogeneity in Equation \ref{eq vector field}, the characteristic equations in our case are
\begin{align} \frac{dz}{dt} &= -2\sqrt{z}\sqrt{z-1}w, \label{eq characteristic 1} \\ \frac{dw}{dt} &= 0. \label{eq characteristic 2} \end{align}
Thus, $w(t) = H(t,z(t))$ is constant; so, in particular, we have
\begin{equation} \label{eq H z0} H(0,z(0)) = w(0) = w(t_0) = H(t_0,z(t_0)) = H(t_0,z_0). \end{equation}
We can determine the position $z(0)$ of the characteristic at time $0$ (given its position $z(t_0) = z_0$) from Equation \ref{eq characteristic 1}.  Since $w=w(t_0)$ is constant, this can be explicitly solved in the upper half-plane, where the function $\frac{1}{2\sqrt{z}\sqrt{z-1}}$ has antiderivative $L$ (cf.\ Equation \ref{eq L'}).  The solution is
\[ \textstyle{\frac{dz}{2\sqrt{z}\sqrt{z-1}}} = -w(t_0)\,dt \quad \implies L(z(t)) = -w(t_0)t + C \]
for a constant $C$, which is determined by the constraint $z(t_0)=z_0$; that is, $C = L(z_0)+w(t_0)t_0$.  Thus, at time $0$ the curve's value $z(0)$ is determined by
\begin{equation} \label{eq L(f) 1} L(z(0)) = -w(t_0)\cdot0 + L(z_0)+w(t_0)t_0 = L(z_0) + t_0H(t_0,z_0) \end{equation}
from Equation \ref{eq H z0}.  It follows that $L(z_0)+t_0H(t_0,z_0) \in S$, and so $z(0) = M[L(z_0)+t_0H(t_0,z_0)] = f_{t_0}(z_0)$ by definition Equation \ref{def ft}.  Thus, Equation \ref{eq H z0} asserts that
\[ H(0,f_{t_0}(z_0)) = H(t_0,z_0). \]
Since $(t_0,z_0)\in\R_+\times\C_+$ was arbitrary, this concludes the proof.
\end{proof}

\begin{remark} \label{remark L(f) 1} The function $f_t = M\circ(L+tH_t)$ need not, a priori, satisfy $L\circ f_t = L + tH_t$.  Nevertheless, Equation \ref{eq L(f) 1} actually clarifies that this further restriction does hold; in other words, since that equation holds true for all $(t_0,z_0)\in\R_+\times\C_+$, we have $f_t(z) \in L(\C_+)=S$ for each $z\in\C_+$, and
\begin{equation} \label{eq L(f) 2} L(f_t(z)) = L(z) + tH_t(z) = L(z) + tH_0(f_t(z)), \quad \forall\, z\in\C_+. \end{equation}
The second equality follows from the statement of Lemma \ref{lemma fixed point}.  It is also worth noting here that the domain of definition for the function $H_t$ in Lemma \ref{lemma fixed point} is $\C_+$; as such, the characteristic $z(t)$ passing through $z_0$ at time $t_0$ is necessarily contained in $\C_+$, and so $f_{t_0}(z_0) = z(0)\in\C_+$.  In other words, it follows from the proof that
\begin{equation} \label{eq ft(C_+) in C_+} f_t(\C_+) \subseteq \C_+, \qquad t>0. \end{equation}
This can also be seen from the definition of $f_t$, Equation \ref{def ft}, at least for small $t>0$, using the Taylor expansion of $M$ about $L(z)$ together with a small-angle argument in a neighborhood of $\R$ in $\C_+$; the details are left to the interested reader.
\end{remark}

\begin{remark} It is possible to derive the relation $H(t,z) = H(0,f_t(z))$ in the following alternative fashion: define $J(t,z) = H_0(f_t(z))$.  Note that $J(0,z) = H_0(f_0(z)) = H_0(z)$.  Elementary but laborious calculations show that $J$ also satisfies PDE \ref{eq PDE H}, and is analytic. It therefore follows from the Cauchy-Kowalewski theorem that $J = H$, as claimed.  We prefer the method of characteristics approach above for two reasons: the calculations are much shorter, and they avoid some technical difficulties that arise differentiating $J(t,z)$ at points $z$ where $f_t(z)$ happens to be in $[0,1]$, outside the known analytic domain of $H_0$.  These difficulties can be overcome with restriction and then analytic continuation techniques once continuity is proven; we prefer this more direct approach, which aids in the proof of continuity.  It is worth noting that points $z$ where $f_t(z)\in[0,1]$ will play an important role in what follows.
\end{remark}

\begin{remark} One might hope to use the method of characteristics to similarly deduce the existence of a subordinator for the solution of the general case, for $(\alpha,\beta)\ne(\frac12,\frac12)$, i.e. $(a,b)\ne(0,0)$.  The same change of coordinates $H = \sqrt{z}\sqrt{z-1}G$ transforms the general PDE \ref{eq PDE} into the {\em in}homogeneous semilinear equation
\[ \frac{\del H}{\del t}-\left[2\sqrt{z}\sqrt{z-1}H -(az+b)\right]\frac{\del H}{\del z} + \frac{(a+2b)z-b}{2z(z-1)}H = 0. \]
The corresponding characteristic equations are
\begin{align*}
\frac{dz}{dt} &= -2\sqrt{z}\sqrt{z-1}w-(az+b), \\ \frac{dw}{dt} &= \frac{(a+2b)z-b}{2z(z-1)}w.
\end{align*}
The inhomogeneity generates a fully intertwined system of ODEs that is not explicitly solvable except in the special case $a=b=0$; this is one demonstration of how the behavior in the case we presently consider is better than the general case.
\end{remark}

We next use the fixed-point equation to show that the solution $H_t$ is bounded for $t>0$, uniformly in $t$ away from $0$.

\begin{lemma} \label{lemma bounded} For any compact subset $K\subset\C$, there is a constant $C_K$ so that
\[ |H(t,z)| \le \max\{C_K/t,1\}, \quad z\in K\cap\C_+. \]
In particular, it follows that, for any $\d>0$,
\[ \sup_{t\ge\d} \sup_{z\in\C_+} |H(t,z)| <\infty. \]
\end{lemma}

\begin{proof} Consider the function $H_0\circ M$.  For fixed $\e>0$, let $B_\e = B_{1/2+\e}(1/2)$ be the open ball of radius $1/2+\e$ centered at $1/2$ (i.e.\ a complex open ball containing $[0,1]$).  Since $H_0(z)=\sqrt{z}\sqrt{z-1}G_0(z)$ and $G_0(z)=G_{\nu_0}(z)$ is the Cauchy transform of a measure supported in $[0,1]$, $H_0$ is analytic on $\C-[0,1]$ and hence is bounded on the closed set $\overline{\C-B_\e}$.  Thus, $H_0\circ M$ is bounded on $L(\overline{\C-B_\e})$.  Since the function $L$ is bounded on compact sets, there is an $\e$-dependent constant $R$ so that
\begin{equation} \label{eq containment} \overline{\C-B_R(0)} \subseteq L(\overline{\C-B_\e}). \end{equation}

\noindent Now, let $K$ be any compact subset of $\C$.  Since $\sup_K|L|<\infty$, we can find a constant $C_K$ such that
\begin{equation} \label{eq main bound} \forall w\in\C \quad |w|+ \sup_K |L| \ge C_K \;\; \implies \;\; |H_0\circ M(w)| \le 1. \end{equation}
To see why, note that we may first choose $C_K$ larger than $R-\sup_K |L|$; then the set of $w$ in question is contained in the set $\overline{\C-B_\e}$ where $H_0\circ M$ is bounded.  Since we also know that the limit as $|z|\to\infty$ of $H_0(z)$ is $1/2$, and that $|M(w)|\to\infty$ as $|w|\to\infty$, the estimate follows from the continuity of $H_0\circ M$ on $\overline{\C-B_R(0)}$.

\noindent Now, for any compact $K\subset \C$, suppose there is a $z\in K-[0,1]$ such that $|H_t(z)|\ge C_K/t$. Set $w = L(z)+ tH_t(z)$.  Then
\[ |w| = |tH_t(z)+L(z)| \ge t|H_t(z)|-|L(z)| \ge C_K - |L(z)|. \]
Therefore
\[ |w|+\sup_K |L(z)| \ge |w|+ |L(z)| \ge C_K. \]
From Equation \ref{eq main bound}, it follows that $|H_0\circ M(w)|\le 1$.  But the fixed-point Equation \ref{eq Ht ft} for $H_t$ then says that, for $z\in K\cap\C_+$,
\[ |H_t(z)| = |(H_0\circ M)[L(z)+tH_t(z)]| = |H_0\circ M(w)| \le 1. \]
Thus, for any $z\in K\cap \C_+$, if $|H_t(z)|\ge C_K/t$, then $|H_t(z)|\le 1$.  This proves the first statement of the lemma.  For the second, note (as used above) that $\lim_{|z|\to\infty}zG(t,z) = \nu_t[0,1] = \frac12$ for any $t\ge 0$, and so $\lim_{|z|\to\infty}H_t(z) = \frac12$.   Moreover, the analyticity in $t$ and the convergence to the steady state (Section \ref{section steady state}) shows that this convergence is uniform in $t\ge\d$.  Thus, there is a compact set $K$ and a fixed constant $C$ so that $\sup_{t\ge\d}\sup_{z\notin K} |H(t,z)|\le C$.  Combining this with the first statement of the lemma proves the second.
\end{proof}
\noindent With Lemma \ref{lemma bounded} in hand, we now make the following assumption, without loss of generality.

\bigskip

\noindent {\bf Assumption 1.} $H_0$ is bounded.  Thus $\sup_{t\ge 0}\sup_{z\in\C_+}|H(t,z)|<\infty$.

\bigskip

\noindent This assumption is justified by the semigroup property for the solution $H_t$ of PDE \ref{eq PDE H}: given $t_0>0$, the solution $H^{t_0}_t$ of the PDE with initial condition $H^{t_0}_0 = H_{t_0}$ is equal to $H^{t_0}_t = H_{t+t_0}$.  So, in each of the following statements that hold for all $t>0$, we may simply begin the proof by selecting some $t_0\in(0,t)$ and proving the theorem instead for $t-t_0>0$, then use the semigroup property; in so doing, we translate to initial condition $H_{t_0}$ which satisfies Assumption 1 by Lemma \ref{lemma bounded}.  So we may make this assumption without loss of generality, and freely do so for the remainder of this section.

\begin{remark} It is important to note that, while $H_t$ satisfies the semigroup property, its subordinator $f_t$ does not.  Indeed, if we denote $f^{t_0}_t$ as the subordinator corresponding to the solution with initial condition $H_{t_0}$, then Equation \ref{def ft} yields
\[ f_t^{t_0}(z) = M[L(z) + tH_t^{t_0}(z)] \ne M[L(z) + (t+t_0)H_t(z)] =  f_{t_0+t}(z). \]
\end{remark}

We now proceed to demonstrate that the subordinator $f_t$ is a homeomorphism.  First, we identify its range.

\begin{definition} \label{def Omega h} For $t>0$, define the region $\Omega_t\in\C_+$ as follows:
\begin{equation} \label{eq Omega} \Omega_t = \{w\in\C_+\colon L(w)-tH_0(w)\in L(\C_+)=S\}. \end{equation}  \end{definition}

\begin{lemma} \label{lemma ft bijection} The map $f_t$ is a conformal (one-to-one) bijection from $\C_+$ onto $\Omega_t$. \end{lemma}

\begin{proof} Let $z_1,z_2\in\C_+$.  If $f_t(z_1)=f_t(z_2)$, then by Equation \ref{eq L(f) 2}, it follows that
\[ L(z_1) + tH_0(f_t(z_1)) = L(f_t(z_1)) = L(f_t(z_2)) = L(z_2) + tH_0(f_t(z_2)). \]
But the assumption $f_t(z_1) = f_t(z_2)$ then implies that $L(z_1) + tH_0(f_t(z_1)) = L(z_2) + tH_0(f_t(z_1))$ and so $L(z_1)=L(z_2)$; so $z_1=z_2$ by Lemma \ref{lemma L properties}, and thus $f_t$ is one-to-one on $\C_+$.  By Equation \ref{eq L(f) 2}, for any $z\in\C_+$, $L(f_t(z)) = L(z)+tH_0(f_t(z))$ which means that $(L-tH_0)(f_t(z)) = L(z)\in L(\C_+)$. This, coupled with Equation \ref{eq ft(C_+) in C_+}, shows that $f_t(z)\in\Omega_t$ for all $z\in\C_+$; and so $f_t(\C_+)\subseteq\Omega_t$.  Conversely, if $w\in\Omega_t$, then there exists $z\in\C_+$ such that $L(w)-tH_0(w) = L(z)$.  Equation \ref{eq L(f) 2} shows that $w=f_t(z)$ is a solution to this fixed-point equation, and moreover the injectivity of $L$ proves that it is the unique solution; hence $w=f_t(z)$, so $w\in f_t(\C_+)$; and so $\Omega_t\subseteq f_t(\C_+)$.
\end{proof}

We will next show that $f_t$ extends continuously to a homeomorphism $\overline{\C_+}\to\overline{\Omega_t}$.  Were we to follow the approach in \cite{Biane free heat}, here we would use the inverse function theorem applied to the putative inverse of $f_t$.  Indeed, Equation \ref{eq L(f) 2} states that $(L-tH_0)\circ f_t = L$, meaning that the inverse $h_t$ of $f_t$, should it exist, must be $h_t = M\circ (L-tH_0)$.  A strictly-positive lower-bound on the Lipschitz constant of $h_t$ would imply $h_t$ extends to a continuous one-to-one map from $\overline{\Omega_t}\to\overline{\C_+}$, yielding the corresponding result for $f_t$.  Unfortunately, this approach is not possible in the present context, as the following example illustrates.
\begin{example} \label{example singularity} Suppose $\mu_0$ is the Bernoulli measure $\mu_0 = \frac12(\delta_0+\delta_1)$.  Hence $\nu_0 = \frac12\delta_1$, and so $G_0(z) = 1/2(z-1)$, and $H_0(z) = \sqrt{z}/2\sqrt{z-1}$.  Simple calculation shows that $\frac{\del}{\del w}[L(w)-tH_0(w)] =0$ at the point $w = w_t =1-t/2$, which is in the unit interval for $0\le t\le 2$; so $h_t'(w_t) = 0$.  In fact, this point is also in $\del\Omega_t$ during this time interval.  Restricting to $t\in(0,1)$ and letting $s=1-t$, we have
\[ \textstyle{ L(w_t) = \frac12\ln\left[\frac12\left(1-t+i\sqrt{(2-t)t}\right)\right] = \frac12\ln\frac12+\frac12\ln\left[s+i\sqrt{1-s^2}\right] = \frac12\ln\frac12+\frac12i\tan^{-1}\left(\frac{\sqrt{1-s^2}}{s}\right). }\]
Also
\[ \textstyle{tH_0(w_t) = t\frac{\sqrt{1-t/2}}{\sqrt{-t/2}} = -i\sqrt{1-s^2}.} \]
Hence
\[ \textstyle{L(w_t)-tH_0(w_t) = \frac12\ln\frac12 + i\cdot\frac12\left[\sqrt{1-s^2}+\tan^{-1}\left(\frac{\sqrt{1-s^2}}{s}\right)\right].} \]
Elementary calculus shows that, for $s\in(0,1)$, this is in $\frac12\ln\frac12+i(0,\pi/2)\subset \overline{L(\C_+)}$.  Thus $w_t\in\overline{\Omega_t}$, and $h_t'(w_t)=0$.  So there is no strictly positive lower-bound on $\|h_t\|_{\mathrm{Lip}(\Omega_t)}$.
\end{example}

\noindent Nevertheless, $f_t$ does indeed possess a continuous, one-to-one extension to the boundary.  They key issue is identifying the nature of $\del\Omega_t$.

\begin{lemma} \label{lemma Jordan curve} The boundary $\del\Omega_t$ is a Jordan curve in $\overline{\C_+}$.
\end{lemma}

\begin{proof} To begin, let us note that it suffices to prove the claim for all sufficiently small $t>0$, by the semigroup property of the solution $H_t$.  Consider the fibration of $\C_+$ provided by the vertical line-segments $y\mapsto x+iy$ for $y>0$ and fixed $x\in\R$.  We will prove the following claim.

\noindent {\bf Claim.} For fixed $x\in\R$, if $x+iy_0\in\Omega_t$ then $x+iy\in\Omega_t$ for any $y>y_0$.

\noindent This suffices to prove the lemma: it demonstrates that $\Omega_t$ is the region above the graph of real-valued function of $x\in\R$.  Since $\Omega_t$ is open, this function is automatically upper semi-continuous, and therefore does not have any oscillatory discontinuities; hence, by appending vertical line segments to any jump discontinuities, we see that $\Omega_t$ is bounded by a Jordan curve.

To prove the claim, we use the convexity of the strip $S=L(\C_+)$.  Indeed, fix two $\frac12$-probability measures $\nu_0,\nu_1$, and let $\Omega^0_t$ and $\Omega^1_t$ be the regions corresponding to these measures.    Fix $s\in(0,1)$, and set $\nu_s = (1-s)\nu_0+s\nu_1$, their convex combination, with corresponding region $\Omega^s_t$.   Let $H^s_0(z) = \sqrt{z}{\sqrt{z-1}}G_{\nu_s}(z)$; then
\begin{align*} L(z)-t H^s_0(z) &= L(z) - t\sqrt{z}{\sqrt{z-1}}\int_0^1 \frac{1}{z-u}\,[(1-s)\nu_0(du)+s\nu_1(du)] \\
&= L(z)-t[(1-s)H^0_0(z) + sH^1_0(z)] \\
&= (1-s)[L(z)-tH_0^0(z)] + s[L(z)-tH^1_0(z)]. \end{align*}
Now, suppose $z\in\Omega^j_t$ for $j=0,1$.  This means (cf.\ \ref{eq Omega}) that $L(z)-tH_0^j(z)\in L(\C_+)=S$.  Since $S$ is convex, any convex combination of the two points $L(z)-tH_0^j(z)$, $i=0,1$, is also in $S$; thus, we have shown that $L(z)-tH^s_0(z)\in L(\C_+)$, and so $z\in\Omega^s_t$ for $0<s<1$.    It follows immediately that, if the Claim holds for $\Omega_t^0$ and for $\Omega_t^1$, then it holds for $\Omega_t^s$ for $0<s<1$.  Since every $\frac12$-probability measure is a weak limit of convex combinations of $\frac12$-point masses in $[0,1]$, it therefore suffices to prove the claim only for the special case that the initial measure of the form $\nu_0 = \frac12\delta_a$ for some point $a\in[0,1]$.

The Cauchy transform of $\frac12\delta_a$ is $z\mapsto \frac{1}{2(z-a)}$; hence, we must study the function
\[ F_{t,a}(z) = L(z) - t\frac{\sqrt{z}\sqrt{z-1}}{2(z-a)} = \frac12\ln\left(z-\frac12+\sqrt{z}\sqrt{z-1}\right)-t\frac{\sqrt{z}\sqrt{z-1}}{2(z-a)} \]
To prove the claim, it suffices to show that, for each $x\in\R$, the image of the line segment $y\mapsto x+iy$ under $F_{t,a}$ intersects the boundary of $S$ at most once for $y>0$.  The following facts may be verified by elementary calculus.
\begin{itemize}
\item[(1)] For $x<a$, $\Im F_{t,a}(x+iy)>0$ for all $y>0$.
\item[(2)]  For $x\ge a$, $\frac{\del}{\del y}\Im F_{t,a}(x+iy)>0$ for all $y>0$.
\item[(3)] For $x\in[0,1]$, $\frac{\del}{\del y}\Re F_{t,a}(x+iy)$ possesses at most one $0$, and is $>0$ for large $y>0$. 
\item[(4)] For $x\notin[0,1]$, $\frac{\del}{\del y}\Re F_{t,a}(x+iy)>0$ for all $y>0$.
\end{itemize}
Item (1) shows that the image of $\Im F_{t,a}$ never intersects the lower boundary $y=0$ when $x<a$, and item (2) shows that it intersects the lower boundary at most once when $x\ge a$.  In both cases, since $t\sqrt{z}\sqrt{z-1}/2(z-a)\to t/2$ as $|z|\to\infty$, its imaginary part tends to $0$; thus for sufficiently small $t$ (independent of $a$), $\Im F_{t,a}$ never intersects the upper boundary $y=\frac{\pi}{2}$.  As for $\Re F_{t,a}$, when $x\in[0,1]$ note that $\Re F_t(x+i0) = \Re L(x+i0) = \frac12\ln\frac12$ constantly, and so item (3) shows that the image curve $y\mapsto \Re F_{t,a}(x+iy)$ may initially dip below this level and return to intersect the line $y=\frac12\ln\frac12$ once, or it may stay above this line for all $y>0$; in either case, it intersects the line at most once.  Similarly, for $x\notin[0,1]$, $\Re F_{t,a}(x+i0)>\frac12\ln\frac12$ for sufficiently small $t>0$ (independent of $a$), and so item (4) shows that in this case $\Re F_{t,a}(x+iy)>\frac12\ln\frac12$ for $y>0$.

We have thus shown that, for all sufficiently small $t>0$, independent of $a\in[0,1]$, for any $x\in\R$, if $F_{t,a}(x+iy_0)\in L(\C_+)$ then $F_{t,a}(x+iy)\in L(\C_+)$ for all $y>y_0$.  This proves the claim in the case of initial measure $\frac12\delta_a$, and thence by the above convexity argument, proves the lemma.  \end{proof}

This actually suffices to prove the main result: that $f_t$ extends continuously (and, in fact, injectively) to the boundary.  This illustrates the general principle that pathological boundary behavior of conformal maps is observable in the topology of the image of the boundary.

\begin{theorem} \label{thm homeomorphism} For $t>0$, the subordinator $f_t\colon\C_+\to\Omega_t$ extends to a homeomorphism $\overline{\C_+}\to\overline{\Omega_t}$.  \end{theorem}

\begin{proof} Since $f_t$ is a conformal map defined on all of $\C_+$, and by Lemma \ref{lemma Jordan curve} the boundary of $f_t(\C_+)$ is a Jordan curve, the result follows immediately from Carath\'eodory's Theorem, cf.\ \cite[Thm 2.6]{Pommerenke}.
\end{proof}

\begin{remark} Example \ref{example singularity} demonstrates that the continuous extension of the conformal map $f_t$ to $\R$ need not be smooth: it can certainly have singularities along the line (at the critical values of $f_t^{-1}=h_t$).  We will quantify exactly where such singularities may occur in Lemma \ref{lemma ft smooth} below.
\end{remark}

\begin{corollary} \label{cor Ht continuous} For $t>0$, the solution $H_t$ to PDE \ref{eq PDE H} possesses a continuous extension to $(0,1)$. \end{corollary}

\begin{proof} Referring to Equation \ref{eq L(f) 2}, we have $L(f_t(z)) = L(z)+tH_t(z)$ for $z\in\C_+$.  Thus
\begin{equation} \label{eq Ht in terms of ft} H_t(z) = \frac{1}{t}\left[L(f_t(z))-L(z)\right]. \end{equation}
Theorem \ref{thm homeomorphism} shows that $f_t$ has a continuous extension to $\overline{\C_+}$, whose range is $\overline{\Omega_t}$, contained in $\overline{\C_+}$.  By Lemma \ref{lemma L properties}, $L$ possesses an analytic continuation to a complex neighborhood of $(0,1)$, concluding the proof.
\end{proof}

\begin{remark} \label{remark Ht continuous} From here on, we will use $H_t$ and $f_t$ to refer to the extended continuous maps defined on $(0,1)$. The continuity of all involved functions and the fact that Equation \ref{eq Ht in terms of ft} holds on $\C_+$ shows that it also holds for the extensions to $(0,1)$. \end{remark}

\begin{corollary} \label{cor continuous density} For each $t>0$, the measure $\nu_t$ possesses a density $\rho_t$, which is continuous on $(0,1)$, and satisfies the bound
\[ \rho_t(x) \le \frac{C}{\sqrt{x(1-x)}}, \quad x\in(0,1) \]
for some constant $C>0$ independent of $t$.
\end{corollary}

\begin{proof} By Corollary \ref{cor Ht continuous}, the map $H_t$ has a continuous extension to $(0,1)$ (which we also refer to as $H_t$).  Thus
\[ G_t(z) = \frac{1}{\sqrt{z}\sqrt{z-1}}H_t(z) \]
possesses a continuous extension to $(0,1)$.  The Stieltjes inversion Formula \ref{eq Stieltjes} in pointwise form then shows that, for $x\in(0,1)$,
\begin{equation} \label{eq rho} \rho_t(x) = -\frac{1}{\pi}\Im G_t(x) = \frac{1}{\pi}\frac{1}{\sqrt{x(1-x)}}\Re H_t(x) \end{equation}
is the density of $\nu_t$, which is therefore continuous.  Finally, Assumption 1 (i.e.\ Lemma \ref{lemma bounded}, together with the semigroup property) shows that $\|H_t\|_{L^\infty(\R)}<\infty$, and so the stated result holds true with $C = \|H_t\|_{L^\infty(\R)}/\pi$.
\end{proof}

\noindent This proves most of our main Theorem \ref{thm smoothing}.  Let us also note one more immediate Corollary of Theorem \ref{thm homeomorphism}, together with Lemma \ref{lemma fixed point}, that will be useful in Section \ref{section unification conjecture}.

\begin{corollary} \label{cor positive density} Suppose that $\nu_0$ has a strictly positive density $\rho_0$.  Then the density $\rho_t$ of $\nu_t$ is strictly positive for any $t>0$. \end{corollary}

\begin{proof} From Equation \ref{eq rho}, $\rho_t(x)$ is positive at any $x\in(0,1)$ for which $\Re H_t(x)>0$.  Equation \ref{eq Ht ft} asserts that $H_t(x) = H_0(f_t(x))$.  The harmonic function $\Re H_0$ is equal to $\pi\sqrt{x(1-x)}\rho_0(x)\ge 0$ on $\R$ and $\lim_{|z|\to\infty}H_0(z)=\frac12$ so it is strictly positive for large $z\in\C_+$; by the minimum principle, $\Re H_0>0$ in $\C_+$.  Thus, since $f_t(x)$ is either in $[0,1]$ or in $\C_+$, the assumption that $\Re H_0(y)>0$ for $y\in[0,1]$ implies that $H_0(f_t(x))>0$ for all $x\in(0,1)$.  Hence $\rho_t(x)>0$.  
\end{proof}

\noindent Having Corollaries \ref{cor Ht continuous} and \ref{cor continuous density} in hand, we may apply the semigroup property as in Assumption 1, and so we freely make the following assumption from here forward, without loss of generality.

\bigskip

\noindent {\bf Assumption 2.} $H_0$ has a continuous, bounded extension to $(0,1)$, and the initial measure $\nu_0$ possesses a density $\rho_0$ which is continuous on $(0,1)$, for which $\sup_{x\in\R}\sqrt{x(1-x)}\rho_0(x)<\infty$.

\bigskip

It remains to prove the smoothness claim of Theorem \ref{thm smoothing} for the measure $\nu_t$.  First, we need the following results on the behavior of the continuous extension of the subordinator $f_t$ on the boundary set $[0,1]$.

\begin{lemma} \label{lemma transform support 0} Let $t>0$ and $x\in[0,1]$.  If $f_t(x)\in[0,1]$, then $\rho_t(x)=0$, and $\rho_0(f_t(x))=0$.  \end{lemma}

\begin{proof} If $f_t(x)\in[0,1]$, then $L(f_t(x))-L(x)$ is purely imaginary, since $\Re L\equiv \frac12\ln\frac12$ is constant on $[0,1]$, cf.\ Lemma \ref{lemma L properties}.  Hence, by Equation \ref{eq Ht in terms of ft},
\[ \Re H_t(x) = \frac1t\Re [L(f_t(x))- L(x)] = 0. \]
It follows immediately from Equation \ref{eq rho} that $\rho_t(x)=0$, as claimed.  Also, from Equation \ref{eq Ht ft} (extended to the boundary by continuity) $H_0(f_t(x)) = H_t(x)$, so the assumption that $f_t(x)\in[0,1]$ shows again by Equation \ref{eq rho} that $\rho_0(f_t(x))=0$ in this case.
\end{proof}

\begin{lemma} \label{lemma ft smooth} Let $t>0$ and let $x\in(0,1)$.  If $\rho_t(x)>0$, then $f_t$ is analytic in a neighborhood of $x$, and $f_t(x)\in\C-[0,1]$. \end{lemma}

\begin{proof} First, from Lemma \ref{lemma transform support 0}, since $\rho_t(x)\ne 0$ it follows that $f_t(x)\notin[0,1]$, as claimed.  Now, from Theorem \ref{thm homeomorphism}, $f_t\colon\overline{\C_+}\to \overline{\Omega_t}$ is a homeomorphism with inverse $h_t = M\circ(L-tH_0)$ on $\overline{\Omega_t}$.  In particular, $L-tH_0$ is one-to-one on this domain, and so $(L-tH_0)^{\langle -1 \rangle}$ is well-defined on $L(\overline{\C_+})$.  What's more, since $L$ has an analytic continuation to a neighborhood of $(0,1)$, and $H_0$ is analytic on $\C-[0,1]$, the fact that $L-tH_0$ is one-to-one on all of $L(\overline{\C_+})$ implies that $(L-tH_0)^{\langle -1\rangle}$ is analytic on the complement of $(L-tH_0)([0,1])$.  For $x\in[0,1]$, if $L(x)\in (L-tH_0)([0,1])$ then there is some $y\in[0,1]$ so that $L(x) = L(y)-tH_0(y)$, which is precisely to say that $x= M(L(y)-tH_0(y)) = h_t(y)$, and so $y=f_t(x)$.  Since $y\in[0,1]$, this implies by the first statement of the lemma that $\rho_t(x)=0$.  Hence, if $\rho_t(x)>0$ then $(L-tH_0)^{-1}$ is analytic in a neighborhood of $L(x)$, and so $f_t$ is analytic in a neighborhood of $x$.
\end{proof}

This completes all the elements needed for the proof of Theorem \ref{thm smoothing}.

\begin{proof}[Proof of Theorem \ref{thm smoothing}] By Corollary \ref{cor continuous density}, the $\nu_t$ possesses a continuous density $\rho_t$ with the appropriate behavior at the boundary points $\{0,1\}$. Lemma \ref{lemma fixed point} asserts that $H_t(z) = H_0(f_t(z))$ for $z\in\C_+$; the continuity of $H_t$ and $f_t$ on $(0,1)$ means that this holds for $z=x\in(0,1)$ as well, and so $H_t(x) = H_0(f_t(x))$ where $f_t(x)\in\C_+$.  Let $x\in[0,1]$ be a point where $\rho_t(x)>0$.  Lemma \ref{lemma ft smooth} proves that $f_t$ is analytic in a neighborhood of $x$, and that $f_t(x)\notin[0,1]$.  Since $H_0$ is analytic on $\C-[0,1]$, the composition $H_t(x) = H_0\circ f_t(x)$ is therefore analytic in a neighborhood of $x$.
\end{proof}

Thus, $\nu_t$ has continuous density which is analytic except at the boundary of its support.  We conclude this section with a corollary regarding the nature of the zero set of $\rho_t$.

\begin{lemma} \label{lemma transform support} For $t>0$, let $Z_t = \{x\in\R\colon \rho_t(x)=0\}$ be the $0$-set of the measure $\nu_t$.  Then $Z_t = f_t^{-1}(Z_0)$.\end{lemma}
\noindent Note: we are explicitly under Assumption 2 here.  The semigroup property has been applied and time has been shifted some small positive amount so that $\nu_0$ possesses a continuous density $\rho_0$.  The result does not say that the topology of $\supp\nu_t$ in any way resembles the topology of the support of the original (unliberated) measure $\mu_{qpq}$, which can be any closed set in $[0,1]$.

\begin{proof} Suppose $x\in Z_t$, so $\rho_t(x)=0$.  Equation \ref{eq rho} shows that $\Re H_t(x) = 0$, and Equation \ref{eq Ht ft} then yields $\Re H_0(f_t(x))=0$.  Now, $\Re H_0(y) = \pi\sqrt{y(1-y)}\rho_t(y)\ge 0$ for $y\in\R$, and $\lim_{|z|\to\infty}H_0(z) = \frac12>0$, so $\Re H_0>0$ for $|z|$ sufficiently large in $\C_+$.  $H_0$ is harmonic in $\C_+$, so it follows from the minimum principle that $\Re H_0(z)>0$ for $z\in\C_+$.  Thus, $\Re H_0(y)=0$ implies that $y\in\R$.  Hence, $\Re H_0(f_t(x))=0$ implies that $f_t(x)=0$, and so by definition $f_t(x)\in Z_0$.  Thus $Z_t\subseteq f_t^{-1}(Z_0)$.

Conversely, if $f_t(x)\in Z_0$, then $y=f_t(x)\in\R$ and
\[ 0 = \rho_0(y) = \frac{1}{\pi\sqrt{y(1-x)}}\Re H_0(y) \implies 0=\Re H_0(y)  = \Re H_0(f_t(x)) = \Re H_t(x) \]
which gives $\rho_t(x)=0$ by Equation \ref{eq rho} once more.  Hence  $x\in Z_t$, so $f_t^{-1}(Z_0) \subseteq Z_t$.
\end{proof}

\begin{remark} It is tempting to conclude from Corollary \ref{lemma transform support} that $\supp\nu_t = f_t^{-1}(\supp\nu_0)$.  This is not generally the case, since $f_t$ maps the interior of the support of $\nu_t$ into the upper half-plane $\C_+$.  The subordinator $f_t$ is continuous and one-to-one, but it is not a map from $[0,1]\to[0,1]$.  In particular, it does not follow that $\supp\nu_t$ is homeomorphic to $\supp\nu_0$ -- it is perfectly possible for components of the support to merge in finite time. \end{remark}

\section{The Unification Conjecture for Projections} \label{section unification conjecture}

This section is devoted to the Unification Conjecture for free entropy and information of projections.  We briefly describe the setting of free entropy and information, and then formulate and prove a special case of the conjecture in the present context.  The reader is directed to the excellent introduction \cite{Voiculescu BLMS}, and the papers \cite{Voiculescu-free-analogs-VI} and \cite{Hiai Ueda 1,Hiai Ueda 2}, for a detailed treatment of the background.

\subsection{Free Entropy, Fisher Information, and Mutual Information} \label{sect free entropy}

Let $x_1,\ldots,x_n$ be self-adjoint operators in a $\mathrm{II}_1$-factor $(\mathscr{A},\t)$.  For parameters $N,m\in\N$ and $R,\e>0$, the set of {\bf matricial microstates} of these operators, denoted $\Gamma_R(x_1,\ldots,x_n;N,m,\e)$, is the set of self-adjoint $N\times N$ matrices $X_1,\ldots,X_n$, with norm $\le R$, all of whose mixed (normalized) trace moments of order $\le m$ are within tolerance $\e>0$ of the corresponding mixed $\tau$ moments of $x_1,\ldots,x_n$.
\begin{align*} \Gamma_R(&x_1,\ldots,x_n;N,m,\e) \\
& \hspace{0.1in} = \Big\{X_1,\ldots,X_n\in M_N^{sa}(\C)\colon |\tr\!_N(X_{i_1}\cdots X_{i_r})-\t(x_{i_1}\cdots x_{i_r})|<\e \; \forall r\in[m], (i_1,\ldots,i_r)\in[n]^r\Big\}\end{align*}
where $[n]=\{1,\ldots,n\}$.  The volume of this set (in the metric  given by the trace norms) grows or decays exponentially in the square of the dimension $N$.  The (microstates) {\bf free entropy} $\chi(x_1,\ldots,x_n)$ is defined to be
\begin{equation} \label{eq chi} \chi(x_1,\ldots,x_n) = \sup_{R>0}\inf_{m\in\N,\e>0} \limsup_{N\to\infty}\frac{1}{N^2}\log\mathrm{Vol}\left(\Gamma_R(x_1,\ldots,x_n;N,m,\e)\right). \end{equation}
In the $n=1$ case of a single self-adjoint operator $x$, Voiculescu calculated that the free entropy is equal to the {\em logarithmic energy} of the spectral measure $\mu_x$ (up to an additive constant): if $\mu_x$ has a density $\mu_x(du) = \rho_x(u)\,du$, then
\begin{equation} \label{eq Sigma} \chi(x) -\frac34-\frac12\ln2\pi = \int \ln|u-v|\rho_x(u)\rho_x(v)\,dudv \equiv \Sigma(\mu_x). \end{equation}
(If $\mu_x$ is singular, $\chi(x)\equiv-\infty$.)  For several non-commuting operators $x_1,\ldots,x_n$, there is no analytic formula for the free entropy; although we will see below that, changing the definition as appropriate, the case of two {\em projections} does afford a closed-form analysis, cf.\ Equation \ref{eq chi proj}.

Let $A,B\subset\mathscr{A}$ be algebraically free unital $\ast$-subalgebras; $A\vee B$ denotes the unital $\ast$-subalgebra generated by $A\cup B$.  There is a unique derivation $\delta_{A:B}\colon A\vee B\to (A\vee B)\tensor(A\vee B)$ determined by
\[ \begin{cases} \delta_{A:B}(a) = a\tensor 1-1\tensor a & a\in A \\ \delta_{A:B}(b) = 0 & b\in B \end{cases}. \]
(Uniqueness is guaranteed by the algebraic freeness of $A,B$; without this assumption, $\delta_{A:B}$ is not well-defined.  It is important to note that algebraic freeness is not related to free independence in general.)  The {\bf liberation gradient} of $A$ relative to $B$, $j(A:B)$, is defined (if it exists) through integration by parts: $j(A:B)\in L^1(W^\ast(A\vee B),\tau)$ satisfies
\begin{equation} \label{eq j(A:B)} \tau\left[j(A:B)x\right]  = \tau\tensor\tau\left(\delta_{A:B}(x)\right), \quad x\in A\vee B. \end{equation}
The (liberation) {\bf free Fisher information} of $A$ relative to $B$, $\varphi^\ast(A:B)$, is the square-$L^2$-norm of the liberation gradient:
\begin{equation} \label{eq Fisher} \varphi^\ast(A:B) \equiv \|j(A:B)\|_2^2 = \tau[j(A:B)^\ast j(A:B)]. \end{equation}
(If $j(A:B)$ does not exist, or exists in $L^1$ but is not in $L^2$, $\varphi^\ast(A:B)\equiv\infty$.)  This definition precisely mirrors the conjugate variables approach to classical Fisher information, cf.\ \cite{Voiculescu BLMS}.  As with free entropy, free Fisher information can rarely be computed explicitly.  One important exception is in the case of two projections: if $\mathscr{A}=W^\ast(p,q)$, then $\varphi^\ast(W^\ast(p):W^\ast(q))$ can be computed directly, cf.\ Equation \ref{eq phi* for projections} below.

The {\bf mutual free information} of subalgebras $A$ and $B$, $i^\ast(A:B)$, is defined in terms of the mutual free Fisher information, via the liberation process:
\begin{equation} \label{eq i*} i^\ast(A:B) = \frac12\int_0^\infty \varphi^\ast(u_tAu_t^\ast:B)\,dt \end{equation}
where $u_t$ is a free unitary Brownian motion, freely independent from $A\vee B$.  This definition is arrived at in \cite[Section 4]{Voiculescu-free-analogs-VI} from the classical relation between Information and Entropy.  Indeed, if $S(X_1,\ldots,X_n)$ denotes the Shannon entropy of random variables $X_1,\ldots,X_n$, the mutual information $I(X:Y)$ of a pair is defined to be $I(X:Y)= -S(X,Y)+S(X)+S(Y)$.  Starting here, and using a microstates-free {\em infinitesimal} version of free entropy $\chi^\ast$, Voiculescu gave a convincing heuristic (modulo continuity issues at $t=0,\infty$) that if $i^\ast(x:y)\equiv -\chi^\ast(x,y)+\chi^\ast(x)+\chi^\ast(y)$ then Equation \ref{eq i*} should hold for $A=W^\ast(x)$ and $B=W^\ast(y)$.

At present, it is unknown whether $\chi=\chi^\ast$ in general, and it is similarly unknown whether the heuristic used in \cite{Voiculescu-free-analogs-VI} to give the definition of $i^\ast$ can be made rigorous.  This question (along with the claim that $\chi=\chi^\ast$) is known as the Unification Conjecture.

\begin{conjecture}[Unification Conjecture] \label{unification conjecture} Let $x_1,\ldots,x_n$ and $y_1,\ldots,y_n$ be self-adjoint operators in a $\mathrm{II}_1$-factor $(\mathscr{A},\tau)$, such that $\chi(x_1,\ldots,x_n)$, $\chi(y_1,\ldots,y_n)$, and $\chi(x_1,\ldots,x_n,y_1,\ldots,y_n)$ are all finite.  Then
\begin{equation*} \label{eq unification conjecture}
i^\ast(W^\ast(x_1,\ldots,x_n):W^\ast(y_1,\ldots,y_n)) = -\chi(x_1,\ldots,x_n,y_1,\ldots,y_n)+\chi(x_1,\ldots,x_n)+\chi(y_1,\ldots,y_n).
\end{equation*}
\end{conjecture}

\begin{remark} One benefit of free mutual information $i^\ast$, defined in Equation \ref{eq i*}, is that it is finite in many situations where free entropy is not; for this reason, the finiteness assumption on $\chi$ is needed for the statement of Conjecture \ref{unification conjecture} to be plausible. \end{remark}

As noted in \cite{Voiculescu BLMS} and still true today, the result of Conjecture \ref{unification conjecture} is a long way from where the theory is at present.

\bigskip

\subsection{Free Entropy for Projections} \label{section free entropy for projections} In the special case $n=2$ and generators $p,q$ that are projections, the abstract quantities of Section \ref{sect free entropy} can be described more directly.  Following \cite{Hiai Ueda 2} and \cite[Section 12]{Voiculescu-free-analogs-VI}, denote
\[ E_{11} = p\wedge q, \quad E_{10} = p\wedge q^\perp, \quad E_{01} = p^\perp\wedge q, \quad E_{00} = p^\perp\wedge q^\perp \]
where $p^\perp = 1-p$ and $q^\perp = 1-q$.  Define $E = 1-(E_{00}+E_{01}+E_{10}+E_{11})$, and let $\alpha_{ij}=\tau(E_{ij})$ for $i,j\in\{0,1\}$.  Then $E_{ij}$ (and hence $E$) are in the center of $W^\ast(p,q)$, and so the compression $\left(EW^\ast(p,q)E,\left.\tau\right|_{EW^\ast(p,q)E}\right)$ is isomorphic to the $2\times 2$-matrix-valued $L^\infty$ space of a measure $\nu$: the representations $EpE\leftrightarrow M_p, EqE\leftrightarrow M_q$ given by
\[ M_p(x) = \left[\begin{array}[c]{cc} 1 & 0 \\ 0 & 0 \end{array}\right]  \qquad M_q(x) = \textstyle{\left[\begin{array}[c]{cc} x & \sqrt{x(1-x)} \\ \sqrt{x(1-x)} & 1-x \end{array}\right]} \]
identify the compression as $W^\ast(M_p,M_q)\cong L^\infty((0,1),\nu;M_2(\C))$ for a uniquely-defined positive measure on $(0,1)$ with mass $\nu((0,1))=1-(\alpha_{00}+\alpha_{01}+\alpha_{10}+\alpha_{11})$. The trace restricted to $EW^\ast(p,q)E$ is then given by
\[ \tau(a) = \int_0^1 \tr[M_a(x)]\,\nu(dx). \]
In this context, the measure $\nu$ encodes all the structure of the algebra; we will soon see that $\nu$ is indeed related to the spectral measure of the operator valued projection $qpq$, our central object of study.  In \cite[Prop 12.7]{Voiculescu-free-analogs-VI}, it was shown that the (liberation) free Fisher information of $W^\ast(p)$ relative to $W^\ast(q)$ can be explicitly computed in terms of the measure $\nu$.

\begin{proposition}[\cite{Voiculescu-free-analogs-VI}] \label{prop phi for proj} Suppose $\alpha_{00}\alpha_{11}=\alpha_{10}\alpha_{01}=0$.  Suppose that $\nu(dx) = \rho(x)\,dx$ has a density $\rho\in L^3((0,1),x(1-x)\,dx)$.  Assume that
\begin{equation} \label{eq alpha assumption} \int_0^1 \left(\frac{\alpha_{01}+\alpha_{10}}{x} + \frac{\alpha_{00}+\alpha_{11}}{1-x}\right)\rho(x)\,dx <\infty. \end{equation}
Let $\phi$ denote the Hilbert transform of $\rho$, modified by the $\alpha_{ij}$ as follows:
\[ \phi(x) = H\rho(x) + \frac{\alpha_{01}+\alpha_{10}}{x} + \frac{\alpha_{00}+\alpha_{11}}{1-x}. \]
Then the (liberation) free Fisher information of $W^\ast(p)$ relative to $W^\ast(q)$ is given by
\begin{equation} \label{eq phi* for projections} \varphi^\ast(W^\ast(p):W^\ast(q)) = \int_0^1 \phi(x)^2\rho(x)x(1-x)\,dx \end{equation}
and is finite.
\end{proposition}

\begin{remark} \label{remark trace 1/2 phi*}  If $\tau(p)=\alpha$ and $\tau(q)=\beta$ as in the foregoing, then the assumption $\alpha_{00}\alpha_{11} = \alpha_{10}\alpha_{01} = 0$ is equivalent to the identifications $\alpha_{11}=\max\{\alpha+\beta-1,0\}$, $\alpha_{00}=\max\{1-\alpha-\beta,0\}$, $\alpha_{10}=\max\{\alpha-\beta,0\}$, and $\alpha_{01}=\max\{\beta-\alpha,0\}$.  In particular, in our special case $\alpha=\beta=\frac12$, this assumption is tantamount to $\alpha_{ij}=0$, meaning that $E_{ij}=0$ for $i,j=\{0,1\}$; thus, the assumption is that {\bf $p$ and $q$ are in general position}.  If they are not, we cannot expect a simple relationship between the Fisher information and the density.  Note that, in this case, Equation \ref{eq alpha assumption} holds trivially.
\end{remark}

With Proposition \ref{prop phi for proj} in hand, we can hope to explicitly verify the Unification Conjecture \ref{unification conjecture} for the case of two projections.  There is a twist, however.  Going back to the definition of free entropy in Equation \ref{eq chi}, dimension considerations show that, unless $p=q=1$, $\chi(p)=\chi(q)=\chi(p,q)=-\infty$; hence, these operators do not fit the statement of the Unification Conjecture per se.  Instead, a modified version of free entropy for projections, $\chi_{\mathrm{proj}}$, is required.

Let $\mathcal{G}(N,k)$ denote the Grassmannian manifold of rank-$k$ projections on $\C^N$.  Any projection in $\mathcal{G}(N,k)$ is conjugate to the projection $\mathrm{diag}(1,1,\ldots,1,0,0,\ldots,0)$ with $k$ $1$s; the conjugating unitary $U\in U(N)$ is only determined up to the block structure of this diagonal projection, and so $U$ is invariant under the action of $U(k)\times U(N-k)$.  This gives an identification of $\mathcal{G}(N,k) \cong U(N)/(U(k)\times U(N-k))$ as a symmetric space.  Let $\pi_{N,k}\colon U(N)\to \mathcal{G}(N,k)$ be the quotient map, and define
\[ \gamma_{\mathcal{G}(N,k)} = \mathrm{Haar}_{U(N)}\circ \pi_{N,k}^{-1}. \]
Thus $\gamma_{\mathcal{G}(N,k)}$ is the unique unitarily invariant probability measure on the Grassmannian of appropriate dimension/rank.  They key to defining microstates free entropy for projections is to use this measure in place of Euclidean volume: if the rank of the limit projections is not full, then they cannot be properly approximated by microstates of full rank.

Fix projections $p_1,\ldots,p_n$.  For $1\le i\le n$, let $(k_i(N))_{N=1}^\infty$ be sequences of positive integers with the property that $k_i(N)/N \to \tau(p_i)$ as $N\to\infty$.  Define $\Gamma_{\mathrm{proj}}(p_1,\ldots,p_n;k_1(N),\ldots,k_n(N);N,m,\e)$, the {\bf projection microstates}, to be the set of projection matrices $P_1,\ldots,P_n$ with $P_i\in\mathcal{G}(N,k_i(N))$, all of whose mixed (normalized) trace moments of order $\le m$ are within tolerance $\e>0$ of the corresponding mixed $\tau$ moments of $p_1,\ldots,p_n$.  That is, $\Gamma_{\mathrm{proj}}$ is
\[ \Big\{(P_1,\ldots,P_n)\in \prod_{i=1}^n \mathcal{G}(N,k_i(N))\colon| \tr\!_N(P_{i_1}\cdots P_{i_r})-\t(p_{i_1}\cdots p_{i_r})|<\e \; \forall r\in[m], (i_1,\ldots,i_r)\in[n]^r\Big\}. \]
Following Voiculescu's remarks in \cite[Sect 14]{Voiculescu-free-analogs-VI}, in \cite{Hiai Ueda 1,Hiai Ueda 2} Hiai and Petz defined the {\bf projection free entropy} as follows:
\begin{equation} \label{eq proj chi} \begin{aligned} \chi_{\mathrm{proj}}&(p_1,\ldots,p_n) \\
 & = \inf_{m\in\N,\e>0}\limsup_{N\to\infty} \frac{1}{N^2}\log \bigotimes_{i=1}^n \gamma_{\mathcal{G}(N,k_i(N))}\left(\Gamma_{\mathrm{proj}}(p_1,\ldots,p_n;k_1(N),\ldots,k_n(N);N,m,\e)\right). \end{aligned}
\end{equation}
The value does not depend on the specific sequences $k_i(N)$.  Note that, when $n=1$, all moments $\tr_N(P^m) = \tr_N(P) = k_i(N)$ and $\t(p^m) = \t(p)$, so the microstate space in this case is the full Grassmannian, which means that $\chi_{\mathrm{proj}}(p)=0$ for a single projection.  For two projections, $\chi_{\mathrm{proj}}$ can be explicitly calculated using large deviations results for projections, cf.\ \cite[Thm 3.2, Prop 3.3]{Hiai Petz}.  The result is as follows.

\begin{proposition}[\cite{Hiai Petz}] \label{prop chi proj} Suppose $\alpha_{00}\alpha_{11} = \alpha_{10}\alpha_{01} = 0$.  There is an explicit constant $C(\alpha,\beta)$ (where $\alpha=\t(p)$ and $\beta=\t(q)$ as usual) such that
\begin{equation} \label{eq chi proj} \chi_{\mathrm{proj}}(p,q) = \frac14\Sigma(\nu) + \frac{\alpha_{10}+\alpha_{01}}{2}\int_0^1 \log x\;\nu(dx) + \frac{\alpha_{11}+\alpha_{00}}{2}\int_0^1\log(1-x)\,\nu(dx)-C(\alpha,\beta). \end{equation}
\end{proposition}

\begin{remark} \label{remark trace 1/2 chi proj} Following Remark \ref{remark trace 1/2 phi*}, in the special case $\alpha=\beta=\frac12$ the assumption $\alpha_{00}\alpha_{11} = \alpha_{10}\alpha_{01} = 0$ is equivalent to the assumption that $p,q$ are in general position: $\tau(p\wedge q)=0$.  In this case $\alpha_{ij}=0$ for $i,j\in\{0,1\}$; moreover, the complicated constant $C(\alpha,\beta)$ satisfies $C(\frac12,\frac12)=0$ (cf.\ \cite[Equation 1.2]{Hiai Ueda 2}).  Thus, $\chi_{\mathrm{proj}}(p,q) = \frac14\Sigma(\nu)$ in our case.
\end{remark}

Thus, a natural extension of Conjecture \ref{unification conjecture} is to ask that it hold for projections, where the free entropy terms on right-hand-side are replaced with {\em projection} free entropy terms, $\chi\to\chi_{\mathrm{proj}}$.  We state this formally in the case $n=1$ as follows.

\begin{conjecture} \label{proj unification conjecture} Let $p,q$ be projections in a $\mathrm{II}_1$-factor.  Then
\begin{equation} \label{eq proj unification conjecture} i^\ast(W^\ast(p):W^\ast(q)) = -\chi_{\mathrm{proj}}(p,q)+ \chi_{\mathrm{proj}}(p)+\chi_{\mathrm{proj}}(q)=-\chi_{\mathrm{proj}}(p,q). \end{equation}
\end{conjecture}

The remainder of this paper is devoted to the proof of Conjecture \ref{proj unification conjecture}, in the special case $\t(p)=\t(q)=\frac12$; i.e.\ Theorem \ref{thm unification conjecture}.

\bigskip

\subsection{The Proof of Theorem \ref{thm unification conjecture}} The theorem is stated on page \pageref{thm unification conjecture}; it asserts that Conjecture \ref{proj unification conjecture} holds true in the special case $\alpha=\beta=\frac12$.  Our method is to employ Propositions \ref{prop phi for proj} and \ref{prop chi proj} to precisely calculate the two sides of Equation \ref{eq proj unification conjecture}.  Note, as explained in Remarks \ref{remark trace 1/2 phi*} and \ref{remark trace 1/2 chi proj}, these Propositions require the assumption that $p,q$ are in general position; otherwise, both sides of Equation \ref{eq proj unification conjecture} are $-\infty$.

The idea for the present proof is essentially due to Hiai and Ueda, and is very briefly outlined in \cite[Sect 3.2]{Hiai Ueda 2}.  Using the explicit formulas of Equations \ref{eq phi* for projections} and \ref{eq chi proj}, direct calculation demonstrates that $\frac{d}{dt}\chi_{\mathrm{proj}}(p_t,q) = \frac12\varphi^\ast(p_t:q)$.  (This is in line with Voiculescu's original heuristics in defining free Fisher information: a version of this equality -- differentiating entropy along a Brownian perturbation yields Fisher information -- does hold for classical entropy and Fisher information.)  Integrating this using Equation \ref{eq i*} defining mutual free information, the fundamental theorem of calculus shows that $i^\ast(W^\ast(p):W^\ast(q)) = -\lim_{t\downarrow 0}\chi_{\mathrm{proj}}(p_t,q)$; passing the limit inside $\chi_{\mathrm{proj}}$ is then justified using the work of \cite{HMU,Hiai Ueda 2}.  The technical difficulties arise in the derivative calculations; it is here that our main smoothing Theorem \ref{thm smoothing} will play a central role in the analysis.

To begin, we identify the measure $\nu$ in Propositions \ref{prop phi for proj} and \ref{prop chi proj} with the (continuous part of the) operator-valued angle studied in Sections \ref{sect flow of mu t}--\ref{section local properties}.

\begin{lemma} \label{lemma equal measures} Let $p,q$ be projections with trace $\t(p)=\t(q)=\frac12$.  Let $(p_t,q)_{t>0}$ be their free liberation.  Denote by $\nu^t$ the measure characterizing $W^\ast(p_t,q)$, cf.\ the discussion at the beginning of Section \ref{section free entropy for projections}.  Then $\nu^t=2\nu_t$, the measure of Theorem \ref{thm PDE} and \ref{thm smoothing}.
\end{lemma}

\begin{proof}  Let $\nu_t$ be as in Theorems \ref{thm PDE} and \ref{thm smoothing}.  By Corollary \ref{cor positive density} and Theorem \ref{thm smoothing}, the density $\rho_t$ of $\nu_t$ is continuous and piecewise real analytic on $(0,1)$.  Let $u_t,v_t$ be the real and imaginary parts of the Cauchy transform of $\nu_t$: $G(t,z)=G_{\nu_t}(x+iy)=u_t(x,y)+iv_t(x,y)$.  Equation \ref{eq Stieltjes Hilbert} identifies the density $\rho_t$ as $v_t(x,0)=-\pi\rho_t(x)$, and its (scaled) Hilbert transform $\phi_t=-\pi H\rho_t$ is $ \phi_t(x) = Hv_t(x,0) = -u_t(x,0)$.  (We scale the Hilbert transform by $-\pi$ here to match up to the notation in \cite{Hiai Ueda 2}.)  In each interval on which $\rho_t>0$, $\rho_t$ is analytic, and hence $\left.G_t\right|_{\overline{\C_+}}$ has an analytic continuation to a neighborhood of this interval.  Hence, PDE \ref{eq PDE 0} extends to $(0,1)$, and we have
\[ \frac{\del}{\del t}v_t(x,0) =  \frac{\del}{\del t}\Im G_t(x) = \frac{\del}{\del x}\left[x(x-1)\Im(G_t(x)^2)\right] = 2\frac{\del}{\del x}\left[x(x-1)u_t(x,0)v_t(x,0)\right]. \]
Thus
\[ -\frac{\del}{\del t}\pi \rho_t(x) = 2\frac{\del}{\del x}[x(x-1)(-\phi_t(x)\pi\rho_t(x))]. \]
Dividing by $\pi$ and and writing this in terms of the scaled functions $2\rho_t$ and $2\phi_t$ yields
\begin{equation} \label{eq PDE rho phi 1} \frac{\del}{\del t}(2\rho_t(x)) = \frac{\del}{\del x}[x(x-1)2\phi_t(x)\cdot2\rho_t(x)]. \end{equation}
This is precisely the differential equation for the density of the measure $\nu^t$, Equation 3.1 in \cite{Hiai Ueda 2}, in the case $\alpha=\beta=\frac12$ (so $\alpha_{ij}=0$ for $i,j\in\{0,1\}$) under the assumption that measure has a smooth (enough) density.  Thus, Theorem \ref{thm smoothing} and the Stieltjes inversion formula \ref{eq Stieltjes Hilbert} show that $2\rho_t$ is indeed the density of the measure $\nu^t$; in other words, $\nu^t = 2\nu_t$.
\end{proof} 

\begin{remark} This technique may be carried out for general $\alpha,\beta$ to show that the measure $\nu$ is proportional to the continuous part of the spectral measure of $qpq$, provided smoothness may be proved first.  It should be possible to see directly that the two measures are (up to a proportionality constant) equal, since they are both defined naturally in terms of the operator-valued angle $qp_tq$ and the intersection $p_t\wedge q$; however, such a direct identification is not obvious to the authors.  \end{remark}

Henceforth, write $\rho^t = 2\rho_t$ and $\phi^t = 2\phi^t$.  Thus, Equation \ref{eq PDE rho phi 1}, which is now proved rigorously as a consequence of Theorems \ref{thm PDE} and \ref{thm smoothing} and Corollary \ref{cor positive density}, takes the form
\begin{equation} \label{eq PDE rho phi 2} \frac{\del}{\del t}\rho^t(x) = \frac{\del}{\del x}[x(x-1)\phi^t(x)\rho^t(x)]. \end{equation}

\noindent We now wish to use Propositions \ref{prop phi for proj} and \ref{prop chi proj} to calculate $i^\ast(p_t,q)$ and $\chi_{\mathrm{proj}}(p_t,q)$; to do so, we must verify the integrability condition of Proposition \ref{prop phi for proj}.

\begin{lemma} \label{lemma L3} For $t\ge 0$, $\rho^t\in L^3(x(1-x)\,dx)$. \end{lemma}

\begin{proof} This follows immediately from the bound of Equation \ref{eq measure bound}:
\begin{align*} \|\rho_t\|_{L^3(x(1-x)\,dx)}^3 &= \int_0^1 \rho_t(x)^3x(1-x)\,dx \\
&\le \int_0^1 \frac{C(t)^3}{(x(1-x))^{3/2}}x(1-x)\,dx = C(t)^3\int_0^1 \frac{dx}{\sqrt{x(1-x)}} = C(t)^3\pi<\infty. \end{align*}
\end{proof}

We now prove the main comparison that will lead to the proof of Theorem \ref{thm unification conjecture}.

\begin{lemma} \label{lemma chi proj derivative} For $t>0$,
\begin{equation} \label{eq chi proj derivative} \frac{d}{dt}\chi_{\mathrm{proj}}(p_t,q) = \frac12\varphi^\ast(W^\ast(p_t):W^\ast(q)). \end{equation}
\end{lemma}

\noindent In this proof, we make the simplifying assumption that $\rho_t$ is real analytic on $(0,1)$.  The same proof justifies the statement in the more general case, that $\rho_t$ is continuous and real analytic on the set where $\rho_t>0$ (hence H\"older continuous),  with additional notation that obscures the idea of the proof.

\begin{proof} By Proposition \ref{prop gp 2}, $p_t$ and $q$ are in general position for $t>0$.  In particular, $\alpha_{ij}=0$ for $i,j\in\{0,1\}$.  Hence, Proposition \ref{prop chi proj} and Remark \ref{remark trace 1/2 chi proj} say that
\begin{equation} \label{eq chi deriv 1} \chi_{\mathrm{proj}}(p_t,q) = \frac14\Sigma(\nu^t) = \frac14\int \ln|x-y|\rho^t(x)\rho^t(y)\,dxdy. \end{equation}
Now, Equation \ref{eq PDE rho phi 2} implies that
\begin{align*} \frac{\del}{\del t}(\rho^t(x)\rho^t(y)) &= \rho^t(x)\frac{\del}{\del t}\rho^t(y) + \rho_t(y)\frac{\del}{\del t}\rho^t(x) \\
&= \rho^t(x)\frac{\del}{\del y}[y(y-1)\phi^t(y)\rho^t(y)] + \rho^t(y)\frac{\del}{\del x}[x(x-1)\phi^t(x)\rho^t(x)]. 
\end{align*}
Integrating against the logarithmic energy kernel, by symmetry we have
\begin{align*}
\frac14\int\ln|x-y|\frac{\del}{\del t}[\rho^t(x)\rho^t(y)]\,dxdy &= \frac12\int \rho^t(y)\ln|x-y|\frac{\del}{\del x}[x(x-1)\phi^t(x)\rho^t(x)]\,dxdy.
\end{align*}
Let $k^t(x) = x(x-1)\phi^t(x)\rho^t(x)$.  The Hilbert transform interchanges harmonic conjugates; hence, the analyticity of $\rho^t$ implies the analyticity of $\phi^t$, and so $k^t$ is analytic on $(0,1)$.  The dependence of the integrand on $t$ is thus also analytic (cf.\ Section \ref{section G analytic t}), and this easily justifies differentiating under the integral sign.  We now integrate by parts, first writing the integral as a principal value:
\begin{align*}
&\frac{d}{dt}\chi_{\mathrm{proj}}(p_t,q) \\
=&\frac12\int\rho^t(y)\ln|x-y|\frac{\del}{\del x}k^t(x)\,dxdy \\
=& \frac12\int_0^1\rho^t(y)\,dy\lim_{\e\downarrow 0}\left(\int_0^{y-\e}\ln|x-y|\frac{\del}{\del x}k^t(x)\,dx + \int_{y+\e}^1\ln|x-y|\frac{\del}{\del x} k^t(x)\,dx\right) \\
=& \frac12\int_0^1\rho^t(y)\,dy\lim_{\e\downarrow 0}\left(\ln\e\cdot  k^t(y-\e)-\int_0^{y-\e}\frac{k^t(x)}{x-y}\,dx - \ln\e\cdot k^t(y+\e)-\int_{y+\e}^1 \frac{k_t(x)}{x-y}\,dx  \right)
\end{align*}
where the first equality follows from the continuity of $\del k^t(x)/\del x$.  The boundary terms from the integration by parts combine to
\[ \ln\e[k^t(y-\e)-k^t(y+\e)] = -\e\ln\e\frac{k^t(y+\e)-k^t(y-\e)}{\e}\to 0 \text{ uniformly on compact subset of }(0,1) \]
since $k^t$ analytic (so, in particular, $C^1$).  Hence, recombining,
\begin{align*} \frac{d}{dt}\chi_{\mathrm{proj}}(p_t,q) &=-\frac12\int_0^1dy\lim_{\e\downarrow0}\left(\int_0^{y-\e} \frac{k^t(x)}{x-y}\,dx + \int_{y+\e}^1 \frac{k^t(x)}{x-y}\,dx\right) \\
&= -\frac12\int_0^1\rho^t(y)\,dy \left(p.v.\int_0^1 \frac{k^t(x)}{x-y}\,dx\right).
\end{align*}
Since $\rho^t$ is supported in $[0,1]$, we can write this as
\[ \frac{d}{dt}\chi_{\mathrm{proj}}(p_t,q) = -\frac{\pi}{2}\int_\R \rho^t(y) H(k^t)(y)\,dy. \]
Write the kernel as $k^t(x)=\sqrt{x(1-x)}\phi^t(x) \cdot \sqrt{x(1-x)}\rho^t(x)$; up to factors of $\pi$, these are the real and imaginary parts of $H_t(x)$ (cf.\ Equation \ref{eq def H}) which, by Lemma \ref{lemma bounded}, is uniformly bounded.  Thus $k^t\in L^\infty$ and is supported in $[0,1]$; hence $k^t\in L^2(\R)$.  The Hilbert transform is self-adjoint on $L^2(\R)$, cf.\ \cite{Stein Weiss}, and so
\[ \frac{d}{dt}\chi_{\mathrm{proj}}(p_t,q) = -\frac{\pi}{2}\int_\R \rho^t(y) H(k^t)(y)\,dy = -\frac{\pi}{2}\int_\R H(\rho^t)(y)k^t(y)\,dy. \]
In the proof of Lemma \ref{lemma equal measures}, we saw that $\phi^t = -\pi H\rho^t$; thus, we have shown that
\[ \frac{d}{dt}\chi_{\mathrm{proj}}(p_t,q) = \frac12\int_\R \phi^t(y)k^t(y)\,dy = \frac12\int_0^1 \phi^t(y)x(1-x)\phi^t(y)\rho^t(y)\,dy. \]
By Lemma \ref{lemma L3}, the density $\rho^t$ is in $L^3(x(1-x)\,dx)$, and so Proposition \ref{prop phi for proj} applies.  Hence, Equation \ref{eq phi* for projections} concludes the proof:
\[ \frac{d}{dt}\chi_{\mathrm{proj}}(p_t,q) = \int_0^1 \phi^t(t)^2\rho_t(y)x(1-x)\,dx = \varphi^\ast(W^\ast(p_t):W^\ast(q)). \]
\end{proof}

\begin{corollary} \label{cor almost there} If $p,q$ are projections of trace $\frac12$, then
\begin{equation} i^\ast(W^\ast(p)\colon W^\ast(q)) = -\lim_{t\downarrow 0}\chi_{\mathrm{proj}}(p_t,q). \end{equation}
\end{corollary}

\begin{proof} Lemma \ref{lemma chi proj derivative} shows that $t\mapsto\chi_{\mathrm{proj}}(p_t,q)$ is differentiable for $t>0$, with derivative equal to 
$\frac12\varphi^\ast(W^\ast(p_t):W^\ast(q))$.  Equation \ref{eq Fisher} defining the (liberation) Fisher information $\varphi^\ast$ shows that it is manifestly $\ge 0$, and so the function $t\mapsto\chi_{\mathrm{proj}}(p_t,q)$ is non-decreasing.  Moreover, Theorem 2.1 in \cite{Hiai Ueda 2} (the main result of that paper) shows that $-\chi_{\mathrm{proj}}(p_t,q) \le \varphi^\ast(W^\ast(p_t):W^\ast(q))$ under our assumptions; whence
\begin{equation} \label{eq chi int} -\int_0^\infty \chi_{\mathrm{proj}}(p_t,q)\,dt \le \int_0^\infty \varphi^\ast(W^\ast(p_t):W^\ast(q))\,dt = 2i^\ast(W^\ast(p):W^\ast(q)) \end{equation}
by Equation \ref{eq i*} defining $i^\ast$.  Since $\varphi^\ast(W^\ast(p)\colon W^\ast(q))$ is finite (by Proposition \ref{prop phi for proj} and Lemma \ref{lemma L3}), \cite[Prop 10.11(c)]{Voiculescu-free-analogs-VI} guarantees that $i^\ast(W^\ast(p):W^\ast(q))$ is finite.  Hence, Inequality \ref{eq chi int} certainly guarantees that $\lim_{t\to\infty}\chi_{\mathrm{proj}}(p_t,q)=0$.  Ergo, by Lemma \ref{lemma chi proj derivative},
\begin{align*} i^\ast(W^\ast(p):W^\ast(q)) = \frac12\int_0^\infty \varphi^\ast(W^\ast(p_t):W^\ast(q))\,dt &= \int_0^\infty \frac{d}{dt}\chi_{\mathrm{proj}}(p_t,q)\,dt \\
&= \lim_{t\uparrow\infty}\chi_{\mathrm{proj}}(p_t,q) - \lim_{t\downarrow0}\chi_{\mathrm{proj}}(p_t,q) \\
&= 0 - \lim_{t\downarrow0}\chi_{\mathrm{proj}}(p_t,q).
\end{align*}
\end{proof}

Finally, this brings us to the proof of Theorem \ref{thm unification conjecture}.

\begin{proof}[Proof of Theorem \ref{thm unification conjecture}] Note that, as $t\downarrow 0$, $(p_t,q)$ converges in non-commutative distribution to $(p,q)$.  (Indeed, for any non-commutative polynomial $P$ in two variables, the function $t\mapsto \tau[P(p_t,q)]$ is $C^\infty[0,\infty)$.)  Hence, by \cite[Proposition 1.2(iii)]{Hiai Ueda 1},
\begin{equation} \label{e.final1} -\chi_{\mathrm{proj}}(p,q) \le \liminf_{t\downarrow 0} -\chi_{\mathrm{proj}}(p_t,q). \end{equation}
On the other hand, by \cite[Proposition 4.6]{HMU}, we have the reverse inequality: taking $v_1=u_t$ and $v_2=1$ (which are freely independent), it follows that $-\chi_{\mathrm{proj}}(p,q) \ge -\chi_{\mathrm{prom}}(v_1pv_1^\ast,v_2qv_2^\ast) = -\chi_{\mathrm{proj}}(p_t,q)$ for all $t\ge 0$.  Thus, we have
\begin{equation} \label{e.final2}  -\chi_{\mathrm{proj}}(p,q) \ge \limsup_{t\downarrow 0} -\chi_{\mathrm{proj}}(p_t,q). \end{equation}
Equations \ref{e.final1} and \ref{e.final2} show that $-\chi_{\mathrm{proj}}(p,q) = \lim_{t\downarrow 0}-\chi_{\mathrm{proj}}(p_t,q)$.  Combined with Corollary \ref{cor almost there}, this concludes the proof.
\end{proof}

\bigskip

\noindent {\em Acknowledgments.} The authors are grateful to Serban Belinschi and Yoann Dabrowski for insightful conversations, particularly regarding the results in Section \ref{section local properties}.  Thanks are also due to Michael Anshelevich, Nizar Demni, Peter Ebenfelt, and Jacob Sterbenz, for helpful feedback on a preliminary version of this paper.  Finally, we wish to thank the anonymous referee for thoroughly reading an earlier version of this manuscript and providing constructive feedback on the results.  In particular, thanks are due for making us aware of the results in \cite{HMU,Hiai Ueda 1} that strengthened the main result of this work, removing a technical assumption that had been needed in an earlier proof of Theorem \ref{thm unification conjecture}.

\end{document}